\documentclass[10pt]{article}

\usepackage[utf8]{inputenc}
\usepackage[T1]{fontenc}
\usepackage[english]{babel}

\usepackage{geometry}
\usepackage{xcolor}
\usepackage[
    colorlinks=true,
    linkcolor=blue,
    citecolor=blue,
    urlcolor=blue
]{hyperref}

\usepackage{amsmath, amssymb, amsthm}

\usepackage{lmodern}

\newtheorem{main}{Theorem}

\newtheorem{cor}[main]{Corollary}

\usepackage{graphicx}

\newtheorem{theorem}{Theorem}[section]
\newtheorem{lemma}[theorem]{Lemma}
\newtheorem{proposition}[theorem]{Proposition}
\newtheorem{corollary}[theorem]{Corollary}

\theoremstyle{definition}

\newtheorem{claim}[theorem]{Claim}

\theoremstyle{remark}
\newtheorem{remark}[theorem]{Remark}

\title{\textbf{Foliations and minimality for twist maps}}
\author{
Salvador Addas-Zanata\thanks{The author thanks FAPESP (project 2023/07076-4) for partial financial support.}
\and
Julian Lazaro Aguirre\thanks{The author thanks CAPES (project 88887.500833/2020-00) for financial support.}
}
\date{}

\begin{document}

\maketitle

\centerline { {\sl Instituto de Matem\'atica e Estat\'\i stica }}

\centerline {{\sl Universidade de S\~ao Paulo}}

\centerline {{\sl Rua do Mat\~ao 1010, Cidade Universit\'aria,}} 

\centerline {{\sl 05508-090 S\~ao Paulo, SP, Brazil}}

\begin{abstract}
In this paper we consider area-preserving twist maps $f:\mathbf{T}^2\to\mathbf{T}^2$ with  vertical rotation sets reduced to a single irrational number  $\alpha$. Our first result states that under these hypotheses, there exists a unique $f$-invariant foliation $\mathcal{F}_G=\{\operatorname{Graph}(\phi_t)\}_{t\in\mathbf{T}^1}$, where each function $\phi_t:\mathbf{T}^1\to\mathbf{T}^1$ is $K$-Lipschitz for some constant $K=K(f)$. 

Moreover, we prove that if $\mathcal{F}_G$ is a Lipschitz foliation, then $f$ is minimal. 

\medskip

\noindent\textbf{Keywords:} twist maps, vertical rotation set, foliations, minimality.

\end{abstract}

\tableofcontents

\section{Introduction}

This paper is motivated by the following general question:

\vskip0.1truecm

$\bullet$ How are certain properties of rotation sets reflected in the dynamics of 
homeomorphisms of closed oriented surfaces?  

\vskip0.2truecm

In the present case, we are not in the identity isotopy class as usual,
but we restrict ourselves to maps of the torus homotopic to Dehn twists, satisfying an additional 
dynamical property, the twist condition. This means that, if we consider a lift of 
the torus map to the plane, then the image of any vertical line projects injectively onto the horizontal
coordinate. 

Under these hypotheses, the rotation set is only one-dimensional (see \cite{alunofranks} and \cite{non04}), 
and similarly to the isotopic to the identity class, it is convex. So,
the only possibilities are: a point or a nondegenerate closed interval. We
will deal solely with the first possibility. Moreover, throughout the paper, this 
single number which captures the vertical rotational behaviour of all points in the 
torus will be assumed irrational.

The rational case was studied in \cite{eubra} and there we proved that when the vertical rotation
set of a torus homeomorphism $f$ homotopic to a Dehn twist is reduced to a single rational number $p/q$,  
then the dynamics is annular, meaning that if $\widehat{f}$ is the lift of $f$ to the vertical 
cylinder used to compute the vertical rotation set, then there exists an essential continuum in the cylinder, $K_{p/q}$,
which is invariant under $\widehat{f}^q-(0,p)$. Clearly, the same holds for all its vertical integer translates.
This is what we mean by annular dynamics: with respect to $\widehat{f}^q-(0,p)$, all orbits are uniformly bounded in the cylinder.  

In that same paper, we showed that in the Dehn twist homotopy class, minimality 
can only happen for vertical rotation sets reduced to a single irrational point. 

The search for minimality was the main motivation for this work, and an important guiding model was
the affine map given in flat
coordinates by

\begin{equation}\label{equaI01}
\rm{Aff}(x,y) := (x+y\bmod 1,\; y+\alpha\bmod 1),
\end{equation}
where $\alpha \notin \mathbb{Q}$.

It has long been known that $\rm{Aff}$ is minimal, preserves Lebesgue measure and is uniquely ergodic. 
Moreover, it preserves the trivial foliation $\mathcal{F}_C=\{\mathbf{T}^1\times\{t\}:t\in\mathbf{T}^1\}$.

For general Lebesgue measure-preserving twist diffeomorphisms of the torus whose rotation sets are reduced to a single irrational number, we show that their dynamics
have many similarities with the dynamics of $\rm{Aff}$, at least up to a certain point.

Combining this with previous results yields the following dichotomy:

\vskip0.1truecm

\textbf{Classification Result.} \textit{Given an area-preserving twist diffeomorphism $f$ of the torus, either it has periodic points, or $f$ admits an invariant foliation by Lipschitz graphs over the horizontal coordinate. Moreover, if the foliation 
has some regularity, then $f$ is minimal.}

\vskip0.1truecm

In the next subsection we present definitions, so that our main results can be precisely stated.

\subsection{Notation and definitions}

\begin{description}
    \item[(1)] Let $\mathbf{T}^1:=\mathbb{R}/\mathbb{Z}$ be the circle and $\mathbf{T}^2:=\mathbb{R}^2/\mathbb{Z}^2$ the flat torus. A point in $\mathbb{R}^2$ will be denoted as $\widetilde{z}=(\widetilde{x},\widetilde{y})$, in the cylinder $\mathbf{T}^1\times \mathbb{R}$ as $\widehat{z}=(\widehat{x},\widehat{y})$, and a point in $\mathbf{T}^2$ will be denoted as $z=(x,y)$. Let $p_1, p_2:\mathbb{R}^2\to \mathbb{R}$ denote the canonical projections, given by $p_1(\widetilde{x},\widetilde{y}):=\widetilde{x}$ and $p_2(\widetilde{x},\widetilde{y}):=\widetilde{y}$. The projections on the torus and on the cylinder will be denoted in the same way. Finally, let $p:\mathbb{R}^2\to \mathbf{T}^2$, $\pi:\mathbb{R}^2\to \mathbf{T}^1\times \mathbb{R}$, and $\tau:\mathbf{T}^1\times\mathbb{R} \to \mathbf{T}^2$ be the covering maps defined as $p(\widetilde{x},\widetilde{y}):=(\widetilde{x}\bmod 1, \widetilde{y}\bmod 1)$, $\pi(\widetilde{x},\widetilde{y}):=(\widetilde{x}\bmod 1, \widetilde{y}),$ and $\tau(\widehat{x},\widehat{y}):=(\widehat{x},\widehat{y}\bmod 1).$ 
    \item[(2)] 
For $\ell\in\mathbb{Z}\setminus\{0\}$, let $\mathrm{Diff}_{\mathrm{tw}}^1(\mathbf{T}^2)$ denote the set of $C^1$-diffeomorphisms of $\mathbf{T}^2$ that are homotopic to a \emph{Dehn twist}: $(x,y)\mapsto (x+\ell y\bmod 1, y\bmod 1)$, and satisfy a \emph{twist condition}: there exists a constant $k_{tw}>0$ such that, for any lift $\widetilde{f}:\mathbb{R}^2\to\mathbb{R}^2$ of $f$,
    \begin{equation}\label{equaI02}
       \frac {\ell}{|\ell|} \times\frac{\partial p_1 \circ \widetilde{f}(\widetilde{x},\widetilde{y})}{\partial\widetilde{y}} > k_{tw}\quad \text{for all } (\widetilde{x},\widetilde{y})\in\mathbb{R}^2.
    \end{equation}
    If $\ell>0$, we say that $f$ is a \emph{right twist map}, and when it is smaller than zero, it is a left one. The inverse of a right twist map is a left one, and vice versa. 

We also denote by $\mathrm{Diff}_{\mathrm{tw}}^1(\mathbb{R}^2)$ (respectively, $\mathrm{Diff}_{\mathrm{tw}}^1(\mathbf{T}^1\times\mathbb{R})$) the set of diffeomorphisms of the plane (respectively, of the cylinder) that are lifts of elements of $\mathrm{Diff}_{\mathrm{tw}}^1(\mathbf{T}^2)$. If an element of $\mathrm{Diff}_{\mathrm{tw}}^1(\mathbf{T}^2)$ is denoted $f$, any lift of $f$ to the vertical cylinder is denoted as $\widehat{f}$ and to the plane as $\widetilde{f}$. 

Here, we also present the weaker notion of topological twist:
Let $f:\mathbf{T}^{2}\to \mathbf{T}^{2}$ be a
homeomorphism homotopic to a Dehn twist $(x,y)\mapsto (x+\ell y\bmod 1, y\bmod 1)$, 
for some $\ell\in\mathbb{Z}\setminus\{0\}$, and let $\widetilde{f}:\mathbb{R}^{2}\to \mathbb{R}^{2}$ be a
lift of $f$. We say that $f$ satisfies the \emph{topological twist condition}
if the image of any vertical line in $\mathbb{R}^{2}$ under $\widetilde{f}$
has an injective horizontal projection. In this topological context, 
the concept of right and left twist can also be applied.

    \item[(3)] For any $\widehat{f}\in\mathrm{Diff}_{\mathrm{tw}}^1(\mathbf{T}^1\times \mathbb{R})$, we define the \emph{vertical rotation set} of $\widehat{f}$ as in \cite{alunofranks} or \cite{non04}, 
    \begin{equation}\label{equaI03}
        \rho_{V}(\widehat{f}):=\bigcap_{i=1}^{\infty}\overline{\bigcup_{n\geq i}\left\{\dfrac{p_2\circ \widehat{f}^n(\widehat{z})-p_2(\widehat{z})}{n} : \widehat{z}\in \mathbf{T}^1\times \mathbb{R}\right\}}\subset\mathbb{R}.
    \end{equation}
    It was proved in \cite{alunofranks} and \cite{non04} that $\rho_{V}(\widehat{f})\neq\emptyset$ is compact and connected, therefore it is either a singleton or a nondegenerate compact interval. We also define the \emph{vertical rotation number} of $z\in\mathbf{T}^2$ with respect to $\widehat{f}$ by
    \begin{equation}\label{equaI04}
        \rho_{V}(\widehat{f},z):=\lim_{n\to \infty}\dfrac{p_2\circ \widehat{f}^n(\widehat{z})-p_2(\widehat{z})}{n},\quad \text{ for any } \widehat{z}\in \tau^{-1}(z),
    \end{equation}
    whenever the limit exists.

When $\widetilde{f}\in\mathrm{Diff}_{\mathrm{tw}}^1(\mathbb{R}^2)$ is a lift to the plane of $\widehat{f}$, 
one can define the vertical rotation set of $\widetilde{f}$ analogously, and clearly $\rho_{V}(\widetilde{f})=\rho_{V}(\widehat{f}).$ 


    \item[(4)] A nonempty open set $U\subset \mathbf{T}^2$ is said to be \emph{inessential} if every closed curve contained in $U$ is homotopically trivial in $\mathbf{T}^2$; otherwise, $U$ is said to be \emph{essential}. We also say that $U$ is \emph{fully essential} if it contains two homotopically nontrivial simple closed curves representing distinct homotopy classes. A connected open set $U\subset \mathbf{T}^2$ is called \emph{annular} if it is neither inessential nor fully essential. Finally, a closed set $K\subset \mathbf{T}^2$ is inessential if it is contained in an inessential open set. 
    \item[(5)] Let $E\subset\mathbf{T}^2$ be an open or closed set. We define $Filled(E)$ as the union of $E$ with all inessential connected components of $\mathbf{T}^2\setminus E$.
    \item[(6)] Let $X$ be a topological space, and let $E\subset X$ be nonempty. We denote by $\partial_{X}E$ the boundary of $E$ relative to $X$, and by $\mathrm{Cl}_XE:=\mathrm{Int}_X(E)\cup \partial_XE$ the closure of $E$ in $X$, where $\mathrm{Int}_X(E)$ denotes the interior of $E$ in $X$. We also write $\mathrm{Cl}_XE=\overline{E}.$
    \item[(7)] We denote by $\mathcal{F}_G$ a foliation of $\mathbf{T}^2$ by simple closed curves, all homotopic to $(1,0)$, each of which is the graph of a $K$-Lipschitz function of the $x$-coordinate, for some constant $K=K(\mathcal{F}_G)>0$. We also denote by $\mathcal{F}_H:=\{\mathbf{T}^1\times\{t\} : t\in \mathbf{T}^1\}$ the horizontal foliation. 

Let $\mathcal{F}$ denote either $\mathcal{F}_G$ or $\mathcal{F}_H$. Given a homeomorphism $f:\mathbf{T}^2\to\mathbf{T}^2$, as usual, we say that $\mathcal{F}$ is $f$-\emph{invariant} if $f(L_z)=L_{f(z)}$ for every  $z\in \mathbf{T}^2$, where $L_z$ denotes the leaf of $\mathcal{F}$ containing $z$. The foliation $\mathcal{F}$ of $\mathbf{T}^2$ admits lifts to both the universal covering $\mathbb{R}^2$ and the cylinder $\mathbf{T}^1\times\mathbb{R}$. More precisely, there exist foliations $\widetilde{\mathcal{F}}$ on $\mathbb{R}^2$ and $\widehat{\mathcal{F}}$ on $\mathbf{T}^1\times\mathbb{R}$ such that $p(\widetilde{L}_{\widetilde{z}})=L_{p(\widetilde{z})}$ and $\tau(\widehat{L}_{\widehat{z}})=L_{\tau(\widehat{z})}$. Furthermore, if $\mathcal{F}$ is $f$-invariant, then any lift of $f$ to $\mathbb{R}^2$ (respectively to $\mathbf{T}^1\times\mathbb{R}$) leaves invariant the lifted foliation $\widetilde{\mathcal{F}}$ (respectively $\widehat{\mathcal{F}}$).
    \item[(8)] We say that such a foliation $\mathcal{F}_G$ is Lipschitz if there exists a constant $K_{Fol}>0$ such that for any two leaves of $\widehat{\mathcal{F}}_G$ given by the graphs of functions $\widehat{g}_1,\widehat{g}_2$, and for all $x,x'\in \mathbf{T}^1$, the following holds:
$$
\frac{1}{K_{Fol}}. \left| \widehat{g}_1(x')-\widehat{g}_2(x') \right| \leq \left| \widehat{g}_1(x)-\widehat{g}_2(x) \right| \leq K_{Fol}. \left| \widehat{g}_1(x')-\widehat{g}_2(x') \right|
$$   
\end{description}

\subsection{Main results}

Now we present our main results. The first theorem establishes that every area-preserving twist map of $\mathbf{T}^2$ whose vertical rotation set consists of a single irrational number admits an invariant foliation by Lipschitz graphs, as described above.

\begin{main}\label{thm:A}
    Let $f\in\mathrm{Diff}_{\mathrm{tw}}^1(\mathbf{T}^2)$ be area-preserving, and suppose that $\rho_V(\widehat{f})=\{\alpha\}$ for some $\alpha\notin\mathbb{Q}$, where $\widehat{f}\in\mathrm{Diff}_{\mathrm{tw}}^1(\mathbf{T}^1\times \mathbb{R})$ is a lift of $f$. Then $f$ admits a unique invariant foliation $\mathcal{F}_G$. 
\end{main}

\begin{remark} The conclusion of Theorem~\ref{thm:A} remains valid if, instead of Lebesgue measure, $f$ preserves a Borel probability measure $\mu$ with full support such that $\mu(\Gamma)=0$ for every rectifiable curve $\Gamma\subset\mathbf{T}^2$. This class contains all full-support probability measures absolutely continuous with respect to Lebesgue measure, including those defined by positive $C^1$ area forms, after normalization, and also allows singular measures.

The condition on curves guarantees, in particular, that vertical segments have zero measure, a fact used in the proof for the gates in Fathi's decomposition. Full support guarantees that every nonempty open set has positive measure. The measure comparisons in the proof then use the lift of $\mu$ in place of Lebesgue measure; in the uniqueness argument, the transverse coordinate is defined by the $\mu$-measure of the annuli between leaves. We do not know whether the conclusion holds under the non-wandering hypothesis alone.
\end{remark}

As a consequence of Theorem~\ref{thm:A}, every nonzero iterate of a map satisfying its hypotheses satisfies the topological twist condition.
\begin{cor}\label{cor:B}
    Under the hypotheses of Theorem~\ref{thm:A}, the iterates $f^n$ 
satisfy the topological twist condition for all $n\in\mathbb{Z}\setminus\{0\}$.
\end{cor}

The next theorem shows that, under the assumptions of Theorem~\ref{thm:A} and the additional condition stated below, the map $f$ is minimal.

\begin{main}\label{thm:C}
    Under the hypotheses of Theorem~\ref{thm:A}, assume also that the invariant foliation $\mathcal{F}_{G}:=\{\mathrm{Graph}(\phi_t)\}_{t\in\mathbf{T}^1}$ is Lipschitz. 
In this situation, $f$ is minimal.
\end{main}

Under the hypotheses of Theorem \ref{thm:A}, after a suitable change of coordinates by a torus homeomorphism $\Phi$ homotopic to the identity, the map $f$ takes the form
\begin{equation}\label{equaI05}
    f^*(x,y)=(x+\ell y +\varphi(x,y)\bmod 1, y+\alpha\bmod 1),\quad (x,y)\in\mathbf{T}^2,
\end{equation}
where $\alpha\notin\mathbb{Q}$ is the single rotation number in the vertical rotation set of $\widehat{f}$, and $\varphi$ is induced by a $\mathbb{Z}^2$-periodic function on $\mathbb{R}^2$. Such a map might not be differentiable anymore, but it satisfies the topological twist condition. 

The existence of such a coordinate change implies that the map induced by $f$ on the space of leaves of the foliation is always conjugate to the irrational rotation $R_{\alpha}$. An important ingredient in the proof of Theorem \ref{thm:C} is that the Lipschitz assumption on the foliation implies that the map $\Phi$ is bi-Lipschitz. 

In the coordinates of (\ref{equaI05}), $f$ preserves a measure of the form $\mu(A)=Leb(\Phi^{-1}(A))$ (called {\it homeomorphic to Lebesgue}). 

\vskip0.1truecm

{\it Proof of the Classification Result from the Introduction.} For any given area-preserving $\widehat{f}\in\mathrm{Diff}_{\mathrm{tw}}^1(\mathbf{T}^1\times \mathbb{R})$, if 
$\rho_V(\widehat{f})$ has interior, all rationals 
$p/q$ in its interior are realized by $f$-periodic orbits, see Theorem 6 of \cite{non04}. And when $\rho_V(\widehat{f})=\{p/q\},$ the map $\widehat{f}^q-(0,p)$ satisfies the curve intersection property, and so Theorem 3 of \cite{non04} gives the desired periodic orbits. Therefore, if $f$ has no periodic points, $\rho_V(\widehat{f})$ is a single irrational number. The rest follows from Theorems \ref{thm:A} and \ref{thm:C}. $\Box$

\vskip0.1truecm

    For $f$ defined as in~\eqref{equaI05}, an argument analogous to the one used in the proof of Proposition~4.2 of \cite{beguin2009denjoy} shows that $f$ admits a unique minimal set $\mathcal{M}$. If $\mathcal{M}\neq\mathbf{T}^2$, then Theorem~4 of \cite{jaeger2013classification} implies that $\mathcal{M}$ is an extension of a Cantor set. In fact, in our setting, every proper closed $f$-invariant set $K\subset\mathbf{T}^2$ is an extension of a Cantor set. Furthermore, by Theorem~E of \cite{kolyada2014minimal} one of the following alternatives holds: either a generic horizontal fibre of $\mathcal{M}$ is a Cantor set, or there exists an integer $N\geq 1$ such that a generic horizontal fibre of $\mathcal{M}$ has cardinality $N$.

Therefore, the remaining natural question is: Are there area-preserving twist maps of the torus whose vertical rotation set consists of a single irrational number but which are not minimal? We do not know the answer. What we know is that the ideas behind the construction of certain examples in \cite{beguin2009denjoy} can be adapted to our setting in order to produce homeomorphisms of the form 
$$
 f(x,y)=(x+\ell y +\varphi(x,y)\bmod 1, y+\alpha\bmod 1),\quad (x,y)\in\mathbf{T}^2
$$
which are transitive, but not minimal. We do not know whether they satisfy the topological twist condition, and whether they preserve a measure homeomorphic to Lebesgue.

This paper is organized as follows. In Section 2 we present some results we use and introduce important concepts and definitions. In Section 3, we prove our main theorems. In particular, subsection 3.1.2 is the heart of the paper. There, for a given circloid $\widehat{C}$ in the vertical cylinder, and a sequence of circloids $\widehat{C}_n,$ in a certain sense converging from above to $\widehat{C},$ we study how left and right components below $\widehat{C}_n$ behave as $n$ increases, with respect to the left and right components below $\widehat{C}$. Although there is no convergence in the Hausdorff sense, we nevertheless establish convergence of the measures of the relevant unions of left and right components. This property is at the heart of the proof of Theorem \ref{thm:A}. 

\section{Preliminaries and additional results}

\subsection{The semi-conjugacy}


Initially, we recall the notion of bounded deviations. Let $\widehat{f}\in\mathrm{Diff}_{\mathrm{tw}}^1(\mathbf{T}^1 \times \mathbb{R})$ and assume that its vertical rotation set satisfies $\rho_V(\widehat{f})=\{\alpha\}$. Let $f$ be the torus map lifted by $\widehat{f}.$ We say that $f$ (or $\widehat{f}$) has \emph{bounded deviations} if there exists a constant $c>0$ such that for all 
$\widehat{z}\in\mathbf{T}^1\times\mathbb{R}$ and all $n\in\mathbb{Z}$,
\begin{equation}\label{eqSC02}
    \left|p_2\circ\widehat{f}^n(\widehat{z})-p_2(\widehat{z})-n\alpha\right|\leq c.
\end{equation}
In other words, the vertical coordinate of the orbit of any point under $\widehat{f}$ follows the vertical translation by $\alpha$, up to a uniformly bounded error.

The next result follows from Theorem~3.3 of
\cite{jager2017irrational}, after interchanging the two coordinates.

\begin{theorem}\label{teoSC01}
    Let $\widehat{f}\in \mathrm{Diff}_{\mathrm{tw}}^1(\mathbf{T}^1 \times \mathbb{R})$ be an area-preserving map satisfying $\rho_{V}(\widehat{f})=\{\alpha\}$, with $\alpha \notin \mathbb{Q}$. Suppose that 
$\widehat{f}$ has bounded deviations. Then $f$, the torus map lifted by $\widehat{f},$ is semi-conjugate to an irrational rotation $R_{\alpha}$ on $\mathbf{T}^1$.
\end{theorem}

For the arguments below, we need the ordered, pairwise disjoint family
of circloids used to construct the semi-conjugacy. We therefore recall
the canonical construction from Lemma~3.2 and Section~4 of
\cite{jager2009linearization}, together with its application to the
Dehn-twist setting in the proof of Theorem~3.3 of
\cite{jager2017irrational}.

A circloid $\widehat{C}$ in $\mathbf{T}^1\times \mathbb{R}$ is a continuum satisfying the following properties: 

\begin{enumerate}
\item  it is essential; 

\item its complement has exactly two (open) connected components, one above $\widehat{C}$, denoted $\mathcal{U}^{+}(\widehat{C})$ and the other one below $\widehat{C}$, denoted $\mathcal{U}^{-}(\widehat{C})$; 

\item $\widehat{C}$ is minimal with respect to inclusion among continua satisfying properties (1) and (2);

\end{enumerate}

Write $\mathbb{A}:=\mathbf{T}^{1}\times\mathbb{R}.$ For any compact
essential set $K\subset\mathbb{A},$ denote by $\mathcal{U}^{+}(K)$
and $\mathcal{U}^{-}(K)$ the components of $\mathbb{A}\setminus K$
containing the upper and lower ends, respectively. Define
\[
\mathcal{U}^{+-}(K):=
\mathcal{U}^{-}\bigl(\partial_{\mathbb{A}}\mathcal{U}^{+}(K)\bigr),
\qquad
\mathcal{U}^{+-+}(K):=
\mathcal{U}^{+}\bigl(\partial_{\mathbb{A}}\mathcal{U}^{+-}(K)\bigr).
\]
The canonical upper circloid associated to $K$ is
\[
\mathcal{C}^{+}(K):=\mathbb{A}\setminus
\bigl(\mathcal{U}^{+-}(K)\cup\mathcal{U}^{+-+}(K)\bigr).
\]
By Lemma~3.2 of \cite{jager2009linearization} (see also Lemma~2.1
of \cite{jager2017irrational}), this is a circloid contained in $K.$
The construction commutes with homeomorphisms of $\mathbb{A}$
preserving its two ends.

Let $\widehat{f}\in\mathrm{Diff}_{\mathrm{tw}}^1(\mathbf{T}^1\times\mathbb{R})$ be a lift of the 
area-preserving $f\in\mathrm{Diff}_{\mathrm{tw}}^1(\mathbf{T}^2)$ such 
that $\rho_V(\widehat{f})=\{\alpha\}$, for some $\alpha\notin\mathbb{Q}$, and 
assume that \eqref{eqSC02} holds. For each $r\in\mathbb{R}$, define
\begin{equation}\label{eqSC03}
    A_r :=\overline{\bigcup_{n\in \mathbb{Z}}\widehat{f}^n\left(\mathbf{T}^1\times \{r-n\alpha\}\right)}.
\end{equation}
By~\eqref{eqSC02}, we have $A_r\subseteq \mathbf{T}^1\times[r-c,r+c]$.
Moreover, $A_r$ is compact and essential. We make the canonical choice
\begin{equation}\label{eqSC04}
    \widehat{C}_r:=\mathcal{C}^{+}(A_r)\subset A_r.
\end{equation}

The identities $A_r+(0,1)=A_{r+1}$ and
$\widehat{f}(A_r)=A_{r+\alpha}$ 
give
\begin{equation}\label{eqSC05}
    \widehat{C}_r+(0,1)=\widehat{C}_{r+1} \text{ and} 
\end{equation}

\begin{equation}\label{eqSC06}
    \widehat{f}(\widehat{C}_r)=\widehat{C}_{r+\alpha}.
\end{equation}

For $r,s\in\mathbb{R}$, the monotonicity of this construction also gives:
\begin{equation}\label{eqSC07}
    \text{If} \ s\leq r \ \text{then} \ \widehat{C}_s\preceq \widehat{C}_r  \Leftrightarrow \widehat{C}_r\subseteq \mathcal{U}^+(\widehat{C}_s)\cup \widehat{C}_s.
\end{equation}

Finally, the disjointness argument in the proof of Theorem~3.3 of
\cite{jager2017irrational} applies to this
canonical family and yields the stronger property:
\begin{equation}\label{eqSC08}
    \text{Disjointness property: if} \ s<r \ \text{then} \ \widehat{C}_s\prec \widehat{C}_r  \Leftrightarrow \widehat{C}_r\subset \mathcal{U}^+(\widehat{C}_s).
\end{equation}
 
We denote the open and closed regions between $\widehat{C}_s$ and $\widehat{C}_r$ by $]\widehat{C}_s,\widehat{C}_r[:= \mathcal{U}^+(\widehat{C}_s)\cap \mathcal{U}^-(\widehat{C}_r)$ and $ [\widehat{C}_s,\widehat{C}_r]:=]\widehat{C}_s,\widehat{C}_r[\cup \widehat{C}_r\cup\widehat{C}_s$, respectively.

Define the map $\widehat{h}:\mathbf{T}^1\times\mathbb{R}\to\mathbb{R}$ by
\begin{equation}\label{eqSC09}
    \widehat{h}(\widehat{z}):=\sup\left\{r\in \mathbb{R} : \widehat{z}\in \mathcal{U}^+(\widehat{C}_r)\right\},\quad \forall \ \widehat{z}\in \mathbf{T}^1\times\mathbb{R}.
\end{equation}
The map $\widehat{h}$ is well defined and satisfies the following properties:
\begin{equation}\label{eqSC10}
    \widehat{h}\circ \widehat{f}(\widehat{z})=\widehat{h}(\widehat{z})+\alpha\quad\text{and}\quad \widehat{h}(\widehat{z}+(0,1))=\widehat{h}(\widehat{z})+1,\quad \forall \ \widehat{z}\in\mathbf{T}^1\times\mathbb{R}.
\end{equation}
For each $r\in\mathbb{R}$, we have
\begin{equation}\label{eqSC11}
    \widehat{C}_r\subseteq \widehat{h}^{-1}(r).
\end{equation}

The continuity of $\widehat{h}$ follows from~\eqref{eqSC08}. Hence, by~\eqref{eqSC10}, $\widehat{h}$ projects to a semi-conjugacy $h$ between $f$ and $R_{\alpha}$.

\subsection{Fathi's techniques}

The ideas below appeared in a paper by Albert Fathi on Birkhoff's Invariant Curve Theorem \cite{fathi1983}. They play a fundamental role in the proof that the above circloids are actually graphs.

Initially, consider an open set $U\subseteq \mathbf{T}^1\times \mathbb{R}$ satisfying the following properties:
\begin{equation}\label{FTU1}
    \begin{array}{llll}
        \mathrm{(i)} & \ & U \ \text{is homeomorphic to} \ \mathbf{T}^1\times \mathbb{R};\\
        \mathrm{(ii)} &\ & \text{There exist real numbers }a<b\text{ such that} \ \mathbf{T}^1\times ]-\infty, a]\subset U \subset \mathbf{T}^1\times ]-\infty,b]; \\
        \mathrm{(iii)} & \ & \mathrm{Int}_{\mathbf{T}^1\times\mathbb{R}}(\mathrm{Cl}_{\mathbf{T}^1\times\mathbb{R}}U)=U.
    \end{array}
\end{equation} 

First, we observe that $U$, its closure $\mathrm{Cl}_{\mathbf{T}^1\times\mathbb{R}}U$, and its boundary $\partial_{\mathbf{T}^1\times\mathbb{R}} U$ are all connected sets. Moreover, since $U$ is open, it is also path-connected. Now, define
\begin{equation}\label{FTU2}
    V(U):=\left\{(\widehat{x},\widehat{y})\in U : (\widehat{x},s)\in U \ \text{for all} \ s\leq \widehat{y}\right\}.
\end{equation}
That is, $V(U)$ consists of those points in $U$ that can be reached from the horizontal curve $\mathbf{T}^1\times\{a\}$ (which is contained in $U$) by a vertical segment entirely contained in $U$. Observe that $V(U)$ is open, connected, hence path-connected, and it contains $\mathbf{T}^1\times]-\infty,a].$

By Lemmas 1–6 of~\cite{fathi1983}, the following results hold:
\begin{equation}\label{FTU3}
    \partial_UV(U)=\bigcup_{i\in I}S_i,
\end{equation}
where $I\subseteq \mathbb{N}$ and the sets $S_i$ are pairwise disjoint. Each $S_i$ is a vertical open segment of the form
$$S_i=\{\widehat{x}_i\}\times ]\widehat{y}_{i,1},\widehat{y}_{i,2}[,$$
where $a<\widehat{y}_{i,1}<\widehat{y}_{i,2}<b$ and $(\widehat{x}_i,\widehat{y}_{i,1}), (\widehat{x}_i,\widehat{y}_{i,2})\in \partial_{\mathbf{T}^1\times\mathbb{R}}U$ for all $i\in I\subseteq \mathbb{N}$. Moreover, each segment $S_i$ locally separates $U$: for any point $\widehat{w} \in S_i$ and any sufficiently small neighborhood $W$ of $\widehat{w}$, the set $W\cap {S_i}^c$ is the union of two disjoint open sets, $W_1$ and $W_2$, such that
\begin{equation}\label{FTU4}
    W_1\subset V(U)\quad\text{and}\quad W_2\subset U\setminus V(U).
\end{equation}
Hence, each set $S_i$ separates $U$ into exactly two connected components: one is homeomorphic to $\mathbf{T}^1\times\mathbb{R}$; the other, denoted $U_i$, is homeomorphic to $\mathbb{R}^2$.

The sets $U_i$ satisfy
\begin{equation}\label{FTU5}
    \mathrm{Cl}_UU_i  = U_i\cup S_i\quad\text{and}\quad V(U)\cap\mathrm{Cl}_UU_i=\emptyset.
\end{equation}
Moreover,  $U_i$ and $\mathrm{Cl}_UU_i$ are the connected components of $U\setminus \mathrm{Cl}_UV(U)$ and  $U\setminus V(U)$, respectively. 

Now, let
\begin{equation}\label{FTU6}
    \{R_j(U)\}_{j\in J\subset I}:=\{U_i \text{ as above}: U_i \ \text{is to the right of} \ S_i\},
\end{equation}
and
\begin{equation}\label{FTU7}
    \{L_k(U)\}_{k\in K\subset I}:=\{U_i \text{ as above}: U_i \ \text{is to the left of} \ S_i\}.
\end{equation}

Also, we define
\begin{equation}\label{FTU8}
    R(U):=\bigcup_{j\in J}R_j(U)\quad\text{and}\quad L(U):=\bigcup_{k\in K}L_k(U).
\end{equation}
And the above sets satisfy
\begin{equation}\label{FTU9}
    \mathrm{Cl}_UR(U)=\bigcup_{j\in J}\mathrm{Cl}_UR_j(U)\quad\text{and}\quad \mathrm{Cl}_UL(U)=\bigcup_{k\in K}\mathrm{Cl}_UL_k(U).
\end{equation}
Finally, we conclude that
\begin{equation}\label{FTU10}
    U=V(U)\cup \mathrm{Cl}_UL(U)\cup \mathrm{Cl}_UR(U).
\end{equation}

Let $\widehat{f}\in \mathrm{Diff}_{\mathrm{tw}}^1(\mathbf{T}^1\times \mathbb{R})$ be a lift of 
$f\in \mathrm{Diff}_{\mathrm{tw}}^1(\mathbf{T}^2)$, and let 
$U\subset \mathbf{T}^1\times \mathbb{R}$ be an open set satisfying the conditions in \eqref{FTU1}. In the present setting, the set $U$ is not invariant under $\widehat{f}$. 

We claim that $\widehat{f}(U)$ still satisfies properties (i)--(iii) of~\eqref{FTU1}. Indeed, it is easy to check that  $\widehat{f}(U)$ satisfies properties (i) and (iii) of \eqref{FTU1}. Define $\phi:\mathbf{T}^2\to \mathbb{R}$ by
$$\phi(z):=p_2\circ\widehat{f}(\widehat{z})-p_2(\widehat{z}),\quad \forall \ \widehat{z}\in\tau^{-1}(z).$$
The function $\phi$ is continuous and well defined; therefore, there exists a constant $Max_{disp}>0$ such that $|\phi(z)|\leq Max_{disp}$ for all $z\in\mathbf{T}^2$. 
So, from $\mathbf{T}^1\times ]-\infty,a]\subset U\subset\mathbf{T}^1\times ]-\infty,b]$ we deduce that

$$
\mathbf{T}^1\times]-\infty,a-Max_{disp}]\subset\widehat{f}(U)\subset \mathbf{T}^1\times]-\infty,b+Max_{disp}].
$$
Therefore, property (ii) holds for $\widehat{f}(U)$. We conclude that $\widehat{f}(U)$ satisfies all the properties in~\eqref{FTU1}. In particular, decomposition~\eqref{FTU10} remains valid for $
\widehat{f}(U)$, that is,
$$
\widehat{f}(U)=V(\widehat{f}(U))\cup\mathrm{Cl}_{\widehat{f}(U)}L(\widehat{f}(U))\cup\mathrm{Cl}_{\widehat{f}(U)}R(\widehat{f}(U)).
$$

The same argument applies to $\widehat{f}^{-1}(U)$, and hence \eqref{FTU10} also holds for $\widehat{f}^{-1}(U)$. We therefore obtain the following lemma, whose proof follows the same lines as the proofs of Lemmas 7–10 in~\cite{fathi1983}.

\begin{lemma}\label{lemFTU01}
    Let $\widehat{f}\in\mathrm{Diff}_{\mathrm{tw}}^1(\mathbf{T}^1\times\mathbb{R})$ be a lift of a right twist map  $f\in\mathrm{Diff}_{\mathrm{tw}}^1(\mathbf{T}^2)$ and let 
$U\subset \mathbf{T}^1\times \mathbb{R}$ be an open set satisfying the conditions in \eqref{FTU1}. Then the following consequences  hold.
    \begin{enumerate}
        \item[$(1)$] $\widehat{f}(V(U))\cap \mathrm{Cl}_{\widehat{f}(U)}L(\widehat{f}(U))=\emptyset$;
        \item[$(2)$] $\widehat{f}^{-1}(V(U))\cap \mathrm{Cl}_{\widehat{f}^{-1}(U)}R(\widehat{f}^{-1}(U))=\emptyset$;
        \item[$(3)$] $\widehat{f}(\mathrm{Cl}_UR(U))\subset R(\widehat{f}(U))$;
        \item[$(4)$] $\widehat{f}^{-1}(\mathrm{Cl}_UL(U))\subset L(\widehat{f}^{-1}(U))$. 
    \end{enumerate}
\end{lemma}

\subsection{Le Calvez's theory}

Here we collect several topological results due to Le Calvez, taken from \cite{le1986existence,le1991proprietes}. Let $\widehat{f}\in\mathrm{Diff}_{\mathrm{tw}}^1(\mathbf{T}^1\times\mathbb{R})$ and $\widetilde{f}\in\mathrm{Diff}_{\mathrm{tw}}^1(\mathbb{R}^2)$ be a lift of $\widehat{f}$. For every $s\in \mathbb{Z}$ and every positive integer $q>0$, we define 
$$\widetilde{K}(s,q):=\left\{(\widetilde{x},\widetilde{y})\in \mathbb{R}^2 : p_1\circ \widetilde{f}^q(\widetilde{x},\widetilde{y}) = \widetilde{x} + s \right\} $$
and
$$K(s,q)=\pi \left(\widetilde{K}(s,q)\right).$$
We have the following important lemma.

\begin{lemma}\label{lemTRLC01}
    The set $K(s,q)$ is compact and admits a unique connected component $C(s,q)$ which separates the two ends of $\mathbf{T}^1\times \mathbb{R}$.
\end{lemma}

%


Fix $s\in \mathbb{Z}$ and a positive integer $q$. For each $\widehat{x}\in\mathbf{T}^1$, we define the functions 
$$\mu^{-}(\widehat{x}):=\min\left\{p_2(\widehat{z}) : \widehat{z}\in K(s,q), \  \ p_1(\widehat{z})=\widehat{x}\right\},$$
and
$$\mu^{+}(\widehat{x}):=\max\left\{p_2(\widehat{z}) : \widehat{z}\in K(s,q), \ \ p_1(\widehat{z})=\widehat{x}\right\}.$$
Similarly, we associate to $\widehat{f}^q\left(K(s,q)\right)$ the functions
$$\nu^{-}(\widehat{x}):=\min\left\{p_2(\widehat{z}) : \widehat{z}\in \widehat{f}^q\left(K(s,q)\right), \ \  p_1(\widehat{z})=\widehat{x}\right\},$$
and
$$\nu^{+}(\widehat{x}):=\max\left\{p_2(\widehat{z}) : \widehat{z}\in \widehat{f}^q\left(K(s,q)\right), \ \ p_1(\widehat{z})=\widehat{x}\right\}.$$

The functions $\mu^{\pm}$ and $\nu^{\pm}$ enjoy the following properties, summarized in the next lemmas.

\begin{lemma}\label{lemTRLC03}
    Define $\mathrm{Graph}(\mu^{\pm}):=\bigl\{(\widehat{x},\mu^{\pm}(\widehat{x})) : \widehat{x}\in\mathbf{T}^1\bigr\}$. Then 
    $$\mathrm{Graph}(\mu^+)\cup\mathrm{Graph}(\mu^-)\subset C(s,q).$$
\end{lemma}

\begin{lemma}\label{lemTRLC04}
    For all $\widehat{x}\in\mathbf{T}^1$,
    $$\widehat{f}^q(\widehat{x},\mu^{-}(\widehat{x}))=(\widehat{x},\nu^{+}(\widehat{x}))\quad\text{and}\quad \widehat{f}^q(\widehat{x},\mu^{+}(\widehat{x}))=(\widehat{x},\nu^{-}(\widehat{x})).$$
\end{lemma}

The following result follows from the ideas developed in \cite{le1995construction} and will be stated in the form of a lemma.

\begin{lemma}\label{lemTRLC05}
    Let $p,s\in\mathbb{Z}$ and let $q$ be a positive integer. Assume that there is no point $(\widetilde{x},\widetilde{y})\in\mathbb{R}^2$ such that
    $$\widetilde{f}^q(\widetilde{x},\widetilde{y})=(\widetilde{x}+s,\widetilde{y}+p).$$
    Then there exists a continuous function $\theta:\mathbf{T}^1\to \mathbb{R}$ whose graph separates
    $$\widehat{f}^q(K(s,q))\quad\text{and}\quad K(s,q)+(0,p).$$
\end{lemma}

\subsection{Distality and proximal pairs}

The following definitions and results are taken from Auslander's book~\cite{auslander1988minimal}. 
Let $X$ be a compact metric space and $f:X\to X$ a homeomorphism. Recall that $f$ is \emph{distal} if for every pair of distinct points $x,y\in X$,
\begin{equation}\label{eqDPH01}
    \inf_{n\in\mathbb{Z}}d\bigl(f^n(x),f^n(y)\bigr)>0.
\end{equation}
Equivalently, for $x,y\in X$,
\begin{equation}\label{eqDPH02}
    \inf_{n\in\mathbb{Z}}d\bigl(f^n(x),f^n(y)\bigr)=0\Longrightarrow x=y.
\end{equation}

If $f$ is not distal, then there exist distinct points $x,y\in X$ forming a \emph{proximal pair}, that is,
\begin{equation}\label{eqDPH03}
    \inf_{n\in\mathbb{Z}}d\bigl(f^n(x),f^n(y)\bigr)=0.
\end{equation}

\begin{remark}
    Observe that both distality and the relation of proximality between points are invariant under conjugation.
\end{remark}

A classical result in topological dynamics asserts that, on compact metric spaces, distality together with transitivity implies minimality: 

\begin{lemma}[\cite{auslander1988minimal}, Auslander]\label{lemDPH01}
     Let $f:X\to X$ be a homeomorphism of the compact metric space $X$. If $f$ is distal and topologically transitive, then $f$ is minimal.
\end{lemma}

\section{Proofs}

\subsection{Proof of Theorem~\ref{thm:A}}

\emph{Sketch of the proof.} We first show that $f$ has bounded deviations.
It then follows from Theorem~\ref{teoSC01} that $f$ is semi-conjugate to the
irrational rotation $R_{\alpha}$ on $\mathbf{T}^1$. Next, we prove that the
family of circloids $\{C_t\}_{t\in \mathbf{T}^1}\subseteq\mathbf{T}^2$ built in subsection 2.1
consists of graphs of uniformly Lipschitz functions over the horizontal 
coordinate. This family forms an
$f$-invariant foliation.

\subsubsection{Bounded deviations}


The next theorem is the bounded deviations result.

\begin{theorem}
\label{thmBD} Let $\widetilde{f}\in \mathrm{Diff}_{\mathrm{tw}}^{1}(\mathbb{R}^{2})$ be
such that $\rho _{V}(\widetilde{f})=\{\alpha \}$ for some $\alpha \notin %
\mathbb{Q}$. If $f$ is the torus twist map lifted by $\widetilde{f}$, 
then $f$ has bounded deviations.
\end{theorem}

Before proving this theorem, we establish some notation. Fix 
$\widetilde{f}\in \mathrm{Diff}_{\mathrm{tw}}^{1}(\mathbb{R}^{2})$. It can 
be written in the form 
\begin{equation}
\widetilde{f}(\widetilde{x},\widetilde{y})=(\widetilde{x}+\ell \widetilde{y}+%
\widetilde{\varphi }_{1}(\widetilde{x},\widetilde{y}),\widetilde{y}+%
\widetilde{\varphi }_{2}(\widetilde{x},\widetilde{y})),\quad (\widetilde{x},%
\widetilde{y})\in \mathbb{R}^{2},  \label{eqBD01}
\end{equation}
where $\ell \in \mathbb{Z}\setminus \{0\}$, and the functions $\widetilde{%
\varphi }_{i}:\mathbb{R}^{2}\to \mathbb{R}$, $i=1,2$, are $\mathbb{Z}^{2}$%
-periodic. Without loss of generality, assume $\ell >0$. Then, there exist
constants $Max_{disp},\ bound,\ k_{tw}>0$ such that for every $\widetilde{z}%
=(\widetilde{x},\widetilde{y})\in \mathbb{R}^{2}$ the following properties
hold: 
\begin{equation}
\left| p_{2}\circ \widetilde{f}(\widetilde{z})-p_{2}(\widetilde{z})\right|
<Max_{disp};  \label{eqBD02}
\end{equation}

\begin{equation}
\displaystyle\dfrac{\partial \ p_{1}\circ \widetilde{f}(\widetilde{z})}{%
\partial \widetilde{y}}>k_{tw}\quad \text{(twist condition)};  \label{eqBD03}
\end{equation}

\begin{equation}
\left| \dfrac{\partial \ p_{1}\circ \widetilde{f}(\widetilde{z})}{\partial 
\widetilde{x}}\right| <bound.  \label{eqBD04}
\end{equation}

With the notation introduced above, define 
\begin{equation}
C:=6+\frac{2+bound}{k_{tw}}+|\alpha|+Max_{disp}>0.  \label{eqBD05}
\end{equation}

Then we have the following lemma (cf.~\cite[Lemma 3.1]{addas2007simple}).

\begin{lemma}
\label{lemBD01} Let $\widetilde{f}\in \mathrm{Diff}_{\mathrm{tw}}^{1}(%
\mathbb{R}^{2})$ and assume that for some real number $\alpha $, one of the
possibilities below holds:

\begin{enumerate}
\item  there exist $\widetilde{z}_{0}\in \mathbb{R}^{2}$ and an integer $N_{0}>1$ such
that $p_{2}\circ \widetilde{f}^{N_{0}}(\widetilde{z}_{0})-p_{2}(\widetilde{z}%
_{0})-N_{0}\alpha >C;$

\item  there exist $\widetilde{z}_{1}\in \mathbb{R}^{2}$ and an integer $N_{1}>1$ such
that $p_{2}\circ \widetilde{f}^{N_{1}}(\widetilde{z}_{1})-p_{2}(\widetilde{z}%
_{1})-N_{1}\alpha <-C;$
\end{enumerate}

Then, respectively, there exist points $\widetilde{z}^{\prime },\widetilde{w}%
^{\prime }\in \mathbb{R}^{2}$ and integers $n_{1}>0$ and $n_{2}>0$ such that

\begin{itemize}
\item[$(1)$]  $\widetilde{f}^{n_{1}}(\widetilde{z}^{\prime })=\widetilde{z}%
^{\prime }+(s,n_{1}\alpha +\mathrm{pos})$, for some $s\in \mathbb{Z}$ and $%
\mathrm{pos}\geq 2,$

\item[$(2)$]  $\widetilde{f}^{n_{2}}(\widetilde{w}^{\prime })=\widetilde{w}%
^{\prime }+(l,n_{2}\alpha +\mathrm{neg})$, for some $l\in \mathbb{Z}$ and $%
\mathrm{neg}\leq -2$.
\end{itemize}
\end{lemma}

\begin{proof} The proof of the existence of $\widetilde{w}^{\prime }$ is analogous
to the one for $\widetilde{z}^{\prime },$ so we omit it. Translating $%
\widetilde{z}_{0}$ by an integer vector, we can assume that $\widetilde{z}%
_{0}\in [0,1]^{2}.$ 

If the conclusion already holds with $n_1=N_0$, we are done. Otherwise,
for all $\widetilde{z}\in [0,1]^{2}$ such that
$\widetilde{f}^{N_{0}}(\widetilde{z})=\widetilde{z}+(s,m),$ for some integer 
$s,$ we have $m<N_0\alpha+2.$

Consider the vertical segment $r=\{p_{1}(\widetilde{%
z}_{0})\}\times [0,1].$ Define: 
\begin{equation}
Max.H.L(\widetilde{f}^{n}(r)):={\sup_{t,t^{\prime }\in [0,1]}}\left|
p_{1}\circ \widetilde{f}^{n}(p_{1}(\widetilde{z}_{0}),\widetilde{t}%
)-p_{1}\circ \widetilde{f}^{n}(p_{1}(\widetilde{z}_{0}),t^{\prime })\right|
\label{chma}
\end{equation}

It is clear that 
\begin{equation}
Max.H.L(\widetilde{f}^{n}(r))\geq n.\left| \ell\right| \stackrel{n\rightarrow
\infty }{\rightarrow }\infty .  \label{chmaxde}
\end{equation}

So for all $n\geq 1,$ there exists at least one integer $s$ such that $%
p_{1}(\widetilde{z}_{0})+s\in p_{1}\left( \widetilde{f}^{n}(r)\right) .$

This assumption implies that for all $0\leq t\leq 1$
such that

\begin{equation}
p_{1}\circ \widetilde{f}^{N_{0}}(p_{1}(\widetilde{z}_{0}),t)=p_{1}(%
\widetilde{z}_{0})\text{ }\text{mod 1},  \label{hipcont1}
\end{equation}
we have:

\begin{equation}
p_{2}\circ \widetilde{f}^{N_{0}}(p_{1}(\widetilde{z}_{0}),t)<N_0\alpha+2+t\leq N_0\alpha+3
\label{hipcont2}
\end{equation}

As 
$$p_{2}\circ \widetilde{f}^{N_{0}}(\widetilde{z}_{0})-p_{2}(\widetilde{z}%
_{0})>N_{0}\alpha +C=N_{0}\alpha +6+\frac{2+bound}{k_{tw}}+|\alpha
|+Max_{disp},$$ there exists $\widetilde{z}^{*}\in r$ such that:

\[
\begin{array}{c}
p_{2}\circ \widetilde{f}^{N_{0}}(\widetilde{z}^{*})=6+N_{0}\alpha +|\alpha
|+Max_{disp} \\ 
\text{and} \\ 
\begin{array}{l}
\forall \text{ }\widetilde{m}\in r^{\prime }:=\{\text{line segment with
endpoints }\widetilde{z}^{*}\text{ and }\widetilde{z}_{0}\}, \\ 
p_{2}\circ \widetilde{f}^{N_{0}}(\widetilde{m})\geq 6+N_{0}\alpha +|\alpha
|+Max_{disp}
\end{array}
\end{array}
\]
The reason why such a point $\widetilde{z}^{*}$ exists is the following: For
at least one integer $s,$ $$p_{1}(\widetilde{z}_{0})+s\in p_{1}\left( 
\widetilde{f}^{N_{0}}(r)\right).$$ Thus, from (\ref{hipcont1}) and (\ref
{hipcont2}), $\widetilde{f}^{N_{0}}(r)$ must cross the line $l$ given by 
$$
l:=\mathbb{R}\times \{6+N_{0}\alpha +|\alpha |+Max_{disp}\}.
$$

Also from (\ref{hipcont1}) and (\ref{hipcont2}) we have 
\[
{\sup_{\widetilde{m},\widetilde{k}\in {r}^{\prime }}}\left| p_{1}\circ 
\widetilde{f}^{N_{0}}(\widetilde{m})-p_{1}\circ \widetilde{f}^{N_{0}}(%
\widetilde{k})\right| <1,
\]
and it is clear that 
\[
p_{2}\circ \widetilde{f}^{N_{0}}(\widetilde{z}_{0})-p_{2}\circ \widetilde{f}%
^{N_{0}}(\widetilde{z}^{*})>\frac{(2+bound)}{k_{tw}}.
\]

So considering $\widetilde{f}(\widetilde{f}^{N_{0}}(r^{\prime})),$ we have 

\[
{\sup_{\widetilde{m},\widetilde{k}\in {r}^{\prime }}}\left| p_{1}\circ 
\widetilde{f}^{N_{0}+1}(\widetilde{m})-p_{1}\circ \widetilde{f}^{N_{0}+1}(%
\widetilde{k})\right| \geq -bound+k_{tw}\frac{(2+bound)}{k_{tw}}\geq 2
\]

And finally, from the choice of $r^{\prime },$ the following holds: 
$$
{\inf_{\widetilde{m}\in r^{\prime }}} p_{2}\left( \widetilde{f}^{_{^{N_{0}+1}}}(\widetilde{m})\right)
 \geq 6+N_{0}\alpha +|\alpha |>5+(N_{0}+1)\alpha
$$
There exists a point in $r^{\prime}$ whose image under $\widetilde{f}^{_{^{N_{0}+1}}}$ lies on an integer translate of the
vertical line through $\widetilde{z}_0$ and has vertical coordinate greater than $5+(N_{0}+1)\alpha.$
Since $r^{\prime}\subset r$, the conclusion holds with $n_{1}=N_{0}+1$.
\end{proof}

We are now ready to prove the bounded deviations theorem:

\begin{proof}[Proof of Theorem~\ref{thmBD}]
    Suppose, for a contradiction, that $f$ has unbounded deviations. Then,  there exist $\widetilde{z}\in [0,1]^2$ and a positive integer 
$n>0$ such that
    \begin{equation}\label{eqBD06}
        \left|p_2\circ\widetilde{f}^n(\widetilde{z})-p_2(\widetilde{z})-n\alpha\right|>C,
    \end{equation}
where $C$ comes from expression (\ref{eqBD05}).

    From expression~\eqref{eqBD06}, there are two cases.
    
    \noindent\textbf{Case 1.} Assume $p_2\circ\widetilde{f}^n(\widetilde{z})-p_2(\widetilde{z})-n\alpha>C.$ By Lemma~\ref{lemBD01} (1), there exist a point $\widehat{w}=(\widehat{x},\widehat{y})\in\mathbf{T}^1\times [0,1]$ and an integer $q\geq 1$ such that
    $$p_2\circ \widehat{f}^q(\widehat{w})-p_2(\widehat{w})\geq q\alpha + 2\quad\text{and}\quad p_1\circ \widehat{f}^q(\widehat{w})=p_1(\widehat{w}).$$
    Hence, for some $s\in\mathbb{Z}$, we have $\widehat{w}\in K(s,q)$.
    By Lemma \ref{lemTRLC04},
    \begin{equation}\label{eqBD07}
        \nu^+(\widehat{x})-\mu^-(\widehat{x})\geq q\alpha + 2 > p > q\alpha,
    \end{equation}
    for some integer $p\in \mathbb{Z}$. Moreover, by Lemma~\ref{lemTRLC03}, $(\widehat{x},\mu^-(\widehat{x}))\in C(s,q)$. So, we have to deal with the two possibilities below.
    \begin{itemize}
        \item There exists a point  $\widehat{z}^*\in C(s,q)$ such that $p_2\circ \widehat{f}^q(\widehat{z}^*)-p_2(\widehat{z}^*)\leq p.$   Since $C(s,q)$ is compact and connected, and the function $p_2\circ \widehat{f}^q(\cdot)-p_2(\cdot)$ is continuous, it follows from~\eqref{eqBD07} that there exists $\widehat{w}^*\in C(s,q)$ such that $\widehat{f}^q(\widehat{w}^*)=\widehat{w}^* + (0,p),$ which implies the existence of a periodic point for $f\in\mathrm{Diff}_{\mathrm{tw}}^1(\mathbf{T}^2)$, a contradiction with the irrationality of $\alpha.$ 
        \item For all $\widehat{z}\in C(s,q)$, we have $p_2\circ\widehat{f}^q(\widehat{z})-p_2(\widehat{z})>p$. Then $\nu^-(\widehat{x})-\mu^+(\widehat{x})>p$ for all $\widehat{x}\in\mathbf{T}^1$. Hence, by Lemma~\ref{lemTRLC05}, there exists a continuous function $\theta:\mathbf{T}^1\to\mathbb{R}$ whose graph separates the sets $\widehat{f}^q(C(s,q))$ and $C(s,q)+(0,p)$. Moreover, $\widehat{f}^q(C(s,q))$ lies strictly above $C(s,q)+(0,p)$. Consequently, $\widehat{f}^q(\mathrm{Graph}(\theta))-(0,p)$ lies above $\mathrm{Graph}(\theta)$. 
This implies that
        $$\liminf_{j\to\infty}\dfrac{p_2\circ\widehat{f}^j(\widehat{z})-p_2(\widehat{z})}{j}\geq\dfrac{p}{q},\quad\text{for all} \ \widehat{z}\in \mathbf{T}^1\times \mathbb{R}.$$
\vskip0.1truecm
        Hence, $\rho_{V}(\widehat{f})\geq \dfrac{p}{q}>\alpha,$ again yielding a contradiction.
    \end{itemize}

    \noindent\textbf{Case 2.} Assume $p_2\circ\widetilde{f}^n(\widetilde{z})-p_2(\widetilde{z})-n\alpha<-C.$ This case is treated similarly to Case~1.
    \end{proof}

So, we obtain the following Poincar\'{e}-like Classification, whose proof follows immediately from Theorem~\ref{teoSC01} and Theorem~\ref{thmBD}.

\begin{corollary}
\label{corSC01} Let $f\in \mathrm{Diff}_{\mathrm{tw}}^{1}(\mathbf{T}^{2})$
be area-preserving such that $\rho _{V}(\widetilde{f})=\{\alpha \}$ for
some $\alpha \notin \mathbb{Q}$, where $\widetilde{f}\in \mathrm{Diff}_{%
\mathrm{tw}}^{1}(\mathbb{R}^{2})$ is a lift of $f$. Then $f$ is
semi-conjugate to the irrational rotation $R_{\alpha }$ on $\mathbf{T}^{1}$.
\end{corollary}

\subsubsection{A Birkhoff-type Invariant Curve Theorem}

In this subsection, we establish a result in the same spirit, adapted to a
different geometric and dynamical setting, as the classical \emph{Birkhoff's
Invariant Curve Theorem}. More precisely, we show that the circloids
constructed in the proof of Theorem \ref{teoSC01} must be the graphs of
uniformly Lipschitz functions over the first coordinate $\mathbf{T}^1$.

We recall that the family of circloids in the torus, $\{C_{t}\}_{t\in 
\mathbf{T}^{1}}$, 
 is obtained as 
$C_{r {\text{ mod 1}}}=\tau (\widehat{C}_{r})$, where $\widehat{C}_{r}$
comes from the construction described in subsection 2.1.

Our main result here is the following, which fully proves Theorem A.

\begin{theorem}
\label{theoBICT01} Let $\widehat{f}\in \mathrm{Diff}_{\mathrm{tw}}^{1}(%
\mathbf{T}^{1}\times \mathbb{R})$ be a lift of some area preserving twist
map $f$ in the torus, such that $\rho _{V}(\widehat{f})=\{\alpha \}$ for
some $\alpha \notin \mathbb{Q}$. Then every circloid $C_{t}\subset \mathbf{T}%
^{2}$ in the above family is the graph of a Lipschitz function $\phi _{t}%
\colon \mathbf{T}^{1}\to \mathbf{T}^{1}$, and the Lipschitz constant can be
chosen uniformly in the family. Moreover, the family $\{\phi _{t}\}_{t\in 
\mathbf{T}^{1}}$ depends continuously on $t$ in the uniform topology, and the graphs of these functions
form a foliation of $\mathbf{T}^{2}$.
\end{theorem}


\paragraph{Properties of circloids.}

For any $r\in \mathbb{R}$, let $\widehat{C}_{r}\subset \mathbf{T}^{1}\times %
\mathbb{R}$ be the circloid that comes from the construction in the proof of
Theorem~\ref{teoSC01} of subsection 2.1. Recall that $\mathcal{U}^{-}(%
\widehat{C}_{r})$ and $\mathcal{U}^{+}(\widehat{C}_{r})$ are the connected
components of the complement of $\widehat{C}_{r},$ respectively containing
the lower and upper ends of $\mathbf{T}^{1}\times \mathbb{R}$. Recall that
the family of all these circloids satisfies the properties in expressions (\ref
{eqSC05}), (\ref{eqSC06}) and (\ref{eqSC08}), that we
restate below:

\begin{itemize}
\item  $\widehat{C}_{r}+(0,1)=\widehat{C}_{r+1};$

\item  $\widehat{f}(\widehat{C}_{r})=\widehat{C}_{r+\alpha };$

\item  Disjointness property: if $s<r$, then $\widehat{C}_{s}\prec \widehat{C}%
_{r} \Leftrightarrow \widehat{C}_{r}\subset \mathcal{U}^{+}(\widehat{C}_{s});
$
\end{itemize}

Clearly, $\mathcal{U}^-(\widehat{C}_r)$ is homeomorphic to $\mathbf{T}%
^1\times\mathbb{R}$. From the fact that $\widehat{C}_r$ is contained in $A_r$
(see expression (\ref{eqSC03})), there exist real numbers $a_r=r-C<b_r=r+C,$
where $C>0$ comes from expression (\ref{eqBD05}),
such that $\widehat{C}_r$ is contained in $\mathbf{T}^1\times [a_r,b_r]$.

On the other hand, we have the following useful proposition:

\begin{proposition}
\label{areazero} For any $r\in \mathbb{R},$ both $\widehat{h}^{-1}(r)$ and $h^{-1}(r \bmod 1)$ 
have zero Lebesgue measure. The same holds for $\widehat{C}_r$ and 
$C_{r \bmod 1}=\tau (\widehat{C}_{r}).$ In particular, $\widehat{C}_r$ is a cofrontier. 
\end{proposition}

\begin{proof}

The circloid $C_{r\text{ mod 1}}$ is
contained in $h^{-1}(r$ mod $1)$, where $h$ is the semi-conjugacy between $f$
and $R_{\alpha }$. As $h^{-1}(r$ mod 1$)$ is wandering for $f$, the
Poincar\'{e} Recurrence Theorem says that $Leb(h^{-1}(r$ mod 1$))=0$,
which implies that both $C_{r\text{ mod 1}}$
and $\widehat{C}_{r}$ have no interior. So, 
$\widehat{C}_{r}$ 
is actually a \emph{cofrontier,} that is 
$$
\partial _{\mathbf{T}^{1}\times \mathbb{R}}\mathcal{U}^{-}(\widehat{C}_{r})=\widehat{C}_{r}=
\partial _{\mathbf{T}^{1}\times \mathbb{R}}\mathcal{U}^{+}(\widehat{C}_{r}).
$$ 
\end{proof}

Hence, $\mathcal{U}^{-}(\widehat{C}_{r})$ satisfies
condition~\eqref{FTU1}.

Now, from~\eqref{FTU3} and~\eqref{FTU8}, for simplicity, we write 
\[
L(\mathcal{U}^-(\widehat{C}_r)):=L(\widehat{C}_r),\quad R(\mathcal{U}^-( 
\widehat{C}_r)):=R(\widehat{C}_r),\quad\text{and}\quad V(\mathcal{U}^-( 
\widehat{C}_r)):=V(\widehat{C}_r). 
\]
Moreover, by~\eqref{FTU10}, we have the decomposition 
\begin{equation}  \label{eq.TC01}
\mathcal{U}^-(\widehat{C}_r)=V(\widehat{C}_r)\cup\mathrm{Cl}_{\mathcal{U}^-(%
\widehat{C}_r)}L(\widehat{C}_r)\cup\mathrm{Cl}_{\mathcal{U}^-(\widehat{C}%
_r)}R(\widehat{C}_r).
\end{equation}
From~\eqref{eqSC05}, we have $\mathcal{U}^{-}(\widehat{C}_{r+1})=\mathcal{U}%
^-(\widehat{C}_r)+(0,1)$. In particular, for every $q\in\mathbb{Z}$, 
\begin{equation}  \label{eq.TC02}
L(\widehat{C}_r+(0,q))=L(\widehat{C}_r)+(0,q)\quad\text{and}\quad R(\widehat{%
C}_r+(0,q))=R(\widehat{C}_r)+(0,q).
\end{equation}

Before beginning the proof of Theorem~\ref{theoBICT01}, we establish some
auxiliary results.

\begin{lemma}
\label{lemTC00} For any $r\in\mathbb{R}$, the following statements hold independently:

\begin{enumerate}
\item[$(1)$] If $R(\widehat{C}_{r})\neq\emptyset$, there exists $\delta _1>0$ such that, for every $n\in\mathbb{N}$,
$$
\mathrm{Leb}_{\mathbf{T}^{1}\times \mathbb{R}}(R(\widehat{f}%
^{n}(\widehat{C}_{r})))>\mathrm{Leb}_{\mathbf{T}^{1}\times \mathbb{R}}(R(%
\widehat{C}_{r}))+\delta _{1},
$$

\item[$(2)$] If $L(\widehat{C}_{r})\neq\emptyset$, there exists $\delta _2>0$ such that, for every $n\in\mathbb{N}$,
$$
\mathrm{Leb}_{\mathbf{T}^{1}\times \mathbb{R}}(L(\widehat{f}%
^{-n}(\widehat{C}_{r})))>\mathrm{Leb}_{\mathbf{T}^{1}\times \mathbb{R}}(L(%
\widehat{C}_{r}))+\delta _{2}.
$$
\end{enumerate}
\end{lemma}

\begin{proof}
    As both cases are analogous, let us consider only the first. 
By Lemma~\ref{lemFTU01} and expression~\eqref{FTU9}, for each $j\in J\subseteq\mathbb{N}$ there exists $j'\in J'\subseteq\mathbb{N}$ such that
    \begin{equation}\label{FTU11}
        \widehat{f}(\mathrm{Cl}_{\mathcal{U}^-(\widehat{C}_r)}R_j(\widehat{C}_r))\subset R_{j'}(\widehat{f}(\widehat{C}_r)).
    \end{equation}
    Fix $j_0\in J$. Since $\mathrm{Cl}_{\mathcal{U}^-(\widehat{C}_r)}R_{j_0}(\widehat{C}_r)=R_{j_0}(\widehat{C}_r)\cup S_{j_0}$, 
let $\widehat{z}\in S_{j_0}$. Because $\widehat{f}$ is continuous and $R_{j_0'}(\widehat{f}
(\widehat{C}_r))$ is open in $\widehat{f}(\mathcal{U}^-(\widehat{C}_r))$, it follows from~\eqref{FTU11} that there exists an open ball $B$ centered at $\widehat{z}$ such that $\widehat{f}(B)\subset R_{j_0'}
(\widehat{f}(\widehat{C}_r))$. By~\eqref{FTU4}, we have $V(\widehat{C}_r)\cap B\neq\emptyset.$ Define  $B':=V(\widehat{C}_r)\cap B$ and set  $\delta_1:=\tfrac12\mathrm{Leb}(B')>0$. Observe that $B'$ is disjoint 
from $R_{j_0}(\widehat{C}_r)$, hence $\widehat{f}(R_{j_0}(\widehat{C}_r))\cap \widehat{f}(B')=\emptyset.$ Therefore, since $\widehat{f}(B')\cup\widehat{f}(R_{j_0}(\widehat{C}_r))\subset R_{j_0'}(\widehat{f}
(\widehat{C}_r))$ and $\widehat{f}$ preserves Lebesgue measure, we obtain that
    $$\mathrm{Leb}_{\mathbf{T}^1\times\mathbb{R}}(R_{j_0'}(\widehat{f}(\widehat{C}_r)))>\mathrm{Leb}_{\mathbf{T}^1\times\mathbb{R}}(R_{j_0}(\widehat{C}_r))+\delta_1.$$
The above inequality, together with expression~\eqref{FTU8} and Lemma \ref{lemFTU01} item (3), implies that
    $$\mathrm{Leb}_{\mathbf{T}^1\times\mathbb{R}}(R(\widehat{f}(\widehat{C}_r)))>\mathrm{Leb}_{\mathbf{T}^1\times\mathbb{R}}(R(\widehat{C}_r))+\delta_1.$$

By an iterated use of Lemma \ref{lemFTU01} item(3), we obtain that
$$
\widehat{f}^{n-1}(R(\widehat{f}(\widehat{C}_r)))\subset R(\widehat{f}^n(\widehat{C}_r)),
$$
which concludes the proof of the first case.
\end{proof}

\begin{lemma}
\label{lemTC01} Let $\widehat{C}_{s},\widehat{C}_{r}\subset \mathbf{T}%
^{1}\times \mathbb{R}$ be circloids given as above, and suppose that $s<r$.
Then we have $\mathrm{Cl}_{\mathbf{T}^{1}\times \mathbb{R}}V(\widehat{C}%
_{s})\subset V(\widehat{C}_{r})$.
\end{lemma}

\begin{proof}
    As $\widehat{C}_{s}\cup\mathcal{U}^-(\widehat{C}_{s})\subset \mathcal{U}^-(\widehat{C}_r)$ and $V(\widehat{C}_s)\subset \mathcal{U}^-(\widehat{C}_{s}),$ clearly 
$$
\mathrm{Cl}_{\mathbf{T}^{1}\times \mathbb{R}}V(\widehat{C}%
_{s})\subset 
\mathcal{U}^-(\widehat{C}_r).
$$ 
So, $\mathrm{Cl}_{\mathbf{T}^{1}\times \mathbb{R}}V(\widehat{C}%
_{s})\subset V(\widehat{C}_{r}).$
\end{proof}

\begin{lemma}
\label{lemTC03} Let $\{s_{n}\}_{n\in \mathbb{N}}$ be a
sequence converging to some $r\in \mathbb{R}$ monotonically from below,
that is, $s_{n}\rightarrow r^{-}$. Then 
\[
\lim_{n\to \infty }\mathrm{Leb}_{\mathbf{T}^{1}\times \mathbb{R}}\left(
[\widehat{C}_{s_{n}},\widehat{C}_{r}]\right) =0.
\]
Analogously, if $s_{n}\to r^{+},$ that is, monotonically from above, we have 
\[
\lim_{n\to \infty }\mathrm{Leb}_{\mathbf{T}^{1}\times \mathbb{R}}\left(
[\widehat{C}_{r},\widehat{C}_{s_{n}}]\right) =0.
\]
\end{lemma}

\begin{proof}
Let $\{s_n\}_{n\in \mathbb{N}}$ be such that 
$s_n\to r^-$. By construction, 
$\widehat{C}_{s_n}\subseteq \widehat{h}^{-1}(s_n)$ and  
$\widehat{C}_r\subseteq \widehat{h}^{-1}(r).$ 

For each $n\in\mathbb{N}$, define 
the open horizontal band $B_n\subset \mathbf{T}^1\times \mathbb{R}$ as the region 
between the sets $\widehat{C}_{s_n}$ and $\widehat{C}_r$. 
Note that $B_{n+1}\subset B_n$ and 
$\mathrm{Cl}_{\mathbf{T}^1 \times \mathbb{R}}B_n\subset B_n\cup \widehat{h}^{-1}(r)\cup \widehat{h}^{-1}(s_{n}).$

    We claim that $$\bigcap_{n\in\mathbb{N}}\mathrm{Cl}_{\mathbf{T}^1 \times \mathbb{R}}B_n\subset \widehat{h}^{-1}(r).$$
    Indeed, let $\widehat{z}\in \bigcap_{n\in\mathbb{N}}\mathrm{Cl}_{\mathbf{T}^1\times \mathbb{R}}B_n$. Then $\widehat{z}\in \mathrm{Cl}_{\mathbf{T}^1\times \mathbb{R}}B_n$ for all $n\in\mathbb{N}$. Since each $B_n$ is connected, it follows that $s_{n}\leq \widehat{h}(\widehat{z})\leq r$ for all $n\in\mathbb{N}$. Letting $n\to\infty$, we obtain that $\widehat{h}(\widehat{z})=r$, and the claim is proved.

    Consequently, by Proposition \ref{areazero} and the continuity from above of the Lebesgue measure,
    $$\lim_{n\to\infty}\mathrm{Leb}_{\mathbf{T}^1\times \mathbb{R}}\left(\mathrm{Cl}_{\mathbf{T}^1\times \mathbb{R}}B_n\right)=\mathrm{Leb}_{\mathbf{T}^1\times \mathbb{R}}\left(\bigcap_{n\in\mathbb{N}}\mathrm{Cl}_{\mathbf{T}^1\times \mathbb{R}}B_n\right)\leq \mathrm{Leb}_{\mathbf{T}^1\times \mathbb{R}}(\widehat{h}^{-1}(r))=0.$$
    Therefore, 
    $$\lim_{n\to\infty}\mathrm{Leb}_{\mathbf{T}^1\times \mathbb{R}}\left([\widehat{C}_{s_{n}},\widehat{C}_r]\right)=0.$$

    The proof of the remaining case is entirely analogous.
\end{proof}

The next lemma almost concludes the proof of the main result of this
subsection:

\begin{lemma}
\label{lemacorrec} For any circloid $\widehat{C}_{r}$ as before,
the open sets $R(\widehat{C}_{r})$ and $L(\widehat{C}_{r})$ are both empty.
\end{lemma}

\begin{proof}
Without loss of generality, suppose, for a contradiction, that for some $r\in \mathbb{R}$, $R(\widehat{C}_{r})\neq
\emptyset .$ 

Consider two sequences, $n_{i}^{+},n_{i}^{-}\rightarrow +\infty 
$ such that for adequate choices of integers $p_{i}^{+}$ and $p_{i}^{-}$, the following holds:
\[
n_{i}^{+}\alpha+p_{i}^{+}\rightarrow 0^{+}\text{ and } -n_{i}^{-}\alpha+p_{i}^-\rightarrow 0^{+}.
\]
Define
\[
\widehat{C}_{i}^{+}:=\widehat{f}^{n_{i}^{+}}(\widehat{C}_{r})+(0,p_{i}^{+})%
\text{ and }\widehat{C}_{i}^{-}:=\widehat{f}^{-n_{i}^{-}}(\widehat{C}%
_{r})+(0,p_{i}^{-}). 
\]
These are circloids above $\widehat{C}_{r}$ and 
\begin{equation}
\lim_{i\to \infty }\mathrm{Leb}_{\mathbf{T}^{1}\times \mathbb{R}}\left( [%
\widehat{C}_{r},\widehat{C}_{i}^{+}]\right) =\lim_{i\to \infty }\mathrm{Leb}%
_{\mathbf{T}^{1}\times \mathbb{R}}\left( [\widehat{C}_{r},\widehat{C}%
_{i}^{-}]\right) =0.  \label{areavaizero}
\end{equation}
By considering subsequences if necessary, we can assume that for each
integer $i>0,$%
\begin{equation}
\widehat{C}_{r}\prec \widehat{C}_{i+1}^{+}\prec \widehat{C}_{i}^{-}\prec 
\widehat{C}_{i}^{+}.  \label{ordena}
\end{equation}
So the horizontal curve $\mathbf{T}^{1}\times \{a_{r}-1\}$ which is contained
in $\mathcal{U}^{-}(\widehat{C}_{r}),$ is also contained in $\mathcal{U}^{-}(%
\widehat{C}_{i}^{+})$ and in $\mathcal{U}^{-}(\widehat{C}_{i}^{-})$ for all
integers $i>0.$

For any circloid $\widehat{C}$ satisfying $\widehat{C}_{r}\preceq \widehat{C},$ define
$V_{a_{r}}(\widehat{C}):=V(\widehat{C})\cap \mathbf{T}^{1}\times
[a_{r},+\infty [.$ From expression (\ref{ordena}) we get that 
\[
\begin{array}{c}
\mathrm{Leb}_{\mathbf{T}^{1}\times \mathbb{R}}(V_{a_{r}}(\widehat{C}%
_{i}^{+}))>\mathrm{Leb}_{\mathbf{T}^{1}\times \mathbb{R}}(V_{a_{r}}(\widehat{%
C}_{i}^{-}))> \\ 
>\mathrm{Leb}_{\mathbf{T}^{1}\times \mathbb{R}}(V_{a_{r}}(\widehat{C}%
_{i+1}^{+}))>\mathrm{Leb}_{\mathbf{T}^{1}\times \mathbb{R}}(V_{a_{r}}(%
\widehat{C}_{r})).
\end{array}
\]
So, $\lim\limits_{i\to \infty }\mathrm{Leb}_{\mathbf{T}^{1}\times \mathbb{R}%
}(V_{a_{r}}(\widehat{C}_{i}^{+}))=\lim\limits_{i\to \infty }\mathrm{Leb}_{%
\mathbf{T}^{1}\times \mathbb{R}}(V_{a_{r}}(\widehat{C}_{i}^{-}))=\lambda
_{V}\geq \mathrm{Leb}_{\mathbf{T}^{1}\times \mathbb{R}}(V_{a_{r}}(\widehat{C}%
_{r})).$ The last inequality is actually an equality, but as we do not need
it in our proof, for now we keep the inequality.

By the decomposition~\eqref{eq.TC01} and the ordering~\eqref{ordena}, for all $i>0,$
\begin{equation}  \label{ajudalim}
\begin{aligned}
&\mathrm{Cl}_{\mathcal{U}^-(\widehat{C}_i^+)}R(\widehat{C}_i^+)
\cup\mathrm{Cl}_{\mathcal{U}^-(\widehat{C}_i^+)}L(\widehat{C}_i^+)
\cup V_{a_r}(\widehat{C}_i^+)\\
\supset{}&\mathrm{Cl}_{\mathcal{U}^-(\widehat{C}_i^-)}R(\widehat{C}_i^-)
\cup\mathrm{Cl}_{\mathcal{U}^-(\widehat{C}_i^-)}L(\widehat{C}_i^-)
\cup V_{a_r}(\widehat{C}_i^-)\\
\supset{}&\mathrm{Cl}_{\mathcal{U}^-(\widehat{C}_{i+1}^+)}R(\widehat{C}_{i+1}^+)
\cup\mathrm{Cl}_{\mathcal{U}^-(\widehat{C}_{i+1}^+)}L(\widehat{C}_{i+1}^+)
\cup V_{a_r}(\widehat{C}_{i+1}^+)\\
\supset{}&\mathrm{Cl}_{\mathcal{U}^-(\widehat{C}_r)}R(\widehat{C}_r)
\cup\mathrm{Cl}_{\mathcal{U}^-(\widehat{C}_r)}L(\widehat{C}_r)
\cup V_{a_r}(\widehat{C}_r).
\end{aligned}
\end{equation}
Indeed, each union is $\mathcal{U}^-(\widehat{C})\cap
(\mathbf{T}^1\times[a_r,+\infty[)$ for its corresponding circloid
$\widehat{C}$. By~\eqref{FTU5} and~\eqref{FTU9}, the relative closures
only add countably many vertical segments. Thus these closures do not change the measures
used below.

On the other hand, every circloid in this family is contained in a closed
vertical annulus of height $2C,$ where $C>0$ is the uniform bound in the
vertical displacement with respect to the translation by $\alpha ,$ see
expression (\ref{eqBD05}). This means that 
$0\leq \mathrm{Leb}%
_{\mathbf{T}^{1}\times \mathbb{R}}(R(\widehat{C}_{i}^{+}))\leq 2C$ and $%
0\leq \mathrm{Leb}_{\mathbf{T}^{1}\times \mathbb{R}}(L(\widehat{C}%
_{i}^{+}))\leq 2C.$ So, again considering subsequences if necessary, we may
assume that the following limits exist: 
\[
\begin{array}{c}
\lim\limits_{i\to \infty }\mathrm{Leb}_{\mathbf{T}^{1}\times \mathbb{R}}(R(%
\widehat{C}_{i}^{+}))=\lambda _{R} \\ 
and \\ 
\lim\limits_{i\to \infty }\mathrm{Leb}_{\mathbf{T}^{1}\times \mathbb{R}}(L(%
\widehat{C}_{i}^{+}))=\lambda _{L}
\end{array}
\]

Actually, a more careful analysis using Proposition \ref{areaaa} below shows
that the limits above exist even without considering subsequences.

From expressions (\ref{areavaizero}) and (\ref{ajudalim}), we get that 
\[
\lambda _{R}+\lambda _{L}\leq \mathrm{Leb}_{\mathbf{T}^{1}\times \mathbb{R}%
}(R(\widehat{C}_{r}))+\mathrm{Leb}_{\mathbf{T}^{1}\times \mathbb{R}}(L(%
\widehat{C}_{r})). 
\]
In the previous inequality, we are using that $\lambda _{V}\geq \mathrm{Leb}%
_{\mathbf{T}^{1}\times \mathbb{R}}(V_{a_{r}}(\widehat{C}_{r})).$

Recall that the map $\widehat{h}:\mathbf{S}^1\times \mathbb{R}\rightarrow %
\mathbb{R}$ built in the proof of Theorem \ref{teoSC01} is continuous and
satisfies $\widehat{h}(\widehat{f}(\widehat{z}))=\widehat{h}(\widehat{z}%
)+\alpha .$

One last piece of information that follows from the way $\widehat{h}$ is
constructed is the following: As $i\rightarrow \infty ,$ both $\widehat{C}%
_{i}^{+}$ and $\widehat{C}_{i}^{-}$ converge in the Hausdorff topology to 
\begin{equation}
K_{r}^{+}:=\left( \widehat{h}^{-1}(r)\cap \mathcal{U}^{+}(\widehat{C}%
_{r})\right) \cup \widehat{C}_{r}=\widehat{h}^{-1}(r)\cap (\mathcal{U}^{-}(%
\widehat{C}_{r}))^{c}.  \label{defkrmais}
\end{equation}

Now we are ready to state the last proposition we need:

\begin{proposition}
\label{areaaa} For any monotone sequence $s_{n}\rightarrow r^{+},$ the
corresponding circloids $\widehat{C}_{s_{n}}$ satisfy: $\lim\limits_{n\to
\infty }\mathrm{Leb}_{\mathbf{T}^{1}\times \mathbb{R}}(R(\widehat{C}%
_{s_{n}}))=\lambda _{R}$ and $\lim\limits_{n\to \infty }\mathrm{Leb}_{%
\mathbf{T}^{1}\times \mathbb{R}}(L(\widehat{C}_{s_{n}}))=\lambda _{L}.$
\end{proposition}

\begin{proof}

Let $s_{n}$ be a monotone sequence converging to $r$ from above, and let 
$\widehat{C}_{s_{n}}$ be the corresponding circloids. 
Consider the set 
\begin{equation}
V_{\infty }:=\bigcap _{n=0}^{+\infty }\mathrm{Cl}_{\mathbf{T}^{1}\times \mathbb{R}}(V(%
\widehat{C}_{s_n})).  \label{vinfinito}
\end{equation}
Let us understand some of its basic properties. As $\mathbf{T}^{1}\times
]-\infty ,a_{r}-1]\subset \mathcal{U}^{-}(\widehat{C}_{r})$ and for all $%
n>0, $ $\widehat{C}_{r}\prec \widehat{C}_{s_{n+1}}\prec \widehat{C}_{s_{n}},$
we get that $\mathbf{T}^{1}\times ]-\infty ,a_{r}-1]\subset V(\widehat{C}_{s_n})$
and $V(\widehat{C}_{s_{n+1}})\subset V(\widehat{C}_{s_n}).$ So $V_{\infty }$ is
not empty. And if $\widehat{w}\in \partial _{\mathbf{T}^{1}\times \mathbb{R}%
}V(\widehat{C}_{s_{n+1}}),$ then as $V(\widehat{C}_{s_{n+1}})$ is open, there exists 
$\widehat{w}_{i}\in V(\widehat{C}_{s_{n+1}})$ such that $\widehat{w}%
_{i}\rightarrow \widehat{w}.$ But this implies that 
\begin{equation}
\{p_{1}(\widehat{w})\}\times ]-\infty ,p_{2}(\widehat{w})]:=V^{-}_{\widehat{w}}\subset \mathrm{Cl}_{%
\mathbf{T}^{1}\times \mathbb{R}}(V(\widehat{C}_{s_{n+1}})).  \label{ddd}
\end{equation}
So, 
$V^{-}_{\widehat{w}}\subset 
\mathcal{U}^{-}(\widehat{C}_{s_{n+1}})\cup \widehat{C}_{s_{n+1}}\subset \mathcal{U}%
^{-}(\widehat{C}_{s_n}),$ which implies that 
$$
\mathrm{Cl}_{\mathbf{T}^{1}\times %
\mathbb{R}}(V(\widehat{C}_{s_{n+1}}))\subset V(\widehat{C}_{s_n}).
$$ 
This tells us
that although the sets $V(\widehat{C}_{s_n})$ are open, $V_{\infty }$ is equal
to $\cap _{n=0}^{+\infty }V(\widehat{C}_{s_n}).$

And if $\widehat{w}\in V_{\infty },$ then $\widehat{w}\in \mathrm{Cl}_{\mathbf{T}%
^{1}\times \mathbb{R}}V(\widehat{C}_{s_n})$ for all $n\geq 0.$ So expression (%
\ref{ddd}) implies that 
$V^{-}_{\widehat{w}}\subset V_{\infty },$ and therefore 
\begin{equation}
\begin{gathered}
V_{\infty }\text{ is a closed, connected set that is invariant under} \\
\text{negative vertical translations and contains }\mathbf{T}^{1}\times]-\infty ,a_{r}-1]. \label{propvinf}
\end{gathered}
\end{equation}
Finally, we claim that 
\begin{equation}
V_{\infty }\cap \left( \mathcal{U}^{+}(\widehat{C}_{r})\right) ^{c}=\{%
\widehat{z}\in \widehat{C}_{r}\cup \mathcal{U}^{-}(\widehat{C}_{r}):V^{-}_{\widehat{z}}
\cap \mathcal{U}^{+}(%
\widehat{C}_{r})=\emptyset \}.  \label{caracvinf}
\end{equation}
To see this, assume $\widehat{z}$ belongs to the second set. This means that 
$
V^{-}_{\widehat{z}}\subset 
\mathcal{U}^{-}(\widehat{C}_{r})\cup \widehat{C}_{r}\subset \mathcal{U}^{-}(%
\widehat{C}_{s_n})$ for all $n\geq 0.$ So $V^{-}_{\widehat{z}}
\subset V(\widehat{C}_{s_n}),$ and thus 
$$
V^{-}_{\widehat{z}}
\subset V_{\infty }\cap
\left( \mathcal{U}^{+}(\widehat{C}_{r})\right) ^{c}.
$$ 
Now, pick some $%
\widehat{z}\in V_{\infty }\cap \left( \mathcal{U}^{+}(\widehat{C}%
_{r})\right) ^{c}.$ Suppose that $V^{-}_{\widehat{z}}
\cap \mathcal{U}^{+}(\widehat{C}_{r})\neq \emptyset .$
This means that there exists a segment $\{p_{1}(\widehat{z})\}\times
]a,b[\subset \mathcal{U}^{+}(\widehat{C}_{r})$ for some $b\leq p_{2}(\widehat{z}%
),$ and both endpoints $(p_{1}(\widehat{z}),a)$ and $(p_{1}(\widehat{z}),b)$
belong to $\widehat{C}_{r}.$ The complement of $\widehat{C}_{r}\cup \{p_{1}(%
\widehat{z})\}\times ]a,b[$ has a bounded connected component $D_{a,b}$
which is contained in $\mathcal{U}^{+}(\widehat{C}_{r}).$ As $\lim\limits_{n\to
\infty }\mathrm{Leb}_{\mathbf{T}^{1}\times \mathbb{R}}\left( [\widehat{C}%
_{r},\widehat{C}_{s_{n}}]\right) =0,$ we obtain that for sufficiently large $%
n>0,$ $\widehat{C}_{s_n}$ intersects $D_{a,b},$ and so, $\mathcal{U}^{+}(%
\widehat{C}_{s_n})$ intersects $\{p_{1}(\widehat{z})\}\times ]a,b[.$ And this
is a contradiction with the choice of $\widehat{z},$ because it belongs to $\cap
_{n=0}^{+\infty }V(\widehat{C}_{s_n}).$

\begin{description}
\item[Fact 1.]  {\it As $n\rightarrow \infty ,$ the sets $V(%
\widehat{C}_{s_n})$ converge to $V_{\infty }$ in the Hausdorff topology.}
\end{description}

\begin{proof} Given $\varepsilon >0,$ and $n\geq 0,$ the $\varepsilon $%
-neighborhood of $V(\widehat{C}_{s_n})$ contains $V_{\infty }.$ This follows
from the fact that $V_{\infty }\subset \mathrm{Cl}_{\mathbf{T}^{1}\times \mathbb{R}%
}(V(\widehat{C}_{s_{n+1}}))\subset V(\widehat{C}_{s_n}).$

So it remains to show that there exists $n_{0}\geq 0$ such that for all $%
n\geq n_{0},$ $V(\widehat{C}_{s_n})$ is contained in the $\varepsilon $%
-neighborhood of $V_{\infty }.$ Suppose, for a contradiction, that this is not
the case. As $V(\widehat{C}_{s_{n+1}})\subset V(\widehat{C}_{s_n}),$ then for some
$\varepsilon>0$, there exists
a sequence of points $\widehat{z}_{n}\in \mathrm{Cl}_{\mathbf{T}^{1}\times \mathbb{R}%
}(V(\widehat{C}_{s_n}))$ such that $\widehat{z}_{n}$ does not belong to the $%
\varepsilon $-neighborhood of $V_{\infty }.$ As $\mathbf{T}^{1}\times
]-\infty ,a_{r}-1]\subset $ $V_{\infty },$ the sequence $\widehat{z}_{n}$
must be bounded in $\mathbf{T}^{1}\times \mathbb{R}.$ Therefore, it has a subsequence
converging to a point $\widehat{z}_{\infty }.$ Suppose, for a
contradiction, that $\widehat{z}_{\infty }\notin V_{\infty }$.
Then there exists $n_{*}>0$ and $\delta >0$ such that $\widehat{z}_{\infty
}\notin \mathrm{Cl}_{\mathbf{T}^{1}\times \mathbb{R}}(V(\widehat{C}_{s_{n_{*}}}))$ and 
\begin{equation}
B_{\delta }(\widehat{z}_{\infty })\cap \mathrm{Cl}_{\mathbf{T}^{1}\times \mathbb{R}%
}(V(\widehat{C}_{s_{n_{*}}}))=\emptyset .  \label{intervad}
\end{equation}
On the other hand, for infinitely many integers $n>n_{*},$ the
corresponding point $\widehat{z}_{n}$ belongs to $\mathrm{Cl}_{\mathbf{T}^{1}\times %
\mathbb{R}}(V(\widehat{C}_{s_n}))\cap $ $B_{\delta }(\widehat{z}_{\infty }).$
So $V(\widehat{C}_{s_n})\cap $ $B_{\delta }(\widehat{z}_{\infty })$ is not
empty, a contradiction with assumption (\ref{intervad}), because $\mathrm{Cl}_{%
\mathbf{T}^{1}\times \mathbb{R}}(V(\widehat{C}_{s_n}))\subset \mathrm{Cl}_{\mathbf{T}%
^{1}\times \mathbb{R}}(V(\widehat{C}_{s_{n_{*}}})).$

Thus the point $\widehat{z}_{\infty }$ belongs to $V_{\infty }.$ And as $%
\widehat{z}_{\infty }$ is an accumulation point of the sequence $\widehat{z}%
_{n},$ for infinitely many values of $n,$ $\widehat{z}_{n}$ belongs to the $%
\varepsilon $-neighborhood of $V_{\infty },$ a contradiction with the choice
of the sequence $\widehat{z}_{n}$ that finishes the proof.
\end{proof}

From our assumption that $s_{n}\rightarrow r^{+}$ monotonically from above,
we know that

\[
\widehat{C}_{r}\prec \widehat{C}_{s_{n}},\text{ }\widehat{C}%
_{s_{n}}\rightarrow K_{r}^{+},\text{ and }\lim_{n\to \infty }\mathrm{Leb}_{%
\mathbf{T}^{1}\times \mathbb{R}}\left( [\widehat{C}_{r},\widehat{C}%
_{s_{n}}]\right) =0. 
\]

And as $V_{\infty }\subset \mathcal{U}^{-}(\widehat{C}_{s_n})\cup \widehat{C}%
_{s_n}$ for all $n\geq 0,$ we obtain that $V_{\infty }\subset \mathcal{U}^{-}(%
\widehat{C}_{r})\cup K_{r}^{+},$ where $K_{r}^{+}$ is defined in expression (\ref{defkrmais}).
As $K_{r}^{+}$ is contained in $\widehat{h}^{-1}(r),$ it has zero Lebesgue measure
and therefore no interior. 
So, if $V_{\infty }$ intersects $\mathcal{U}^{+}(\widehat{C}_{r})$, this intersection
has zero Lebesgue measure, it is contained in $K_{r}^{+}$ and consists of a union 
of vertical segments whose lower points are contained in $\widehat{C}_{r}.$ 
 
Now, consider the set 
\[
V_{\infty }^{-}=V_{\infty }\cap \left( \mathcal{U}^{+}(\widehat{C}%
_{r})\right) ^{c} 
\]
and let us fix some $\widehat{z}=(\widehat{x},\widehat{y})\in \partial _{%
\mathbf{T}^{1}\times \mathbb{R}}V_{\infty }^{-}.$ Either:

\begin{itemize}
\item  $\widehat{z}\in \widehat{C}_{r};$

\item  or not;
\end{itemize}

In the second possibility, $\widehat{z}\in \mathcal{U}^{-}(\widehat{C}_{r}).$
So, there exists $\eta >0$ such that $[\widehat{x}-\eta ,\widehat{x}+\eta
]\times [\widehat{y}-\eta ,\widehat{y}+\eta ]\subset \mathcal{U}^{-}(%
\widehat{C}_{r}).$ We already know that $\{\widehat{x}\}\times ]-\infty ,%
\widehat{y}]=V^{-}_{\widehat{z}}\subset V_{\infty }^{-},$ so $\{\widehat{x}\}\times ]-\infty ,%
\widehat{y}+\eta ]$ is also contained in $V_{\infty }^{-}.$ And there are
two possibilities now: either all points in $[\widehat{x}-\eta ,\widehat{x}%
+\eta ]\times [\widehat{y}-\eta ,\widehat{y}+\eta ]\cap V_{\infty }^{-}$ are
contained in the vertical through $\widehat{z},$ or there is some point in 
$$
\bigl[\left( [\widehat{x}-\eta ,\widehat{x}[\cup ]\widehat{x},\widehat{x}+\eta
]\right) \times [\widehat{y}-\eta ,\widehat{y}+\eta ]\bigr]\cap V_{\infty }^{-},
$$
to the right of $\widehat{z}$ for instance. This means that for some $%
\widehat{x}<\widehat{x}^{\prime }\leq \widehat{x}+\eta ,$  the vertical \\ $\{%
\widehat{x}^{\prime }\}\times ]-\infty ,\widehat{y}+\eta ]$ is also
contained in $V_{\infty }^{-}.$ But as $[\widehat{x},\widehat{x}^{\prime
}]\times \{\widehat{y}+\eta \}\subset \mathcal{U}^{-}(\widehat{C}_{r})$ and $%
\widehat{C}_{r}$ is a circloid with no interior, the whole region $]\widehat{x},\widehat{x}%
^{\prime }[\times ]-\infty ,\widehat{y}+\eta ]$ does not intersect $\widehat{%
C}_{r},$ and thus it is contained in $V_{\infty }^{-}.$ In this case there
cannot be a point in $V_{\infty }^{-}\cap $ $[\widehat{x}-\eta ,\widehat{x}[%
\times [\widehat{y}-\eta ,\widehat{y}+\eta ],$ otherwise $\widehat{z}$ would
be an interior point of $V_{\infty }^{-}$ and equivalently, $V^{-}_{\widehat{z}}$
would be contained in $\mathcal{U}^{-}(\widehat{C}_{r}).$ 

Summarizing, let $\widehat{z}$ be a point on the boundary of $V_{\infty }^{-}$
that does not belong to $\widehat{C}_{r}$. For a sufficiently small
rectangular neighborhood $R$ of $\widehat{z}$, the intersection $V_{\infty }^{-}\cap R$ has %
one of the following forms:

\begin{enumerate}
\item  Only the vertical segment containing $\widehat{z}$;

\item  The closure of the left or of the right side of the vertical segment containing  
$\widehat{z}$ in $R;$
\end{enumerate}


Our next task is to consider a connected component $D$ of $\mathcal{U}^{-}(\widehat{C}%
_{r})\setminus V_{\infty }^{-}.$ 
Let us prove that $D$ is an open topological disk in $\mathbf{T}^{1}\times %
\mathbb{R}.$ By~\eqref{caracvinf}, $V_{\infty}^{-}$ is connected and
contains $\mathbf{T}^{1}\times]-\infty,a_r-1]$. Any contractible Jordan
curve in $D$ bounds a disk contained in $\mathcal{U}^{-}(\widehat{C}_r)$.
Since $V_{\infty}^{-}$ is connected, disjoint from that curve, and
contains the lower end, it cannot meet the disk. Thus the disk is
contained in $D$. Consequently, if $D$ is not a disk, then it is
essential. But then, any
essential loop inside $D$ separates $V_{\infty }^{-}$ from $\widehat{C}_r,$
a contradiction. So all connected components of $\mathcal{U}^{-}(\widehat{C}%
_{r})\setminus V_{\infty }^{-}$ are open topological disks, and they are contained
in $\mathbf{T}^{1}\times ]a_{r}-1,+\infty [.$

The boundary of $D$ is contained in $\widehat{C}_{r}\cup \partial _{\mathbf{T%
}^{1}\times \mathbb{R}}V_{\infty }^{-}.$ 
Consider $\partial _{\mathbf{T}^{1}\times \mathbb{R}}D\cap \mathcal{U}^{-}(%
\widehat{C}_{r}):$ from the previous discussion, it is given by a countable
union of disjoint vertical segments of the form $S_{i}:=\{\widehat{x}%
_{i}\}\times ]\widehat{y}_{i,1},\widehat{y}_{i,2}[$ (called elementary
gates), whose endpoints $(\widehat{x}_{i},\widehat{y}_{i,1})$ and $(\widehat{x}%
_{i},\widehat{y}_{i,2})$ belong to $\widehat{C}_{r}.$ The remaining points
in  $\partial _{\mathbf{T}^{1}\times \mathbb{R}}D$ simply belong to $%
\widehat{C}_{r}$. 

\begin{description}
\item[Claim 1.]  
\textit{All elementary gates of }$D$\textit{\ are contained in a single
vertical segment }
$$
\{\widehat{x}_{D}\}\times ]\widehat{y}_{L},\widehat{y}%
_{H}[\subset \partial _{\mathbf{T}^{1}\times \mathbb{R}}D,
$$ 
called
the unique extended gate of $D.$
\end{description}

\begin{proof} 

As the set $V^{-}_\infty$ is different from $V(\widehat{C}_r)$, we decided
to present an original proof for this claim, although probably it could be
proved using similar ideas to the ones used by Fathi and others in the case
of connected components of $\mathcal{U}^{-}(\widehat{C}_{r})\setminus V(%
\widehat{C}_r)$.

\vskip0.1truecm
\textit{Part 1. Existence of elementary gates;}
\vskip0.1truecm
If $D$ had no elementary gates, then $\partial _{\mathbf{T}^{1}\times %
\mathbb{R}}D$ would be contained in $\widehat{C}_{r}.$ As $\left( \widehat{C}%
_{r}\right) ^{c}$ has exactly two connected components, both unbounded, we
obtain a contradiction.
\vskip0.1truecm

Enumerate all the elementary gates in $\partial _{\mathbf{T}^{1}\times %
\mathbb{R}}D$ as $S_{i}:=\{\widehat{x}_{i}\}\times ]\widehat{y}_{i,1},%
\widehat{y}_{i,2}[$ for $i\in I\subset \mathbb{N},$ some countable index
set. For any such gate $S_{i},$ the complement $\mathcal{U}^{-}(\widehat{C}%
_{r})\setminus S_{i}$ has two connected components: one is homeomorphic
to $\mathbf{T}^{1}\times \mathbb{R}$ and the other is a bounded open
topological disk. 

So for each elementary gate $S_{i}$ of $D,$ some point $\widehat{w}%
_{i}\in S_{i}$ and an open ball $B_{i}$ centered at $\widehat{w}_{i}$ which
avoids $\widehat{C}_{r},$ one connected component of $B_{i}\backslash S_{i}$
is contained in the bounded connected component of $\mathcal{U}^{-}(\widehat{%
C}_{r})\setminus S_{i},$ and the other one is contained in the unbounded
connected component. Thus for each elementary gate $S_{i},$ either $D$ is
contained in the bounded connected component of its complement (in this case
we say that $S_{i}$ is an exit), or $D$ is contained in the unbounded
connected component (here we say that $S_{i}$ is an entrance).

\vskip0.1truecm
\textit{Part 2. Existence of exit gates;}
\vskip0.1truecm
Consider a simple arc $\beta \subset $ $\mathcal{U}^{-}(\widehat{C}_{r})$
starting at a point in  $\mathbf{T}^{1}\times ]-\infty,a_{r}-1[$ and ending
at a point in $D.$ As $\beta \subset $ $\mathcal{U}^{-}(\widehat{C}_{r}),$
we can suppose that every time $\beta $ crosses an elementary gate, it
crosses it horizontally. As $\beta $ ends at a point in $D,$ it has to enter 
$D$ for the first time, by crossing an elementary gate $S_{i_{0}}.$ By
definition, this is an exit gate. This means that $D$ is contained in the
bounded connected component of $\mathcal{U}^{-}(\widehat{C}_{r})\setminus
S_{i_{0}}.$ It is easy to see that $D$ cannot have more than one exit.
Assume $S_{i}$  is another elementary gate for $D.$ This means that $D$ is
contained in the unbounded connected component of $\mathcal{U}^{-}(\widehat{C%
}_{r})\setminus S_{i}.$ So, $S_{i}$ is an entrance: All other elementary
gates are entrances.

\vskip0.1truecm


\textit{Part 3. All elementary gates of $D$ are contained in the same vertical and the exit is the lowest;}
\vskip0.1truecm

Denote the only exit gate as $S_{1}=\{\widehat{x}_{1}\}\times ]\widehat{y}%
_{1,1},\widehat{y}_{1,2}[.$ We are going to show that all the other
elementary gates are contained in the vertical that contains $S_{1}.$ And moreover, 
$S_{1}$ is the lowest gate in that vertical.

Suppose, for a contradiction, that there exists another elementary gate of $D$,
denoted $S_{i^{\prime }}=\{\widehat{x}_{i^{\prime }}\}\times ]\widehat{y}%
_{i^{\prime },1},\widehat{y}_{i^{\prime },2}[,$ such that $\widehat{x}%
_{1}\neq \widehat{x}_{i^{\prime }},$ or if they are equal, assume that $%
\widehat{y}_{i^{\prime },2}\leq $ $\widehat{y}_{1,1}.$ Then as the vertical 
\[
\{\widehat{x}_{i^{\prime }}\}\times ]-\infty ,\widehat{y}_{2i^{\prime }}]%
\text{ is contained in }V_{\infty }^{-},\text{ }
\]
it avoids $\mathcal{U}^{+}(\widehat{C}_{r}).$ To see this, recall that $%
\overline{S_{i^{\prime }}}=\{\widehat{x}_{i^{\prime }}\}\times [\widehat{y}%
_{i^{\prime },1},\widehat{y}_{i^{\prime },2}]$ is contained in $\partial _{%
\mathbf{T}^{1}\times \mathbb{R}}V_{\infty }^{-}\subset V_{\infty }^{-}.$ As
for any $\widehat{w}\in V_{\infty }^{-},$ the whole semi-line $V_{\widehat{w}%
}^{-}=\{p_{1}(\widehat{w})\}\times ]-\infty ,p_{2}(\widehat{w})]$ is
contained in $V_{\infty }^{-},$ we obtain the above inclusion. Finally, pick
a simple arc $\alpha \subset \mathcal{U}^{-}(\widehat{C}_{r})$ that connects
a point $\widehat{z}_{i^{\prime }}$ in $S_{i^{\prime }}$ to some point in $%
S_{1},$ crosses $S_{1}$ horizontally (only once), exiting $D$ and continues
inside $\mathcal{U}^{-}(\widehat{C}_{r})$ until it ends at some point of the
form $(\widehat{x}_{i^{\prime }},\widehat{y}^{\prime \prime })\in \mathcal{U}%
^{-}(\widehat{C}_{r}).$ So, the simple closed curve 
\[
\upsilon :=\alpha \cup \{\widehat{x}_{i^{\prime }}\}\times [\widehat{y}%
^{\prime \prime },p_{2}(\widehat{z}_{i^{\prime }})]
\]
is contained in $\mathcal{U}^{-}(\widehat{C}_{r})\cup \widehat{C}_{r},$ and
one of the connected components of $\upsilon ^{c}$ contains an endpoint of $%
S_{1},$ and the other connected component of $\upsilon ^{c}$ contains the
other endpoint of $S_{1}.$ Thus $\mathcal{U}^{+}(\widehat{C}_{r})$ intersects
both connected components of $\upsilon ^{c}$, avoiding $\upsilon .$ This
contradiction ends the argument.

So, for any connected component $D$ of $\mathcal{U}^{-}(\widehat{C}%
_{r})\setminus V_{\infty }^{-},$ there exists $\widehat{x}_{D}\in $ $\mathbf{%
T}^{1}$ such that all elementary gates of $D$ are contained in the vertical $\{%
\widehat{x}_{D}\}\times \mathbb{R}.$ As $V_{\infty }^{-}$ is closed, there
exist numbers 
\[
\widehat{y}_{L}=\inf \{t\in \mathbb{R}:(\widehat{x}_{D},t)\text{ belongs to
an elementary gate of }D\}
\]
\[
\text{ and }
\]
\[
\widehat{y}_{H}=\sup \{t\in \mathbb{R}:(\widehat{x}_{D},t)\text{ belongs to
an elementary gate of }D\},
\]
such that $S_{Total}:=\{\widehat{x}_{D}\}\times ]\widehat{y}_{L},\widehat{y}%
_{H}[\subset V_{\infty }^{-}$ and all elementary gates of $D$ are contained
in $S_{Total}$ by construction. Let us show that $S_{Total}$ is actually
contained in $\partial _{\mathbf{T}^{1}\times \mathbb{R}}V_{\infty }^{-}.$

Suppose, for a contradiction, that for some $\widehat{y}_{L}<\widehat{y}<%
\widehat{y}_{H},$ the point $(\widehat{x}_{D},\widehat{y})$ belongs to the
interior of $V_{\infty }^{-}$. Then $\{\widehat{x}_{D}\}\times ]-\infty ,%
\widehat{y}[\subset \mathcal{U}^{-}(\widehat{C}_{r}),$ a contradiction with
the existence of pieces of elementary gates in $\{\widehat{x}_{D}\}\times %
]-\infty ,\widehat{y}[.$ 

Before the conclusion, as an intermediate step, let us show the following: 

\vskip0.1truecm
\textit{Part 4. D is either to the right of all its elementary gates, or it is to
the left;}
\vskip0.1truecm

For any two points $\widehat{z}_{1},\widehat{z}_{2}\in \{$Elementary gates of $D\},$ 
if we pick $\delta >0$ such that $B_{\delta }(\widehat{z}_{i})$ avoids 
$\widehat{C}_{r}$ for $i=1,2,$ then either the left or the right components
of both  $B_{\delta }(\widehat{z}_{1})\backslash\{\widehat{x}_{D}\}
\times \mathbb{R}$ and $B_{\delta }(\widehat{z}_{2})\backslash\{\widehat{x}_{D}\}\times \mathbb{R}$ 
are contained in $D.$ 

Suppose, for a contradiction, that this is not the case for a choice of $\widehat{z}%
_{1}$ and $\widehat{z}_{2}$ as before. So, the points $\widehat{z}_{1}$ and $%
\widehat{z}_{2}$ belong to $\{\widehat{x}_{D}\}\times \mathbb{R},$ $\widehat{%
z}_{1}$ is below $\widehat{z}_{2},$ and $D$ is to the right of the
elementary gate containing $\widehat{z}_{1}$ and to the left of the
elementary gate containing $\widehat{z}_{2}.$ Pick a simple arc $\beta
\subset D,$ starting horizontally at $\widehat{z}_{1}$ and ending
horizontally at $\widehat{z}_{2}.$ As $\beta \subset D$ and $D$ avoids $%
V_{\infty }^{-},$ $\beta $ avoids the vertical segment  $\{\widehat{x}%
_{D}\}\times [p_{2}(\widehat{z}_{1}),p_{2}(\widehat{z}_{2})]\subset
V_{\infty }^{-}.$ So, 
\[
\beta \cup \{\widehat{x}_{D}\}\times [p_{2}(\widehat{z}_{1}),p_{2}(\widehat{z%
}_{2})]\text{ is a simple closed curve.}
\]
Let $\widehat{B}$ be the lowest point in $\{\widehat{x}%
_{D}\}\times \mathbb{R}\cap \widehat{C}_{r}$ which is above $\widehat{z}_{2},$ and analogously,
let $\widehat{A}$ be the highest point in $\{\widehat{x}%
_{D}\}\times \mathbb{R}\cap \widehat{C}_{r}$ which is below $\widehat{z}_{1}.$ Both $\widehat{A}$
and $\widehat{B}$ avoid $\beta \cup \{\widehat{x}_{D}\}\times [p_{2}(%
\widehat{z}_{1}),p_{2}(\widehat{z}_{2})].$ As the part of $\beta $ close to $%
\{\widehat{x}_{D}\}\times \mathbb{R}$ is horizontal, it is possible to choose
a vertical segment to the right of $\{\widehat{x}_{D}\}\times \mathbb{R}$
sufficiently close to it, that starts at a point in a neighborhood $N_{%
\widehat{A}}$ of $\widehat{A}$ which avoids the curve $\beta \cup \{\widehat{%
x}_{D}\}\times [p_{2}(\widehat{z}_{1}),p_{2}(\widehat{z}_{2})]$ and ends at
a neighborhood $N_{\widehat{B}}$ of $\widehat{B}$ which also avoids the same
curve. This vertical segment crosses $\beta $ only once, in case it is
sufficiently close to $\{\widehat{x}_{D}\}\times \mathbb{R}.$ So $N_{%
\widehat{A}}$ is contained in a connected component of $\left( \beta \cup \{%
\widehat{x}_{D}\}\times [p_{2}(\widehat{z}_{1}),p_{2}(\widehat{z}%
_{2})]\right) ^{c}$ and $N_{\widehat{B}}$ is contained in the other
connected component. As both these neighborhoods contain points in $\mathcal{%
U}^{+}(\widehat{C}_{r})$ and the curve $\beta \cup \{\widehat{x}_{D}\}\times
[p_{2}(\widehat{z}_{1}),p_{2}(\widehat{z}_{2})]$ avoids $\mathcal{U}^{+}(%
\widehat{C}_{r}),$ we arrived at a contradiction with the connectedness of $%
\mathcal{U}^{+}(\widehat{C}_{r}).$ So, either $D$ is to the right of all its
elementary gates, or it is to the left.

\vskip0.1truecm
\textit{Part 5. $S_{Total}$ is contained in $\partial _{\mathbf{T}^{1}\times \mathbb{R}}D;$}
\vskip0.1truecm
 Pick a point $\widehat{w}$ in $S_{Total}.$
Consider points $\widehat{z}_{1},\widehat{z}_{2}$ in $S_{Total}\cap \mathcal{%
U}^{-}(\widehat{C}_{r}),$ $\widehat{z}_{1}$ below $\widehat{w}$ and $%
\widehat{z}_{2}$ above. As they belong to elementary gates, $D$ is either to
the left or to the right of both points. Assume it is to the right. Let $%
\beta $ be a simple arc contained in $D,$ starting horizontally at $\widehat{%
z}_{1}$ and ending horizontally at $\widehat{z}_{2}.$ In the same way as
above, as $\beta \backslash \{\widehat{z}_{1},\widehat{z}_{2}\}\subset D,$
we get that $\beta $ avoids the segments $\{\widehat{x}_{D}\}\times ]p_{2}(%
\widehat{z}_{2}),\widehat{y}_{H}]$ and $\{\widehat{x}_{D}\}\times ]-\infty
,p_{2}(\widehat{z}_{1})[$ because they are both contained in $V_{\infty
}^{-}.$ Let $\theta$ be the simple closed curve given by 
$\beta \cup \{\widehat{x}_{D}\}\times [p_{2}(\widehat{z}_{1}),p_{2}(\widehat{z}_{2})].$
A vertical segment to the right of $\{\widehat{x}_{D}\}\times \mathbb{R}$ with one endpoint 
close to $(\widehat{x}_{D},\widehat{y}_{H})$ and the other endpoint arbitrarily low
crosses $\theta$ twice. So there exists a connected component $E$ 
of the complement of $\theta$ which contains a neighborhood of the point $(%
\widehat{x}_{D},\widehat{y}_{H})\in $ $\widehat{C}_{r}$ and the whole lower
end of the cylinder. As points in  $\widehat{C}_{r}$ are accumulated by
points in $\mathcal{U}^{+}(\widehat{C}_{r}),$ and $\theta$ avoids 
$\mathcal{U}^{+}(\widehat{C}_{r}),$ we get that $E$ contains $\mathcal{U}^{+}(%
\widehat{C}_{r})$ and the lower end of the cylinder. This implies that the
other connected component of the complement of $\theta$ is bounded,
and therefore it is an open topological disk denoted $E^{*}$. As $E^{*}$
avoids $\mathcal{U}^{+}(\widehat{C}_{r}),$ it also avoids $\widehat{C}_{r},$
because points in $\widehat{C}_{r}$ are accumulated by points in $\mathcal{U}%
^{+}(\widehat{C}_{r}).$ So, $E^{*}\subset \mathcal{U}^{-}(\widehat{C}_{r}).$
As points in $E^{*}$ close to $\beta $ are contained in $D,$ we get that $%
E^{*}$ intersects $D.$ Now we show that  $E^{*}\cap \partial _{\mathbf{T}%
^{1}\times \mathbb{R}}D=\emptyset .$ If it is not, then for some $\widehat{J}
$ in $\{\widehat{x}_{D}\}\times \mathbb{R}\cap \partial _{\mathbf{T}^{1}\times 
\mathbb{R}}D,$ either
below $\widehat{z}_{1}$ or above $\widehat{z}_{2},$ $\widehat{J}$ $\in E^{*}.
$ But these points already belong to $E.$ So, $E^{*}\cap \partial _{\mathbf{%
T}^{1}\times \mathbb{R}}D=\emptyset ,$ and finally, since $E^{*}$ is
connected, it is contained in $D.$ As $\{\widehat{x}_{D}\}\times [p_{2}(%
\widehat{z}_{1}),p_{2}(\widehat{z}_{2})]$ is contained in $Cl_{\mathbf{T}%
^{1}\times \mathbb{R}}D\cap V_{\infty }^{-},$ we finally obtain that 
\[
\widehat{w}\in \{\widehat{x}_{D}\}\times [p_{2}(\widehat{z}_{1}),p_{2}(%
\widehat{z}_{2})]\subset \partial _{\mathbf{T}^{1}\times \mathbb{R}}D.
\]
So the whole $S_{Total}$ is contained in $\partial _{\mathbf{T}^{1}\times %
\mathbb{R}}D.$ 
\end{proof}

\begin{remark}
Therefore, the above claim implies that each connected component of the set
$\mathcal{U}^{-}(\widehat{C}_{r})\setminus V_{\infty }^{-}$ is locally to the 
left or to the right of its elementary
gates. This divides the connected components 
 into two classes: right components and left components. So we can
write: 
\[
\mathcal{U}^{-}(\widehat{C}_{r})\setminus V_{\infty }^{-}=\left[\cup _{i\in I_{L}}L_{i}\right]\cup \left[\cup _{i\in
I_{R}}R_{i}\right],
\]
where $I_{L},I_{R}$ are index sets contained in $\mathbb{N}.$ The boundary
of each $L_{i}$ or $R_{i}$ is given by the union of a connected set contained in 
$\widehat{C}_{r}$ with an extended gate denoted $S_{i}^{L}=\{\widehat{x}%
_{i}^{L}\}\times ]\widehat{y}_{i,1}^{L},\widehat{y}_{i,2}^{L}[$ for the components $%
L_i,$ and $S_{i}^{R}=\{\widehat{x}_{i}^{R}\}\times ]\widehat{%
y}_{i,1}^{R},\widehat{y}_{i,2}^{R}[$ for the components $R_i$. Define $R:=$ $%
\cup _{i\in I_{R}}R_{i}$ and $L:=\cup _{i\in I_{L}}L_{i}.$
\end{remark}

We considered the set $V_{\infty }^{-}$ and the connected components of its
complement in $\mathcal{U}^{-}(\widehat{C}_{r})$, because certain connected components of $R(\widehat{C}_{r})$ or
of $L(\widehat{C}_{r})$ might be divided into connected components of the
right and left parts of upper nearby circloids, see for instance Figure 1. 
So, the functions $\widehat{C}\rightarrow V(\widehat{C}),$ $\widehat{C}%
\rightarrow R(\widehat{C})$ and $\widehat{C}\rightarrow L(\widehat{C})$ are
not continuous in the Hausdorff topology.

\begin{figure}[!h]
	\centering
	\includegraphics[scale=0.75]{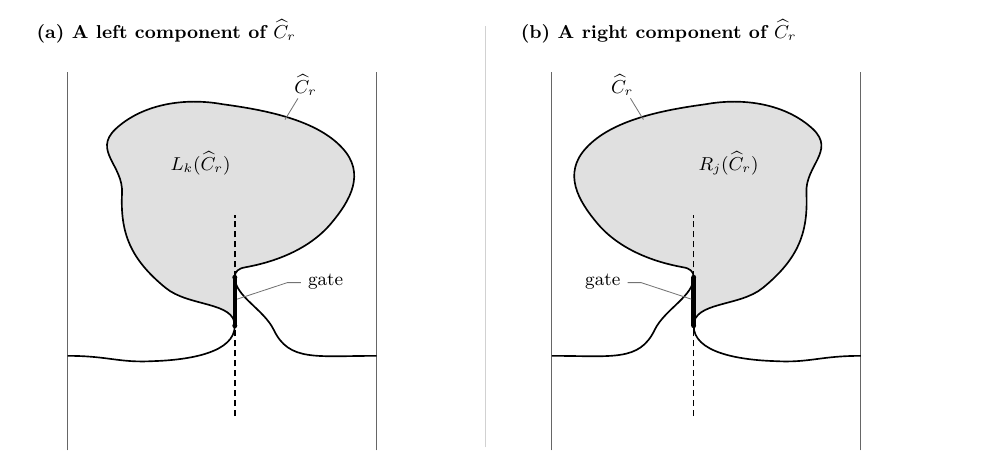}
	\caption{In a) we see a left component of $\widehat{C}_r$ which is divided into left and right components of nearby circloids
above it, and b) shows a similar situation for a right component.}
\end{figure}

\vskip 0.2truecm

Let us explain how considering $%
V_{\infty }^{-}$ instead of $V(\widehat{C}_{r})$ made things easier.
Following our notation, consider $L_{i_{0}}$ as above and the corresponding
extended gate $S_{i_{0}}^{L}=\{\widehat{x}_{i_{0}}^{L}\}\times ]\widehat{y}%
_{i_{0},1}^{L},\widehat{y}_{i_{0},2}^{L}[.$ 
For some $\eta_{0}>0$ and $0<\eta _{1}<\eta _{0},$ $[\widehat{x}_{i_{0}}^{L}-\eta _{1},\widehat{x}%
_{i_{0}}^{L}]\times ]-\infty ,\widehat{y}_{i_{0},1}^{L}+\eta _{0}]\cap
(\partial _{\mathbf{T}^{1}\times \mathbb{R}}L_{i_{0}})$ contains a connected
component $K^{L}_{1i_{0}}$ which intersects both vertical sides of 
$[\widehat{x}_{i_{0}}^{L}-\eta _{1},\widehat{x}_{i_{0}}^{L}]\times ]-\infty ,%
\widehat{y}_{i_{0},1}^{L}+\eta _{0}].$ This follows from the fact that 
$[\widehat{x}_{i_{0}}^{L}-\eta _{1},\widehat{x}_{i_{0}}^{L}[\times 
\{\widehat{y}_{i_{0},1}^{L}+\eta _{0}\}$ is contained in $L_{i_{0}}$ and 
$[\widehat{x}_{i_{0}}^{L}-\eta _{1},\widehat{x}_{i_{0}}^{L}[\times \{t\}$ 
is contained in $\left(\mathrm{Cl}_{\mathbf{T}^1\times\mathbb{R}}L_{i_{0}}\right)^c$
for some $t<\widehat{y}_{i_{0},1}^{L}.$

The next thing we want to understand is what happens near $(\widehat{x}%
_{i_{0}}^{L},\widehat{y}_{i_{0},2}^{L})$ and above. For this, 
define the following number: 
\begin{equation}
\widehat{y}_{i_{0}}^{L\prime }=\min \{t\geq \widehat{y}_{i_{0},2}^{L}:(%
\widehat{x}_{i_{0}}^{L},t)\text{ is accumulated by points in }\{\widehat{x}%
_{i_{0}}^{L}\}\times [t,+\infty [\cap \left( \overline{L_{i_{0}}}\right)
^{c}\},  \label{defylin}
\end{equation}
where $\overline{L_{i_{0}}}:=\mathrm{Cl}_{\mathbf{T}^1\times\mathbb{R}}L_{i_{0}}$. 
The compactness of $\overline{L_{i_{0}}}$ implies that $(\widehat{x}%
_{i_{0}}^{L},\widehat{y}_{i_{0}}^{L\prime })\in \partial _{\mathbf{T}%
^{1}\times \mathbb{R}}L_{i_{0}}\backslash S_{i_{0}}^{L}.$ And the vertical
segment $\{\widehat{x}_{i_{0}}^{L}\}\times [\widehat{y}_{i_{0},2}^{L},%
\widehat{y}_{i_{0}}^{L\prime }]$ is, by definition, contained in $\overline{%
L_{i_{0}}}.$

\begin{description}
\item[Claim 2.]  For any $\kappa >0,$ there exist $\mu >0$ and a horizontal
segment contained in $(\overline{L_{i_{0}}})^{c}$ intersecting both vertical sides of 
$[\widehat{x}_{i_{0}}^{L}-\mu ,\widehat{x}_{i_{0}}^{L}+\mu
]\times ]\widehat{y}_{i_{0}}^{L\prime },\widehat{y}_{i_{0}}^{L\prime
}+\kappa [.$
\end{description}

\begin{proof}

From the definition of $\widehat{y}_{i_{0}}^{L\prime },$ there exists a
sequence $t_{n}\rightarrow (\widehat{y}_{i_{0}}^{L\prime })^{+}$ such that $(%
\widehat{x}_{i_{0}}^{L},t_{n})\notin \overline{L_{i_{0}}}.$ Given $\kappa
>0, $ pick $\widehat{y}_{i_{0}}^{L\prime }<t_{n}<\widehat{y}%
_{i_{0}}^{L\prime }+\kappa .$ Then, as $(\widehat{x}_{i_{0}}^{L},t_{n})%
\notin \overline{L_{i_{0}}},$ there exists an open ball of radius $r_{n}>0$
centered at $(\widehat{x}_{i_{0}}^{L},t_{n})$ which avoids $\overline{%
L_{i_{0}}}.$ Finally, pick $\mu =r_{n}/2.$
\end{proof}

Now we consider two cases:

\begin{enumerate}
\item  Suppose that $\widehat{y}_{i_{0}}^{L\prime }=\widehat{y}_{i_{0},2}^{L}$
or $\widehat{y}_{i_{0}}^{L\prime }>\widehat{y}_{i_{0},2}^{L}$ and $\{\widehat{%
x}_{i_{0}}^{L}\}\times [\widehat{y}_{i_{0},2}^{L},\widehat{y}%
_{i_{0}}^{L\prime }]\subset \partial _{\mathbf{T}^{1}\times \mathbb{R}%
}L_{i_{0}}\backslash S_{i_{0}}.$

In this case, for any $\widehat{y}_{i_{0},1}^{L}<t<\widehat{y}_{i_{0},2}^{L},$
there exists $\mu =\mu (t)>0$ such that the segment $[\widehat{x}%
_{i_{0}}^{L}-\mu ,\widehat{x}_{i_{0}}^{L}[\times \{t\}$ is contained in $%
L_{i_{0}}.$ For any $\kappa >0,$ decreasing $\mu >0$ from the above claim if
necessary, we get that for some $\widehat{y}_{i_{0}}^{L\prime }<\widehat{y}<%
\widehat{y}_{i_{0}}^{L\prime }+\kappa ,$ the segment $[\widehat{x}%
_{i_{0}}^{L}-\mu ,\widehat{x}_{i_{0}}^{L}+\mu ]\times \{\widehat{y}\}$ is
contained in $\left( \overline{L_{i_{0}}}\right) ^{c}.$ But this implies
that there exists a connected component $K_{2i_{0}}^{\text{Case 1,L}}$ of $%
\left( \partial _{\mathbf{T}^{1}\times \mathbb{R}}L_{i_{0}}\backslash
S_{i_{0}}^{L}\right) \cap $ $[\widehat{x}_{i_{0}}^{L}-\mu ,\widehat{x}%
_{i_{0}}^{L}]\times ]t,\widehat{y}[$ which intersects both vertical sides
of $[\widehat{x}_{i_{0}}^{L}-\mu ,\widehat{x}_{i_{0}}^{L}]\times %
]t,\widehat{y}[.$

\item  Here we treat the remaining case, that is, we assume that $\widehat{y}%
_{i_{0},2}^{L}<\widehat{y}_{i_{0}}^{L\prime }$ and some point in $\{\widehat{x%
}_{i_{0}}^{L}\}\times ]\widehat{y}_{i_{0},2}^{L},\widehat{y}_{i_{0}}^{L\prime
}[$ belongs to $L_{i_{0}}.$

From the definition of $\widehat{y}_{i_{0}}^{L\prime },$ the segment $\{%
\widehat{x}_{i_{0}}^{L}\}\times [\widehat{y}_{i_{0},1}^{L},\widehat{y}%
_{i_{0}}^{L\prime }]$ is contained in $\overline{L_{i_{0}}}.$ Therefore, it is
contained in $\mathcal{U}^{-}(\widehat{C}_{r})\cup \widehat{C}_{r}.$ As $(%
\widehat{x}_{i_{0}}^{L},\widehat{y}_{i_{0},1}^{L})\in \partial _{\mathbf{T}%
^{1}\times \mathbb{R}}V_{\infty }^{-},$ we get that $\{\widehat{x}%
_{i_{0}}^{L}\}\times [\widehat{y}_{i_{0},1}^{L},\widehat{y}_{i_{0}}^{L\prime
}]\subset \partial _{\mathbf{T}^{1}\times \mathbb{R}}V_{\infty }^{-}.$ So,
it is not possible that some point in $L_{i_{0}}$ belongs to $\{\widehat{x}%
_{i_{0}}^{L}\}\times ]\widehat{y}_{i_{0},2}^{L},\widehat{y}_{i_{0}}^{L\prime
}[,$ because $\{\widehat{x}_{i_{0}}^{L}\}\times [\widehat{y}_{i_{0},1}^{L},%
\widehat{y}_{i_{0}}^{L\prime }]$ is connected and contained in $\partial _{%
\mathbf{T}^{1}\times \mathbb{R}}V_{\infty }^{-}$ and $L_{i_{0}}$ is a
connected component of $\mathcal{U}^{-}(\widehat{C}_{r})\setminus V_{\infty }^{-}.$ Therefore, this
case does not appear when studying the connected components of 
$\mathcal{U}^{-}(\widehat{C}_{r})\setminus V_{\infty }^{-}.$

This is the reason why, instead of considering the complement of $\mathrm{Cl}_{%
\mathbf{T}^{1}\times \mathbb{R}}V(\widehat{C}_{r})$ in $\mathcal{U}^{-}(\widehat{C}_{r}),$ 
we studied the
complement of $V_{\infty }^{-}$ in $\mathcal{U}^{-}(\widehat{C}_{r}).$ 
When considering the first set, 
this second case could appear
and it corresponds to circloids $\widehat{C}_{r}$ for which the function $%
\widehat{C}\rightarrow L(\widehat{C})$ is not continuous in the Hausdorff
topology at $\widehat{C}=\widehat{C}_{r}.$ The sequence $\widehat{C}_{s_n}$ converges to $K_{r}^{+}$.
For sufficiently large $n$, the circloid $\widehat{C}_{s_n}$, which lies above $\widehat{C}_{r}$, %
divides a connected component of $L(\widehat{C}_{r})$ into
connected components of $L(\widehat{C}_{s_n})$ and of $R(\widehat{C}_{s_n}).$
\end{enumerate}

Finally, given $\varepsilon >0,$ pick $n_{R},n_{L}\geq 0$ such that 
\begin{equation}
\left| \mathrm{Leb}_{\mathbf{T}^{1}\times \mathbb{R}}(R)-\left( \Sigma
_{i=1}^{n_{R}}\mathrm{Leb}_{\mathbf{T}^{1}\times \mathbb{R}}(R_{i})\right)
\right| <\varepsilon /4  \label{finiteR}
\end{equation}
\[
\text{and} 
\]
\begin{equation}
\left| \mathrm{Leb}_{\mathbf{T}^{1}\times \mathbb{R}}(L)-\left( \Sigma
_{i=1}^{n_{L}}\mathrm{Leb}_{\mathbf{T}^{1}\times \mathbb{R}}(L_{i})\right)
\right| <\varepsilon /4.  \label{finiteL}
\end{equation}
For each $1\leq i\leq n_{R},$ first choose an open Jordan domain
$R_{i}^{*}$ whose closure is contained in $R_{i}$ and such that
\[
\mathrm{Leb}_{\mathbf{T}^{1}\times \mathbb{R}}(R_{i})-
\mathrm{Leb}_{\mathbf{T}^{1}\times \mathbb{R}}(R_{i}^{*})
<\varepsilon /(10n_{R}).
\]
This is possible because $R_{i}$ is an open topological disk of finite
area and can be exhausted by Jordan domains compactly contained in it.
In particular,
\[
d_{i}^{R}:=\operatorname{dist}\left(\overline{R_{i}^{*}},
\partial _{\mathbf{T}^{1}\times \mathbb{R}}R_{i}\right)>0.
\]
Choose a point $a_{i}^{R}=(\widehat{x}_{i}^{R},t_{i}^{R})$ in an
elementary gate contained in $S_{i}^{R}$ and an open simple arc
$\alpha_{i}^{R}\subset R_{i}$ joining $a_{i}^{R}$ to a point of
$R_{i}^{*}.$ We choose this arc to start with a horizontal segment to
the right of $a_{i}^{R};$ the closed remaining part of the arc is a
compact subset of $R_{i}.$

For sufficiently small $\beta>0,$ let
\[
J_{i,\beta}^{R}:=\{\widehat{x}_{i}^{R}+\beta\}\times
]\rho_{i,1}^{R}(\beta),\rho_{i,2}^{R}(\beta)[
\]
be the connected component of
$R_{i}\cap(\{\widehat{x}_{i}^{R}+\beta\}\times\mathbb{R})$
containing $(\widehat{x}_{i}^{R}+\beta,t_{i}^{R}).$ This is a vertical
crosscut of $R_{i},$ with both endpoints in
$\partial _{\mathbf{T}^{1}\times \mathbb{R}}R_{i}.$ We claim that
$J_{i,\beta}^{R}$ is contained in any prescribed neighborhood of
$\partial _{\mathbf{T}^{1}\times \mathbb{R}}R_{i}$ when $\beta$ is
sufficiently small.

Indeed, define $\widehat{y}_{i}^{R\prime}$ as in
(\ref{defylin}), with $R_{i}$ in place of $L_{i_{0}}.$ By Claim 2,
for every $\kappa>0$ and all sufficiently small $\beta>0,$ the upper
endpoint of $J_{i,\beta}^{R}$ lies below
$\widehat{y}_{i}^{R\prime}+\kappa.$ Consequently, any accumulation
point of these crosscuts as $\beta\rightarrow0$ belongs to
\[
\overline{R_{i}}\cap
\left(\{\widehat{x}_{i}^{R}\}\times
]-\infty,\widehat{y}_{i}^{R\prime}]\right).
\]
Below $\widehat{y}_{i,2}^{R},$ this vertical lies in
$V_{\infty}^{-}$ and hence avoids $R_{i}.$ The remaining segment,
from $\widehat{y}_{i,2}^{R}$ to $\widehat{y}_{i}^{R\prime},$ belongs
to $\partial _{\mathbf{T}^{1}\times \mathbb{R}}R_{i},$ by the
exclusion of Case 2. Thus every such accumulation point belongs to
$\partial _{\mathbf{T}^{1}\times \mathbb{R}}R_{i},$ and the claim
follows from the compactness of $\overline{R_{i}}.$

We can therefore choose $0<\beta_{i}^{R}<d_{i}^{R}/2$ such that,
for every $0<\beta\leq\beta_{i}^{R},$ the crosscut $J_{i,\beta}^{R}$
avoids both $\overline{R_{i}^{*}}$ and the compact remaining part of
$\alpha_{i}^{R},$ and meets its initial horizontal segment exactly
once. It follows that $R_{i}^{*}$ is contained in the component of
$R_{i}\setminus J_{i,\beta}^{R}$ on the side away from $S_{i}^{R}.$
Indeed, the arc $\alpha_{i}^R$ crosses $J_{i,\beta}^{R}$ once when leaving the gate
side and then reaches $R_{i}^{*}$ without crossing it again; the
connectedness of $R_{i}^{*}$ gives the assertion for the whole
Jordan domain. Moreover,
\[
\operatorname{dist}\left(R_{i}^{*},
\partial _{\mathbf{T}^{1}\times \mathbb{R}}R_{i}\right)>2\beta.
\]
Thus the component cut off away from the gate contains the fixed
Jordan domain $R_{i}^{*}$ and accounts for all but at most 
$\varepsilon/(10n_{R})$ of the area of $R_{i},$ see Figure 2.

\begin{figure}[!h]
	\centering
	\includegraphics[scale=0.75]{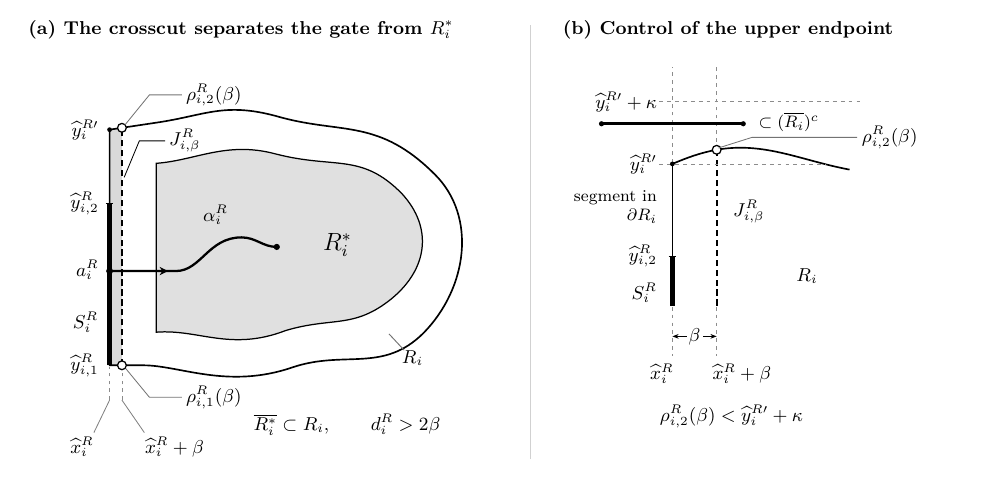}
	\caption{Description of the set $R_{i}^*$.}
\end{figure}

\vskip 0.2truecm

Analogously, for each $1\leq i\leq n_{L},$ first choose an open
Jordan domain $L_{i}^{*}$ with $\overline{L_{i}^{*}}\subset L_{i}$
and
\[
\mathrm{Leb}_{\mathbf{T}^{1}\times \mathbb{R}}(L_{i})-
\mathrm{Leb}_{\mathbf{T}^{1}\times \mathbb{R}}(L_{i}^{*})
<\varepsilon/(10n_{L}).
\]
Choose an open simple arc $\alpha_{i}^{L}\subset L_{i}$ with one
endpoint in an elementary gate contained in $S_{i}^{L}$ and the
other in $L_{i}^{*},$ starting horizontally to the left. The same
argument gives $\beta_{i}^{L}>0$ such that, for every
$0<\beta\leq\beta_{i}^{L},$ the corresponding vertical crosscut
\[
J_{i,\beta}^{L}:=\{\widehat{x}_{i}^{L}-\beta\}\times
]\rho_{i,1}^{L}(\beta),\rho_{i,2}^{L}(\beta)[
\]
has its endpoints in $\partial _{\mathbf{T}^{1}\times \mathbb{R}}L_{i},$
avoids $\overline{L_{i}^{*}},$ and meets $\alpha_{i}^{L}$ exactly
once, on its initial horizontal segment. Hence $L_{i}^{*}$ lies in
the component of $L_{i}\setminus J_{i,\beta}^{L}$ on the side away
from $S_{i}^{L},$ and we may also require
\[
\operatorname{dist}\left(L_{i}^{*},
\partial _{\mathbf{T}^{1}\times \mathbb{R}}L_{i}\right)>2\beta.
\]

%

So, pick 
$$
0<\beta <\min \{\min\limits_{1\leq i\leq n_{L}}\{\beta
_{i}^{L}\},\min\limits_{1\leq i\leq n_{R}}\{\beta _{i}^{R}\}\}
$$
 and $n_{0}>0$
large enough that if $n\geq n_{0},$ then 
\[
V(\widehat{C}_{s_n})\subset \beta /4\text{-neighborhood of }V_{\infty }. 
\]
This is possible because $V(\widehat{C}_{s_n})$ converges in the Hausdorff
topology to $V_{\infty }.$ 
From the choice of $n_0$ above, $V(\widehat{C}_{s_n})$ does not intersect any of the $R_{i}^{*}$
and $L_{i}^{*}.$ Therefore they are still in the complement of $\mathrm{Cl}_{\mathbf{T%
}^{1}\times \mathbb{R}}V(\widehat{C}_{s_n})$ in $\mathcal{U}^{-}(\widehat{C}_{s_n}).$ 
And all $S_{i}^{L}$ and $%
S_{i}^{R}$ are contained in $V(\widehat{C}_{s_n}).$ So, $V(\widehat{C}_{s_n})$ intersects
the arcs $\alpha^R_i$ and $\alpha^L_i$, which are contained in $\mathcal{U}^{-}(\widehat{C}_{s_n})$.
And this finally implies that
\[
\mathrm{Leb}_{\mathbf{T}^{1}\times \mathbb{R}}(R(\widehat{C}_{s_n}))>\Sigma
_{i=1}^{n_{R}}\mathrm{Leb}_{\mathbf{T}^{1}\times \mathbb{R}}(R_{i}^{*})\text{
and }\mathrm{Leb}_{\mathbf{T}^{1}\times \mathbb{R}}(L(\widehat{C}%
_{s_{n}}))>\Sigma _{i=1}^{n_{L}}\mathrm{Leb}_{\mathbf{T}^{1}\times \mathbb{R}%
}(L_{i}^{*}). 
\]

And we finally obtain that for all $n\geq n_0$, 
\begin{equation}
\mathrm{Leb}_{\mathbf{T}^{1}\times \mathbb{R}}(R(\widehat{C}_{s_n}))\geq 
\mathrm{Leb}_{\mathbf{T}^{1}\times \mathbb{R}}(R)-\varepsilon \text{ and }%
\mathrm{Leb}_{\mathbf{T}^{1}\times \mathbb{R}}(L(\widehat{C}_{s_n}))\geq 
\mathrm{Leb}_{\mathbf{T}^{1}\times \mathbb{R}}(L)-\varepsilon ,
\label{1e2juntas}
\end{equation}
where $L$ is the union of all left components of $\mathcal{U}^{-}(\widehat{C}_{r})\setminus V^{-}_{\infty },$ and $R$ is the union of all right components.

Comparing areas, we obtain that:
\[
\mathrm{Leb}_{\mathbf{T}^{1}\times \mathbb{R}}\left( [\widehat{C}_{r},%
\widehat{C}_{s_{n}}]\right) +\mathrm{Leb}_{\mathbf{T}^{1}\times \mathbb{R}%
}\left( R\right) +\mathrm{Leb}_{\mathbf{T}^{1}\times \mathbb{R}}\left(
L\right) +\mathrm{Leb}_{\mathbf{T}^{1}\times \mathbb{R}}\left( V_{\infty
}^{-}\cap \mathbf{T}^{1}\times [a_{r}-1,+\infty [\right) = 
\]

\[
=\mathrm{Leb}_{\mathbf{T}^{1}\times \mathbb{R}}\left( V(\widehat{C}_{s_n})\cap 
\mathbf{T}^{1}\times [a_{r}-1,+\infty [\right) +\mathrm{Leb}_{\mathbf{T}%
^{1}\times \mathbb{R}}(R(\widehat{C}_{s_n}))+\mathrm{Leb}_{\mathbf{T}%
^{1}\times \mathbb{R}}(L(\widehat{C}_{s_n})). 
\]

As $V(\widehat{C}_{s_n})\supset V_{\infty }^{-},$ we obtain that 
$$\mathrm{Leb}%
_{\mathbf{T}^{1}\times \mathbb{R}}\left( V(\widehat{C}_{s_n})\cap \mathbf{T}%
^{1}\times [a_{r}-1,+\infty [\right) \geq \mathrm{Leb}_{\mathbf{T}^{1}\times %
\mathbb{R}}\left( V_{\infty }^{-}\cap \mathbf{T}^{1}\times [a_{r}-1,+\infty
[\right).
$$ 

And the fact that $\mathrm{Leb}_{\mathbf{T}^{1}\times \mathbb{R}}\left( [\widehat{C}_{r},%
\widehat{C}_{s_{n}}]\right) \rightarrow 0$ as $n\rightarrow \infty$ 
finally implies that:
$$
\limsup_{n \to \infty} \left[\mathrm{Leb}_{\mathbf{T}^{1}\times \mathbb{R}}(R(\widehat{C}_{s_n}))+\mathrm{%
Leb}_{\mathbf{T}^{1}\times \mathbb{R}}(L(\widehat{C}_{s_n}))\right]\leq 
$$
$$\leq \mathrm{Leb}_{%
\mathbf{T}^{1}\times \mathbb{R}}\left( R\right) +\mathrm{Leb}_{\mathbf{T}%
^{1}\times \mathbb{R}}\left( L\right) . 
$$
As $\varepsilon >0$ is arbitrary, expression (\ref{1e2juntas}) concludes the
proof of the proposition: 

\vskip0.1truecm
$\lim\limits_{n\rightarrow \infty }\mathrm{Leb}_{\mathbf{T}^{1}\times \mathbb{R}}(R(%
\widehat{C}_{s_n}))=\mathrm{Leb}_{\mathbf{T}^{1}\times \mathbb{R}}\left(
R\right)$ and $\lim\limits_{n\rightarrow \infty }\mathrm{Leb}_{\mathbf{T}%
^{1}\times \mathbb{R}}(L(\widehat{C}_{s_n}))=\mathrm{Leb}_{\mathbf{T}%
^{1}\times \mathbb{R}}\left( L\right)$ 
\end{proof}

So, the previous proposition implies that 
$\lambda _{L}=\mathrm{Leb}_{\mathbf{T}^{1}\times \mathbb{R}}\left(L\right)$ and
 $\lambda _{R}=\mathrm{Leb}_{\mathbf{T}^{1}\times \mathbb{R}}\left(R\right).$  

Thus,  
\[
\begin{array}{c}
\lim\limits_{i\to \infty }\mathrm{Leb}_{\mathbf{T}^{1}\times \mathbb{R}}(R(%
\widehat{C}_{i}^{+}))= \lim\limits_{i\to \infty }\mathrm{Leb}_{\mathbf{T}^{1}\times \mathbb{R}}(R(%
\widehat{C}_{i}^{-}))= \lambda _{R}\\ 
\text{and} \\ 
\lim\limits_{i\to \infty }\mathrm{Leb}_{\mathbf{T}^{1}\times \mathbb{R}}(L(%
\widehat{C}_{i}^{+}))= \lim\limits_{i\to \infty }\mathrm{Leb}_{\mathbf{T}^{1}\times \mathbb{R}}(L(%
\widehat{C}_{i}^{-}))=\lambda _{L}. \\ 
\end{array}
\]

To finish the proof that $R(\widehat{C}_r)$ is empty, let us analyze two
cases:

\begin{enumerate}
\item  $\lambda _{R}\leq \mathrm{Leb}_{\mathbf{T}^{1}\times \mathbb{R}}(R(\widehat{C}_{r}));$

In this case, we use Lemma \ref{lemTC00}, which says that $\mathrm{Leb}_{%
\mathbf{T}^{1}\times \mathbb{R}}(R(\widehat{f}^{n}(\widehat{C}_{r})))>%
\mathrm{Leb}_{\mathbf{T}^{1}\times \mathbb{R}}(R(\widehat{C}_{r}))+\delta
_{1},$ for all integers $n>0$ and some constant $\delta _{1}>0.$ As $%
\widehat{C}_{i}^{+}:=\widehat{f}^{n_{i}^{+}}(\widehat{C}_{r})+(0,p_{i}^{+}),$
we get that 
\[
\mathrm{Leb}_{\mathbf{T}^{1}\times \mathbb{R}}(R(\widehat{C}_{i}^{+}))=%
\mathrm{Leb}_{\mathbf{T}^{1}\times \mathbb{R}}(R(\widehat{f}^{n_{i}^{+}}(%
\widehat{C}_{r})))>
\]
\[
>\mathrm{Leb}_{\mathbf{T}^{1}\times \mathbb{R}}(R(\widehat{C}_{r}))+\delta
_{1}.
\]

Taking the limit $i\rightarrow \infty ,$ we obtain that $\lambda _{R}\geq 
\mathrm{Leb}_{\mathbf{T}^{1}\times \mathbb{R}}(R(\widehat{C}_{r}))+\delta
_{1},$ a contradiction with our assumption.

\item  $\lambda _{R}>\mathrm{Leb}_{\mathbf{T}^{1}\times \mathbb{R}}(R(\widehat{C}_{r}));$

By Proposition \ref{areaaa}, $\lim_{i\to \infty }\mathrm{Leb}_{\mathbf{T}^{1}\times \mathbb{R}}(R(\widehat{C}%
_{i}^{-}))=\lambda _{R}.$ Now we use the fact that for all integers $n>0$
and for any circloid $\widehat{C}\subset \mathbf{T}^{1}\times \mathbb{R},$ $%
R(\widehat{f}^{n}(\widehat{C}))\supset \widehat{f}^{n}(R(\widehat{C})).$ So, 
$R(\widehat{C}_{r})=R(\widehat{f}^{n_{i}^{-}}(\widehat{C}%
_{i}^{-}))-(0,p_{i}^{-})\supset \widehat{f}^{n_{i}^{-}}(R(\widehat{C}%
_{i}^{-}))-(0,p_{i}^{-}).$ Computing the Lebesgue measure on both sides, and
taking the limit $i\rightarrow \infty $ gives:\ $\mathrm{Leb}_{\mathbf{T}%
^{1}\times \mathbb{R}}(R(\widehat{C}_{r}))\geq \mathrm{Leb}_{\mathbf{T}%
^{1}\times \mathbb{R}}(\widehat{f}^{n_{i}^{-}}(R(\widehat{C}_{i}^{-})))=%
\mathrm{Leb}_{\mathbf{T}^{1}\times \mathbb{R}}(R(\widehat{C}%
_{i}^{-}))\rightarrow \lambda _{R},$ again a contradiction with our
assumption.
\end{enumerate}

So $R(\widehat{C}_r)$ must be empty. In a similar way, we show that $L( 
\widehat{C}_r)$ must also be empty.
\end{proof}

In this way each circloid is almost a graph.

\paragraph{Conclusion of the Proof of Theorem A}

We have established that $R(\widehat{C}_{r})=L(\widehat{C}_{r})=\emptyset$ for every circloid $\widehat{C}_{r}$
in the family constructed using bounded deviations.
In the next lemma, we conclude the proof of existence of the foliation. 
In particular, we show that 
$\widehat{C}_{r}=\mathrm{Graph}(\widehat{\phi }_{r})$, for some Lipschitz
continuous function $\widehat{\phi }_{r}:\mathbf{T}^{1}\to \mathbb{R}$
satisfying the following property 
\[
\widehat{\phi }_{r+1}(x)=\widehat{\phi }_{r}(x)+1\quad \text{for all}\
x\in \mathbf{T}^{1}. 
\]
The fact that $R(\widehat{C}_{r})=L(\widehat{C}_{r})=\emptyset $ implies
that 
\begin{equation}
\partial _{\mathbf{T}^{1}\times \mathbb{R}}\mathcal{U}^{-}(\widehat{C}%
_{r})=\partial _{\mathbf{T}^{1}\times \mathbb{R}}V(\widehat{C}_{r})=\widehat{%
C}_{r}.  \label{eq.TC19}
\end{equation}
This places us very close to proving that circloids are graphs.
Nevertheless, it is still conceivable that $\widehat{C}_{r}$ contains
vertical segments. 

\begin{lemma}
\label{lemTC07} For all $r\in \mathbb{R},$ each circloid $\widehat{C}_{r}$
contained in $\widehat{h}^{-1}(r),$ where $\widehat{h}:\mathbf{T}^{1}\times %
\mathbb{R}\rightarrow \mathbb{R}$ is the continuous surjective map that
satisfies $\widehat{h}(\widehat{z}+(0,1))=\widehat{h}(\widehat{z})+1$ and $%
\widehat{h}\circ \widehat{f}(\widehat{z})=\widehat{h}(\widehat{z})+\alpha ,$
is given by the $\mathrm{Graph}$ of a Lipschitz continuous function $%
\widehat{\phi }_{r}:\mathbf{T}^{1}\to \mathbb{R}$ satisfying $\widehat{\phi }%
_{r+1}(x)=\widehat{\phi }_{r}(x)+1.$ The Lipschitz constant can be chosen
uniformly for all $r,$ and $r\rightarrow \mathrm{Graph}(\widehat{\phi }_{r})$
is a foliation of $\mathbf{T}^{1}\times \mathbb{R}$ which clearly induces a
foliation of $\mathbf{T}^{2},$ because of its equivariance. Finally, the foliation is 
unique.
\end{lemma}

\begin{proof}
Part of the proof of this lemma is quite standard; it is contained in the proof of
Birkhoff's Invariant Curve Theorem, plus an application of Arzel\'a-Ascoli's
Theorem. Let us divide the argument into 5 parts:

\vskip0.3truecm

\textit{Part 1. }$\widehat{C}_{r}$ is the graph of a continuous function;

\vskip0.1truecm

If $p_{1}\mid _{\widehat{C}_{r}}$ is not injective, then for some $\widehat{x%
}\in \mathbf{T}^{1},$ there exist real numbers $t_{1}<t_{2}$ such that $\{%
\widehat{x}\}\times [t_{1},t_{2}]$ is contained in $\widehat{C}_{r},$ and
for $\widehat{t}=(t_{1}+t_{2})/2$ and some $0<\varepsilon<(t_2-t_1)/4,$ 
either on the left or right part of $B_{\varepsilon }((%
\widehat{x},\widehat{t}))\backslash \{\widehat{x}\}\times [t_{1},t_{2}],$
there exists a point $\widehat{z}_{out}$ which belongs to $\mathcal{U}^{+}(\widehat{C}_{r}).$ 
Depending on whether $\widehat{z}_{out}$ is to the left or to the right of the vertical
$\{\widehat{x}\}\times [t_{1},t_{2}],$ we get from the twist
condition that either $\widehat{f}^{-1}(\widehat{z}_{out})\subset 
\mathcal{U}^{-}(\widehat{f}^{-1}(\widehat{C}_{r}))$ or $\widehat{f}(\widehat{z}_{out}) \subset
 \mathcal{U}^{-}(\widehat{f}(\widehat{C}_{r})),$ 
because $\widehat{f}^{\pm 1}(\widehat{C}_{r})=\widehat{C}_{r\pm\alpha }$ also
satisfy $R(\widehat{C}_{r\pm \alpha })=L(\widehat{C}_{r\pm \alpha })=\emptyset $
and we are assuming that $\widehat{f}$ is a right twist map. But this is a
contradiction, since for $\widehat{z}_{out}\in \mathcal{U}^{+}(\widehat{C}_{r}),$
$\widehat{f}^{\pm 1}(\widehat{z}_{out})$ belongs to $\mathcal{U}^{+}(\widehat{f}^{\pm 1}(%
\widehat{C}_{r})).$ So, for all $r\in \mathbb{R},$ each circloid $\widehat{C}%
_{r}$ is given by the $\mathrm{Graph}$ of a continuous function $\widehat{%
\phi }_{r}:\mathbf{T}^{1}\to \mathbb{R}$ satisfying $\widehat{\phi }_{r+1}(%
\widehat{x})=\widehat{\phi }_{r}(\widehat{x})+1$ for all $\widehat{x}\in $ $%
\mathbf{T}^{1}$ because $\widehat{C}_{r+1}=\widehat{C}_{r}+(0,1).$

\vskip0.2truecm

\textit{Part 2. }There exists a constant $K:=K(f)>0$ such that all $\widehat{%
\phi }_{r}$ as above are $K$-Lipschitz;

\vskip0.1truecm

Fix a lift $\widetilde{f}\in \mathrm{Diff}_{\mathrm{tw}}^{1}(\mathbb{R}%
^{2})$ of $\widehat{f}$. It can be written as follows 
\[
\widetilde{f}(\widetilde{x},\widetilde{y})=(\widetilde{x}+\ell \widetilde{y}%
+\varphi _{1}(\widetilde{x},\widetilde{y}),\widetilde{y}+\varphi _{2}(%
\widetilde{x},\widetilde{y})),
\]
for all $(\widetilde{x},\widetilde{y})\in \mathbb{R}^{2}$, where $\ell \in %
\mathbb{Z}\setminus \{0\}$, and the functions $\widetilde{\varphi }_{i}:%
\mathbb{R}^{2}\to \mathbb{R}$, $i=1,2$ are $\mathbb{Z}^{2}$-periodic.
Similarly, 
\[
\widetilde{f}^{-1}(\widetilde{x},\widetilde{y})=(\widetilde{x}-\ell 
\widetilde{y}+\varphi _{1}^{*}(\widetilde{x},\widetilde{y}),\widetilde{y}%
+\varphi _{2}^{*}(\widetilde{x},\widetilde{y})),\quad \forall \ (\widetilde{x%
},\widetilde{y})\in \mathbb{R}^{2},
\]
where the functions $\widetilde{\varphi }_{i}^{*}:\mathbb{R}^{2}\to %
\mathbb{R}$ are also $\mathbb{Z}^{2}$-periodic for $i=1,2$.

As we assumed throughout this paper, suppose $f$ is a right twist map,
that is $\ell >0$ and for all $\widetilde{z}=(\widetilde{x},\widetilde{y}%
)\in \mathbb{R}^{2}$, the following properties hold:

\begin{itemize}
\item  (twist condition) for a constant $k_{tw}>0$ the following inequalities hold (the
condition on the inverse follows from our assumption of area preservation): 
\[
\dfrac{\partial p_{1}\circ \widetilde{f}(\widetilde{z})}{\partial \widetilde{%
y}}>k_{tw}\quad \text{and}\quad \dfrac{\partial p_{1}\circ \widetilde{f}%
^{-1}(\widetilde{z})}{\partial \widetilde{y}}<-k_{tw}
\]

\item  there exists a constant $b>0$ such that 
\[
\left| \dfrac{\partial p_{1}\circ \widetilde{f}(\widetilde{z})}{\partial 
\widetilde{x}}\right| <b\quad \text{and}\quad \left| \dfrac{\partial
p_{1}\circ \widetilde{f}^{-1}(\widetilde{z})}{\partial \widetilde{x}}\right|
<b
\]
\end{itemize}

Our claim now is the following: For $K:=b/k_{tw},$ if a continuous function $\widehat{\phi }:\mathbf{T}%
^{1}\to \mathbb{R}$ is not $K$-Lipschitz, then either $\widehat{f}(\mathrm{Graph}(%
\widehat{\phi }))$ or $\widehat{f}^{-1}(\mathrm{Graph}(\widehat{\phi }))$
fails to be a graph.

Let $\widetilde{\phi }:\mathbb{R}\to \mathbb{R}$ be a lift of $\widehat{\phi 
}$, and let $\widetilde{f}$ be a lift of $\widehat{f}$. Suppose that there
exist $a>0$ and $\widetilde{x}\in \mathbb{R}$ such that $|\widetilde{\phi }(%
\widetilde{x}-a)-\widetilde{\phi }(\widetilde{x})|>Ka$. Hence, $\widetilde{%
\phi }$ is not $K$-Lipschitz. From our assumption, there are two
possibilities: either $\widetilde{\phi }(\widetilde{x}-a)>\widetilde{\phi }(%
\widetilde{x})+Ka$ or $\widetilde{\phi }(\widetilde{x})>\widetilde{\phi }(%
\widetilde{x}-a)+Ka$. The first case is dealt with as follows:

\[
p_{1}\circ \widetilde{f}(\widetilde{x}-a,\widetilde{\phi }(\widetilde{x}%
-a))-p_{1}\circ \widetilde{f}(\widetilde{x},\widetilde{\phi }(\widetilde{x}%
))=
\]
\[
=p_{1}\circ \widetilde{f}(\widetilde{x}-a,\widetilde{\phi }(\widetilde{x}%
-a))-p_{1}\circ \widetilde{f}(\widetilde{x}-a,\widetilde{\phi }(%
\widetilde{x}))+p_{1}\circ \widetilde{f}(\widetilde{x}-a,\widetilde{\phi }(%
\widetilde{x}))-p_{1}\circ \widetilde{f}(\widetilde{x},\widetilde{\phi }(\widetilde{x}%
))>
\]

$$
>-b.a+k_{tw}.K.a=0,
$$
by the Mean Value Inequality.

Therefore, $\widetilde{f}(\mathrm{Graph}(\widetilde{\phi }))$ fails to be a
graph. Consequently, $\widehat{f}(\mathrm{Graph}(\widehat{\phi }))$ is not a
graph.

To treat the second case, we consider $\widetilde{f}^{-1}.$ Recall that we
are assuming that for some $a>0$ and $\widetilde{x}\in \mathbb{R},$ $%
\widetilde{\phi }(\widetilde{x})>\widetilde{\phi }(\widetilde{x}-a)+Ka$.
As before, 
\[
p_{1}\circ \widetilde{f}^{-1}(\widetilde{x},\widetilde{\phi }(\widetilde{x}%
))-p_{1}\circ \widetilde{f}^{-1}(\widetilde{x}-a,\widetilde{\phi }(%
\widetilde{x}-a))<-k_{tw}.K.a+b.a=0. 
\]
Thus, $\widetilde{f}^{-1}(\mathrm{Graph}(\widetilde{\phi }))$ is not a
graph, and therefore, $\widehat{f}^{-1}(\mathrm{Graph}(\widehat{\phi }))$ is
also not a graph.

And finally, as for all $r\in \mathbb{R},$ both $\widehat{f}(\mathrm{Graph}(%
\widehat{\phi }_{r}))=\mathrm{Graph}(\widehat{\phi }_{r+\alpha })$ and $%
\widehat{f}^{-1}(\mathrm{Graph}(\widehat{\phi }_{r}))=\mathrm{Graph}(%
\widehat{\phi }_{r-\alpha }),$ we obtain that the family $\{\widehat{\phi }%
_{r}\}$ is $K$-Lipschitz. Since each $\widehat{\phi }_{r}$ induces a
function $\phi _{r\text{ mod 1}}:\mathbf{T}^{1}\to \mathbf{T}^{1}$ under the
canonical projection, it follows that every circloid $C_{r\text{ mod 1}%
}\subset \mathbf{T}^{2}$ is the graph of a $K$-Lipschitz function.

\vskip0.2truecm
\textit{Part 3. }Proof of uniform continuity.
\vskip0.1truecm

For $t\in \mathbf{T}^{1},$ we have proved that the circloids $C_{t}$
actually coincide with $\mathrm{Graph}(\phi _{t}),$ where each $\phi _{t}:\mathbf{%
T}^{1}\to \mathbf{T}^{1}$ is $K$-Lipschitz. We aim to
show that the map $t\mapsto \phi _{t}\in C^{0}(\mathbf{T}^{1},\mathbf{T}^{1})
$ is continuous in the uniform topology. That is, for every $\varepsilon >0$%
, there exists $\delta >0$ such that whenever $d_{\mathbf{T}%
^{1}}(t,s)<\delta $, one has 
\[
\sup_{x\in \mathbf{T}^{1}}d_{\mathbf{T}^{1}}(\phi _{t}(x),\phi _{s}(x))<\varepsilon .
\]

Fix $t\in \mathbf{T}^{1}$ and let $r\in \mathbb{R}$ be such that $r \text{ mod 1}=t$.
Consider sequences $s_{n}^{-}\uparrow r$ and $s_{n}^{+}\downarrow r$ monotonically 
converging to $r,$ the first sequence, converging from below and
the second, converging from above. As for every $x\in \mathbf{T}^{1},$ 
\[
\widehat{\phi} _{s_{n}^{-}}(x)<\widehat{\phi}_{s_{n+1}^{-}}(x)<...<\widehat{\phi}
_{r}(x)<...<\widehat{\phi}_{s_{n+1}^{+}}(x)<\widehat{\phi}_{s_{n}^{+}}(x)
\]
and all the above functions are $K$-Lipschitz, we obtain, by the
Arzel\'a--Ascoli theorem and monotonicity, that $\widehat{\phi}_{s_{n}^{-}}$ converges uniformly to some $K$%
-Lipschitz function $\widehat{\phi}_{r^{-}}\leq \widehat{\phi}_{r}$ and $\widehat{\phi}_{s_{n}^{+}}$
converges uniformly to some $K$-Lipschitz function $\widehat{\phi}_{r^{+}}\geq \widehat{\phi}
_{r}.$ Clearly, the set 
\[
\{(x,y)\in \mathbf{T}^{1}\times \mathbb{R}:\widehat\phi _{r^{-}}(x)\leq y\leq \widehat{\phi} _{r^{+}}(x)\}
\]
is contained in $\widehat{h}^{-1}(r)\supset \mathrm{Graph}(\widehat{\phi}_{r}).$ Since $\widehat{h}^{-1}(r)$
has zero Lebesgue measure, we get that 
$\widehat{\phi}_{r^{-}}=\widehat{\phi}_{r}=\widehat{\phi}_{r^{+}}.$ So, for any sequence $s_{n}\rightarrow r,$
the corresponding functions $\widehat{\phi}_{s_{n}}$ converge uniformly to $\widehat{\phi}_{r},$ 
that is, the uniform continuity is proved in $\mathbf{T}^{1}\times \mathbb{R}.$ Projecting under
$\tau$ proves the statement on the torus.

\vskip0.2truecm
\textit{Part 4. }Existence of the foliation.
\vskip0.1truecm

Let $E:=\bigcup_{t\in\mathbf{T}^{1}}\mathrm{Graph}(\phi_t)$.
First, we show that $E$ is closed.

Let $z_{n}=(x_{n},\phi _{t_{n}}(x_{n}))\in E$ be a sequence converging to $%
z\in \mathbf{T}^{2}$. Passing to a subsequence, we may assume that $t_{n}\to
t\in \mathbf{T}^{1}$ and $x_{n}\to x\in \mathbf{T}^{1}$. From what we did in
Part 3, $\phi _{t_{n}}\to \phi _{t}$ uniformly. Therefore, $(x_{n},\phi
_{t_{n}}(x_{n}))\to (x,\phi _{t}(x))$, and hence $z=(x,\phi _{t}(x))\in 
\mathrm{Graph}(\phi _{t})\subset E$. Thus, $E$ is closed.

We claim that

\begin{equation}
\mathbf{T}^{2}=E.  \label{toroetudo}
\end{equation}
Suppose, for a contradiction, that this equality fails. Then there exists an open ball $B$ which
avoids the $f$-invariant set 
$E.$
The union of all iterates of $B$ under $f$ is an open invariant subset of 
$\mathbf{T}^{2},$ denoted $B_{sat}.$ As $f$ preserves Lebesgue measure, every
connected component of $B_{sat}$ is $f$-periodic. Moreover, the next proposition gives the 
desired contradiction: $B_{sat}$ needs to be fully essential, and at the same time, 
it avoids $E.$ 
And this is not possible.

\begin{proposition}
\label{transitivao} Under our hypotheses, every open $f$-invariant subset $M$
of $\mathbf{T}^{2}$ is fully essential and dense. In particular, $f$ is
transitive.
\end{proposition}

\begin{proof}

First suppose that some  connected component $M^{\prime }$ of $M$ is
homotopically trivial. Then $Filled(M^{\prime })$ is an $f$-periodic open disk, 
and as $f$ preserves area, a well-known consequence of Brouwer's Lemma (see \cite{brown1984new}) 
implies that $f$ has periodic points, a contradiction with our assumption on the
vertical rotation set.

Also, no connected component of $M$ can be essential, but not fully
essential. As $f$ is homotopic to a Dehn twist, if $M$ has an essential
connected component $M_{ess}$, which is not fully essential, 
then it must be horizontal, that is, 
$M_{ess}$ contains essential curves homotopic to $(1,0).$ Finally, as some 
$f^{q}(M_{ess})=M_{ess},$ lifting it to $\mathbf{T}^{1}\times \mathbb{R},$ we
would obtain $\widehat{f}^{q}(\widehat{M}_{ess})=$ $\widehat{M}_{ess}+(0,p),$
for some integer $p,$ where $\widehat{M}_{ess}$ is any connected component
of $\tau ^{-1}({M}_{ess}).$ Our assumption that $M_{ess}$ is not fully essential
implies that $\widehat{M}_{ess}$ is bounded. And this gives a rational number in the
vertical rotation set, something that contradicts our hypotheses.

So, $M$ needs to be fully essential. Now, consider the components of the
complement of the closure of $M.$ If any such component exists, it must be an $f$-periodic inessential
open set. But we already ruled out the existence of such sets. Therefore, $M$ is
dense in the torus. Finally, building $M$ as the union of all iterates of
any given open ball, as $M$ needs to be dense, we obtain transitivity for $f.$
\end{proof}

Therefore, the equality in expression (\ref{toroetudo}) is true.
Since the family $%
\{\mathrm{Graph}(\phi _{t})\}_{t\in \mathbf{S}^{1}}$ consists of pairwise
disjoint sets, it follows that for every point $z\in \mathbf{T}^{2}$ there
exists a unique parameter $t\in \mathbf{T}^{1}$ such that $z\in \mathrm{Graph%
}(\phi _{t})$.

\vskip0.2truecm
\textit{Part 5. }Uniqueness of the foliation.
\vskip0.1truecm

Now, consider the lifted foliation in $\mathbf{T}^{1}\times $ $\mathbb{R},$
given by $r\rightarrow \mathrm{Graph}(\widehat{\phi }_{r})$ for $r\in %
\mathbb{R}.$ Define the function $\widehat{m}:\mathbf{T}^{1}\times \mathbb{R}%
\rightarrow \mathbb{R}$ as

\[
\widehat{m}(\widehat{z}):=\int_{\mathbf{T}^{1}}\widehat{\phi }_{r(\widehat{z}%
)}(x)dx,\text{ where }\mathrm{Graph}(\widehat{\phi }_{r(\widehat{z})})\text{
is the leaf containing }\widehat{z}.
\]
Clearly, $\widehat{m}$ is continuous. For every $a\in \mathbb{R}$, the set $\widehat{m}^{-1}(a)$
is a leaf of the foliation. Moreover, $\widehat{m}$ satisfies $\widehat{m}(\widehat{z}+(0,1))=%
\widehat{m}(\widehat{z})+1.$ Thus, it induces a continuous map
$m:\mathbf{T}^{2}\rightarrow\mathbf{T}^{1}$, defined by
$m(\tau(\widehat{z}))=\widehat{m}(\widehat{z})\text{ mod }1$.
For $s<r,$ the Lebesgue measure of the annulus $A_{s,r}$
between $\mathrm{Graph}(\widehat{\phi }_{s})$ and $\mathrm{Graph}(\widehat{%
\phi }_{r})$ is equal to $\widehat{m}(\widehat{z}_{r})-\widehat{m}(\widehat{z%
}_{s}),$ where $\widehat{z}_{r}$ and $\widehat{z}_{s}$ are (any) points 
in $\mathrm{Graph}(\widehat{\phi }_{r})$ and $\mathrm{Graph}(\widehat{\phi }_{s}),$ 
respectively. As $\widehat{f}$ preserves Lebesgue measure,  
\[
\widehat{m}(\widehat{z}_{r})-\widehat{m}(\widehat{z}_{s})=\widehat{m}(%
\widehat{f}(\widehat{z}_{r}))-\widehat{m}(\widehat{f}(\widehat{z}_{s})).%
\text{ }
\]
So there exists a number $l\in \mathbb{R}$ such that for all $r\in \mathbb{R}$ 
and any $\widehat{z}_{r}\in \mathrm{Graph}(\widehat{\phi }_{r}),$
\begin{equation}
\widehat{m}(\widehat{f}(\widehat{z}_{r}))=\widehat{m}(\widehat{z}%
_{r})+l\Rightarrow \text{for all integers }n>0,\text{ }\widehat{m}(\widehat{f%
}^{n}(\widehat{z}_{r}))=\widehat{m}(\widehat{z}_{r})+nl\Rightarrow l=\alpha ,
\label{akela}
\end{equation}
the single rotation number in $\rho _{V}(\widehat{f}).$ 

Suppose, for a contradiction, that there exist 2 invariant foliations, $\{\mathrm{%
Graph}(\phi _{t})\}_{t\in \mathbf{T}^{1}}$ and \\ $\{\mathrm{Graph}(\phi
_{t}^{*})\}_{t\in \mathbf{T}^{1}},$ and define the functions $m,\widehat{m}$
and $\widehat{m}^{*},m^{*}$ as above for each of them. Since both lifts
increase by $1$ under the translation $(0,1)$, their difference induces
a well-defined continuous function $\Delta:\mathbf{T}^{2}\to\mathbb{R}$:
\[
\Delta(z):=\widehat{m}^{*}(\widehat{z})-\widehat{m}(\widehat{z}),
\qquad \widehat{z}\in\tau^{-1}(z).
\]
By~\eqref{akela}, $\Delta\circ f=\Delta$. As Proposition~\ref{transitivao}
implies that $f$ is transitive, $\Delta$ is constant. Hence
$\widehat{m}^{*}=\widehat{m}+c$ for some $c\in\mathbb{R}$, so their
level sets coincide up to a constant shift of the values. This gives
the uniqueness of the foliation.
\end{proof}

\subsection{Proof of Corollary~\ref{cor:B}}

From now on, we denote the foliation built in the previous theorem,
$$
t\rightarrow \mathrm{Graph}(\phi _{t}), \text{ for } t\in \mathbf{T}^{1},
$$ 
by $\mathcal{F}%
_{G},$ following Definition (7) in Section 1.1. For 
$x\in \mathbf{T}^{1},$ let $V_{x}:=\{x\}\times \mathbf{T}^{1}.$


\begin{claim}\label{claPCM01}
For any $x\in \mathbf{T}^{1},$
given $L\in \mathcal{F}_{G}$ and $n\in \mathbb{Z},$ 
the intersection $f^{n}(V_{x})\cap L$ consists of exactly one point.

\end{claim}
\begin{proof}[Proof of the claim]
As $L\in \mathcal{F}_{G},$ the subset $f^{-n}(L)$ is also a leaf of $\mathcal{F}_{G}.$ 
So it coincides with the graph of a function $\phi _{t}$ for some 
$t\in \mathbf{T}^{1}.$ Therefore, $f^{-n}(L)\cap V_{x}$ is a single
point, and the same holds for $L\cap f^{n}(V_{x})$.
\end{proof}

Let $\widetilde{\mathcal{F}}_{G}$ denote the lift of the foliation $\mathcal{F}%
_{G}$ to $\mathbb{R}^{2}$, and fix some $\widetilde{f}\in \mathrm{Diff}_{\mathrm{tw}%
}^{1}(\mathbb{R}^{2}),$ a lift of $f\in \mathrm{Diff}_{\mathrm{tw}}^{1}(%
\mathbf{T}^{2})$ used to compute the vertical rotation interval. Denote the vertical
at some $\widetilde{x}\in \mathbb{R}$ as $V_{\widetilde{x}}:=\{\widetilde{x}\}\times \mathbb{R}.$ 
The following claim provides a geometric proof for Corollary~\ref{cor:B}.

\begin{claim}
Fix $\widetilde{x}_{0}\in \mathbb{R}$. Then for every $\widetilde{x}\in %
\mathbb{R}$ and every $n\in \mathbb{N}$, the intersection $\widetilde{f}%
^{n}(V_{\widetilde{x}_{0}})\cap V_{\widetilde{x}}$ consists of a single
point.
\end{claim}
\begin{proof}[Proof of the claim]
For some $n\in \mathbb{N}$ and $\widetilde{x}\in \mathbb{R}$, if the
intersection $\widetilde{f}^{n}(V_{\widetilde{x}_{0}})\cap V_{\widetilde{x}}$
contains more than one point, then as the curve $\mathbb{R}\ni t\rightarrow \widetilde{f}%
^{n}(\widetilde{x}_{0},t)$ is a negative curve,
there exists a closed interval $[t_{F},t_{L}]\subset \mathbb{R}$ 
such that for $t<t_{F},$ $\widetilde{f}^{n}(\widetilde{x}_{0},t)$ 
is to the left of $V_{\widetilde{x}},$ 
and for $t>t_{L},$ $\widetilde{f}^{n}(\widetilde{x}_{0},t)$ is to the right of 
$V_{\widetilde{x}}.$ 
Moreover, $\widetilde{f}^{n}(\widetilde{x}%
_{0},t_{F})$ is the highest point in $\widetilde{f}^{n}(V_{\widetilde{x}%
_{0}})\cap V_{\widetilde{x}}$ and $\widetilde{f}^{n}(\widetilde{x}_{0},t_{L})
$ is the lowest (see Le Calvez \cite{le1991proprietes}). This means that $\widetilde{f}^{n}(%
\widetilde{x}_{0}\times [t_{F},t_{L}])$ is a simple arc connecting 2 points
in $V_{\widetilde{x}}.$ Therefore, there exists a leaf $L$ of $\mathcal{F}%
_{G}$ which intersects this arc, because it can be chosen so that $%
\widetilde{f}^{n}(\widetilde{x}_{0},t_{F})$ is above it and $\widetilde{f}%
^{n}(\widetilde{x}_{0},t_{L})$ is below. But now, we consider $\widetilde{f}%
^{n}(\widetilde{x}_{0}\times ]-\infty ,t_{F}]).$ As $\widetilde{f}^{n}(%
\widetilde{x}_{0},t_{F})$ is above $L$ and $p_{2}\circ \widetilde{f}^{n}(%
\widetilde{x}_{0},t)\rightarrow -\infty $ as $t\rightarrow -\infty ,$ we get
that for some $t^{\prime }<t_{F},$ $\widetilde{f}^{n}(\widetilde{x}%
_{0},t^{\prime })$ also belongs to $L.$ This contradicts the previous
claim and completes the proof.
\end{proof}

The previous claim implies that for all integers $n\neq0$, $f^n$ satisfies the topological
twist condition.

\subsection{Proof of Theorem~\ref{thm:C}}

\emph{Sketch of the proof.} First, we prove that there exists a
homeomorphism $\Phi :\mathbf{T}^{2}\to \mathbf{T}^{2}$, homotopic to the
identity, such that $\Phi (\mathcal{F}_{G})=\mathcal{F}_{C},$ the
horizontal foliation, and defining $\Phi \circ f\circ \Phi ^{-1}:=f^{*},$ it
is given by 
\begin{equation}
f^{*}(x,y)=(x+\ell y+\varphi (x,y)\text{ mod 1},y+\alpha \text{ mod 1}),
\label{deff*}
\end{equation}
where $\alpha $ is the unique vertical rotation number of $f$ and $\ell
\in \mathbb{Z}\setminus \{0\}$ is the same as for $f.$ Moreover, $f^{*}$
satisfies the topological twist condition 
and preserves the probability measure $\mu :=\Phi _{*}\mathrm{Leb}_{\mathbf{T%
}^{2}}.$ 
Next, under the hypotheses of Theorem~\ref{thm:C} and
by the classical theory of distality and proximality, we conclude that $f$
is minimal.

\subsubsection{Homeomorphic foliations}



Initially, let us reparameterize the leaves of $\mathcal{F}_{G}$ using the
vertical $V_{0}:=\{0\}\times \mathbf{T}^{1}.$ Each leaf of $\mathcal{F}_{G}$
is of the form $\mathrm{Graph}(\phi _{r})$ for $r\in \mathbf{T}^{1}.$ Let $%
t=t(r)\in \mathbf{T}^{1}$ be given by $t(r):=\{t\in \mathbf{T}^{1}:t=\phi
_{r}(0)\}.$ In the parameter $t,$ the foliation $\mathcal{F}_{G}$ is given
by the graphs of the family $\phi _{t}^{*}:=\phi _{r(t)},$ where $r(t)$ is
the inverse of $t(r).$ From now on, we consider the family $\phi _{t}^{*},$
which, by construction, satisfies $\phi _{t}^{*}(0)=t.$ With a slight abuse
of notation, we drop the * in $\phi _{t}^{*}$. 

Next, we consider the \emph{projection along the leaves} $P:\mathbf{T}%
^{2}\to V_{0}$ as 
\begin{equation}
P(z):=\{\text{the intersection of the leaf containing }z\text{ with }%
V_{0}\}=(0,t),  \label{equaEIFD08}
\end{equation}
where $\mathrm{Graph}(\phi _{t})$ contains $z.$

In other words, the projection along the leaves is obtained by sliding the
point $z$ along its leaf until it meets the transversal $V_{0}$. From our
reparameterization, it coincides with the parameter of the leaf containing $z.
$ Since $\mathcal{F}_{G}$ is a uniformly continuous foliation by graphs,
the map $P$ is well defined and continuous. Moreover, by invariance of $%
\mathcal{F}_{G}$ under $f,$ and since $V_{0}$ is transversal to all the
leaves, $f(\mathrm{Graph}(\phi _{t}))=\mathrm{Graph}(\phi _{s}),$ for some $%
s\in \mathbf{T}^{1}.$ This defines a map $\psi :\mathbf{T}^{1}\to \mathbf{T}%
^{1}$ by 
\begin{equation}
\psi (t):=s,\quad \text{when}\ f(0,t)\in \mathrm{Graph}(\phi _{s}).
\label{equaEIFD09}
\end{equation}
The map $\psi $ captures the dynamics induced by $f$ on the space of
leaves of the foliation $\mathcal{F}_{G}.$ Identifying $V_{0}$ with $\mathbf{%
T}^{1}$, we get that $\psi :=P\circ f|_{V_{0}}$; therefore, it is continuous. Since $%
\mathcal{F}_{G}$ is also $f^{-1}$-invariant, the same argument shows that $%
\psi ^{-1}$ exists and it is also continuous. In particular, $\psi $
preserves orientation, and so it is a homeomorphism of the circle to which a
rotation number can be assigned. And the proposition below shows
that it appears after a natural coordinate change that sends $\mathcal{F}_{G}
$ into the horizontal foliation. So, for a suitable lift of $\psi ,$ its
rotation number is $\alpha $ and $\psi $ is conjugate to $R_{\alpha }.$

The main goal of this subsection is to prove the following result.

\begin{proposition}
\label{propCFT01} Let $\widetilde{f}\in \mathrm{Diff}_{\mathrm{tw}}^{1}(%
\mathbb{R}^{2})$ be area-preserving and suppose that $\rho _{V}(\widetilde{f}%
)=\{\alpha \}$, for some $\alpha \notin \mathbb{Q}.$ Let 
$f:\mathbf{T}^{2}\to \mathbf{T}^{2}$ be the
map lifted by $\widetilde{f}$. Then there exists a homeomorphism $\Phi :%
\mathbf{T}^{2}\to \mathbf{T}^{2}$, homotopic to the identity, such that:

\begin{enumerate}
\item[$(1)$]  $\Phi (\mathcal{F}_{G})=\mathcal{F}_{C}$;

\item[$(2)$]  $\Phi \circ f=f^{*}\circ \Phi $, where $f^{*}:\mathbf{T}%
^{2}\to \mathbf{T}^{2}$ is given in 
(\ref{deff*}) and satisfies the 
topological twist condition.
\end{enumerate}
\end{proposition}

\begin{proof}

Let us start by defining the map $\Phi _{1}:\mathbf{T}^{2}\to \mathbf{T}^{2}$
by 
\begin{equation}
\Phi _{1}(x,y):=(x,t),  \label{eqCFT02}
\end{equation}
where $t\in \mathbf{T}^{1}$ is the unique parameter such that $(x,y)\in 
\mathrm{Graph}(\phi _{t})\in \mathcal{F}_{G}$. The parameter $t$ is well
defined and unique since $\mathcal{F}_{G}$ is a foliation by graphs over the
first coordinate. Part 3 of the proof of Lemma \ref{lemTC07} implies that $%
\Phi _{1}$ is continuous and that the inverse of $\Phi _{1},$ given by $\Phi
_{1}^{-1}(x,s):=(x,\phi _{s}(x))$, is also continuous. Hence, $\Phi _{1}$ is
a torus homeomorphism homotopic to the identity.
As $\Phi _{1}(\mathcal{F}_{G})=\mathcal{F}_{C},$ the map $g:=\Phi
_{1}\circ f\circ \Phi _{1}^{-1}$ is homotopic to a Dehn twist and it
preserves the foliation $\mathcal{F}_{C}$.

Let $\widetilde{g}:\mathbb{R}^{2}\to \mathbb{R}^{2}$ be a lift of $g$ to the
plane. Fix $\widetilde{x}_{0}\in \mathbb{R}.$ 
By definition, $\widetilde{\Phi }_{1}(V_{\widetilde{x}_{0}})=V_{\widetilde{x}_{0}}$, 
and thus $\widetilde{\Phi }_{1}^{-1}(V_{\widetilde{x}_{0}})=V_{\widetilde{x}_{0}}$.
Since $\widetilde{f}\in \mathrm{Diff}_{\mathrm{tw}}^{1}(\mathbb{R}^{2})$,
the curve $\widetilde{f}(V_{\widetilde{x}_{0}})$ has a single intersection
point with the vertical $V_{\widetilde{x}},$ for all $\widetilde{x}\in %
\mathbb{R}$. Hence, $\widetilde{\Phi }_{1}(\widetilde{f}(V_{\widetilde{x}%
_{0}}))$ also has a single intersection point with each vertical $V_{%
\widetilde{x}},$ for all $\widetilde{x}\in \mathbb{R}$. As $\widetilde{g}=%
\widetilde{\Phi }_{1}\circ \widetilde{f}\circ \widetilde{\Phi }_{1}^{-1}$,
we conclude that $\widetilde{g}$ satisfies the topological twist condition.


We may write the map $g$ above in coordinates as 
\[
g(x,y)=(g_{1}(x,y),g_{2}(x,y)),\quad (x,y)\in \mathbf{T}^{2}.
\]
Since $g$ preserves the foliation $\mathcal{F}_{C},$ the second coordinate
is constant along each horizontal leaf. Hence, $g_{2}(x,y)=g_{2}(y)$, and we
can consider the circle homeomorphism $g_{2}:\mathbf{T}^{1}\to \mathbf{T}%
^{1}.$ From the fact that $g$ is homotopic to a Dehn twist, $g_{2}$ is
homotopic to the identity. As $\rho _{V}(\widetilde{f})=\{\alpha \}$ for
some chosen lift $\widetilde{f}\in \mathrm{Diff}_{\mathrm{tw}}^{1}(%
\mathbb{R}^{2}),$ the corresponding lift for $g,$ given by 
\[
\widetilde{g}=\widetilde{\Phi }_{1}\circ \widetilde{f}\circ \widetilde{\Phi }%
_{1}^{-1}=(\widetilde{g}_{1}(\widetilde{x},\widetilde{y}),\widetilde{g}_{2}(%
\widetilde{y})),
\]
where $\widetilde{\Phi }_{1}$ is any lift of $\Phi _{1},$ also satisfies $%
\rho _{V}(\widetilde{g})=\{\alpha \}.$ So, the rotation number of $(g_{2},%
\widetilde{g}_{2})$ is $\alpha .$ And moreover, $g_{2}$ is topologically
conjugate to the rigid rotation $R_{\alpha }$.

Otherwise, $g_{2}$ would be a \emph{Denjoy counterexample,} thus
admitting wandering intervals. Let $I=]s,t[\subset \mathbf{S}^{1}$ be a
wandering interval for $g_{2}$, and define the open set $U:=\mathbf{T}%
^{1}\times I\subset \mathbf{T}^{2}.$ It is easy to see that $U$ is a
wandering set for $g$. Since $f$ is area-preserving, $g$ is non-wandering.
So $g_{2}$ does not admit wandering intervals, and consequently, it must be
topologically conjugate to the irrational rotation $R_{\alpha }$.

Now, let us show that $g_{2}=\psi ,$ defined in expression (\ref
{equaEIFD09}).
Pick some $\phi _{t}$ in $\mathcal{F}_{G}.$ 

The map $\Phi _{1}$
maps $\mathrm{Graph}(\phi _{t})$ into $\mathbf{T}^{1}\times \{t\}.$ 
And $f$($%
\mathrm{Graph}(\phi _{t}))=\mathrm{Graph}(\phi _{\psi (t)}).$ So, 

{\small
\[
\mathbf{T}^{1}\times \{g_{2}(t)\}=g(\mathbf{T}^{1}\times \{t\})=g\circ \Phi
_{1}(\mathrm{Graph}(\phi _{t}))=\Phi _{1}\circ f(\mathrm{Graph}(\phi
_{t}))=\Phi _{1}(\mathrm{Graph}(\phi _{\psi (t)}))=\mathbf{T}^{1}\times
\{\psi (t)\},
\]
}
which implies that $g_{2}(t)=$ $\psi (t)$ for all $t\in \mathbf{T}^{1}.$

Now, consider the lifted foliation $\widehat{\mathcal{F}}_{G}$ of $\mathbf{T}%
^{1}\times $ $\mathbb{R},$ given by $r\rightarrow \mathrm{Graph}(\widehat{%
\phi }_{r})$ for $r\in \mathbb{R}.$ Similar to what we did in Part 5 of the proof 
of Lemma \ref{lemTC07}, define the function $\widehat{m}:\mathbb{R}\rightarrow \mathbb{R}$ as 

\begin{equation}
\widehat{m}(r):=\int_{\mathbf{T}^{1}}\widehat{\phi }_{r}(x)dx.
\label{defconjg}
\end{equation}
Clearly, $\widehat{m}$ is strictly increasing, continuous and satisfies $\widehat{m}%
(r+1)=\widehat{m}(r)+1.$ So, it is the lift of a circle homeomorphism $m.$ 

For $s<r,$ the Lebesgue measure of the annulus $A_{s,r}$ between $\mathrm{%
Graph}(\widehat{\phi }_{s})$ and $\mathrm{Graph}(\widehat{\phi }_{r})$ is
equal to $\widehat{m}(r)-\widehat{m}(s).$ As $\widehat{f}$ preserves
Lebesgue measure and maps $A_{s,r}$ into $A_{\widehat{\psi }(s),\widehat{\psi }(r)},
$ we get that 
\[
\widehat{m}(r)-\widehat{m}(s)=\widehat{m}(\widehat{\psi }(r))-\widehat{m}(%
\widehat{\psi }(s)).\text{ }
\]
So there exists a number $l\in \mathbb{R}$ such that for all $r\in $ $%
\mathbb{R},$%
\[
\widehat{m}(\widehat{\psi }(r))=\widehat{m}(r)+l.
\]
From the above, $\widehat{\psi }$ has rotation number $\alpha
,$ so $l$ is equal to $\alpha $ and 
$$
m\circ \psi (r \text{ mod 1})=R_{\alpha
}\circ m(r \text{ mod 1}).
$$ 
Therefore,  $m:\mathbf{T}^{1}\to \mathbf{T}^{1}$ is
the conjugacy between $\psi $ and $R_{\alpha }.$ Define the map $\Phi _{2}:%
\mathbf{T}^{2}\to \mathbf{T}^{2}$ as
\begin{equation}
\Phi _{2}(x,y):=(x,m(y)),\quad (x,y)\in \mathbf{T}^{2}.  \label{novacaraphi2}
\end{equation}

Notice that $\Phi _{2}$ is a homeomorphism homotopic to the identity that
preserves $\mathcal{F}_{C}$. Moreover, $f^{*}:=\Phi _{2}\circ g\circ \Phi
_{2}^{-1}:\mathbf{T}^{2}\to \mathbf{T}^{2}$ is of the form (\ref{deff*}). 
Define the map $\Phi :\mathbf{T}^{2}\to \mathbf{T}^{2}$ by $\Phi :=\Phi
_{2}\circ \Phi _{1}$. This map is homotopic to the identity and satisfies $\Phi (\mathcal{F%
}_{G})=\mathcal{F}_{C}$. Moreover, $\Phi \circ
f(x,y)=f^{*}\circ \Phi (x,y)$. Finally, define the probability measure $\mu$ by 
$\mu (A):=\mathrm{Leb}_{\mathbf{T}^{2}}\circ \Phi ^{-1}(A)$ for every Borel
set $A\subset \mathbf{T}^{2}$. This measure has full support and no atoms. 
By construction, $f^{*}$ preserves $\mu.$
\end{proof}

In the next lemma, which is actually not used  in the proof of Theorem \ref{thm:C}, we show that 
$\psi$, the map induced by $f$ on the space of leaves of $\mathcal{F}_{G},$ is actually 
bi-Lipschitz when the foliation $\mathcal{F}_{G}$ is Lipschitz. 

\begin{lemma}
Under the hypotheses of Proposition~\ref{propCFT01}, suppose in addition
that the foliation $\mathcal{F}_{G}:=\{\mathrm{Graph}(\phi _{t})\}_{t\in 
\mathbf{T}^{1}}$ is Lipschitz, see definition (8) of  subsection 1.1, that
is, there exists a constant $K_{Fol}>0$ such that for $t_{1},t_{2}\in \mathbf{%
T}^{1}$ and any $x,x^{\prime }\in \mathbf{T}^{1},$ we have 
\[
\frac{1}{K_{Fol}}.d_{\mathbf{T}^{1}}(\phi _{t_{1}}(x^{\prime }),\phi
_{t_{2}}(x^{\prime }))\leq d_{\mathbf{T}^{1}}(\phi _{t_{1}}(x),\phi
_{t_{2}}(x))\leq K_{Fol}.d_{\mathbf{T}^{1}}(\phi _{t_{1}}(x^{\prime }),\phi
_{t_{2}}(x^{\prime })).
\]
Then the induced map $\psi :\mathbf{T}^{1}\to \mathbf{T}^{1}$ is 
bi-Lipschitz.
\end{lemma}

\begin{proof}

Let $t_{1},t_{2}\in \mathbf{T}^{1}.$ Using our reparameterization of the leaves and
setting $x^{\prime }=0$ in the preceding inequality,
we obtain

\begin{equation}
\frac{1}{K_{Fol}}.d_{\mathbf{T}^{1}}(t_{1},t_{2})\leq d_{\mathbf{T}%
^{1}}(\phi _{t_{1}}(x),\phi _{t_{2}}(x))\leq K_{Fol}.d_{\mathbf{T}%
^{1}}(t_{1},t_{2}).  \label{condjulian}
\end{equation}

Since $f$ is $C^{1}$ and $\mathbf{T}^{2}$ is compact, there exists $M>0$ such
that 
\[
d_{\mathbf{T}^{2}}\bigl(f(0,t_{1}),f(0,t_{2})\bigr)\leq Md_{\mathbf{T}%
^{1}}(t_{1},t_{2}).
\]
Writing $f(0,t_{i})=(x_{i},\phi _{\psi (t_{i})}(x_{i}))$ for $i=1,2,$ and
using that each $\phi _{t}$ is $K$-Lipschitz, we obtain 
\[
d_{\mathbf{T}^{1}}(\phi _{\psi (t_{1})}(x_{1}),\phi _{\psi
(t_{2})}(x_{1}))\leq M(1+K)d_{\mathbf{T}^{1}}(t_{1},t_{2}).
\]
By the Lipschitz property of the foliation, 
\[
\frac{1}{K_{Fol}}.d_{\mathbf{T}^{1}}(\psi (t_{1}),\psi (t_{2}))\leq d_{%
\mathbf{T}^{1}}(\phi _{\psi (t_{1})}(x_{1}),\phi _{\psi (t_{2})}(x_{1})),
\]
and therefore, 
\[
d_{\mathbf{T}^{1}}(\psi (t_{1}),\psi (t_{2}))\leq K_{Fol}.M(1+K)d_{\mathbf{T}%
^{1}}(t_{1},t_{2}).
\]
Hence, $\psi $ is Lipschitz.

By an analogous argument, one also concludes that $\psi^{-1}$ is Lipschitz.
\end{proof}

\subsubsection{Minimality: a special case}

By Proposition~\ref{propCFT01}, the map $f$ is conjugate to
$f^{*}:=\Phi\circ f\circ\Phi^{-1}$, which has the form
\[
f^{*}(x,y):=\left( x+\ell y+\varphi (x,y)\text{ mod 1},y+\alpha \text{ mod 1}%
\right) ,\quad (x,y)\in \mathbf{T}^{2},
\]
where $\alpha \notin \mathbb{Q}$, $\ell \in \mathbb{Z}\setminus \{0\}$, and $%
\varphi :\mathbf{T}^{2}\to \mathbb{R}$ is the function induced by a $%
\mathbb{Z}^{2}$-periodic map $\widetilde{\varphi }:\mathbb{R}^{2}\to %
\mathbb{R}$.

Since the leaves of the foliation $\mathcal{F}_C$ are horizontal circles, let us 
fix, from now on, an orientation of $\mathbf{T}^1$. This choice allows us
to define open segments along the leaves unambiguously.

By Proposition~\ref{transitivao}, $f$ is topologically transitive, and
so is its conjugate $f^*$. Thus, to prove that $f$ is minimal, it suffices,
by Lemma~\ref{lemDPH01} and conjugacy, to show that $f^*$ is distal. Since $f^*$ preserves the foliation $\mathcal{F}_{C}$, it
is natural to analyze the dynamics inside each leaf of $\mathcal{F}_{C}$.
Clearly, distality can be established for points lying on distinct leaves.
The remaining case, where two points belong to the same leaf, is more subtle
and requires additional assumptions.

\begin{lemma}
\label{lemMSC01} The map $f^*$ is distal for pairs of points lying on distinct
leaves of $\mathcal{F}_{C}$.
\end{lemma}

\begin{proof}
Fix $s,t\in \mathbf{T}^{1}$ with $s\neq t$, and let $w\in \mathbf{T}%
^{1}\times \{s\}$ and $z\in \mathbf{T}^{1}\times \{t\}$. As $f^*$ acts in the 
vertical coordinate as  a rotation by $\alpha$, 
\[
d_{\mathbf{T}^{1}}(p_{2}\circ (f^*)^{n}(w),p_{2}\circ (f^*)^{n}(z))=d_{\mathbf{T}^{1}}(s,t)>0
\]
for all $n\in \mathbb{Z}$. Hence, for any product metric on $\mathbf{T}^{2}$%
, 
\[
\inf_{n\in \mathbb{Z}}d_{\mathbf{T}^{2}}((f^*)^{n}(w),(f^*)^{n}(z))\ge d_{\mathbf{T}%
^{1}}(s,t)>0,
\]
therefore, $f^*$ is distal for pairs of points lying on distinct leaves of $\mathcal{F}_{C}$.
\end{proof}

For completeness, we now prove that the map $f^{*}$ from expression (\ref{deff*}) is
minimal in the particular case when $\varphi (x,y):=\varphi (y)$.

\begin{proposition}
\label{propMSC02} Let $f^{*}:\mathbf{T}^{2}\to \mathbf{T}^{2}$ be given as
in~(\ref{deff*}), and suppose $\varphi (x,y):=\varphi (y)$. Then $f^*$ is
minimal.
\end{proposition}

\begin{proof}
By Lemma~\ref{lemMSC01}, it remains to prove that $f^*$ is distal for pairs of
distinct points lying on the same leaf of $\mathcal{F}_C$. 
Let $w=(x_{1},t)$ and $z=(x_{2},t)$ be distinct points in $\mathbf{T}%
^{1}\times \{t\}$. When $\varphi (x,y)$ does not depend on $x$, we obtain that 
\[
d_{\mathbf{T}^{2}}((f^*)^{n}(w),(f^*)^{n}(z))=d_{\mathbf{T}^{1}}(x_{1},x_{2})\text{
for every }n\in \mathbb{Z}.
\]
Therefore, $f^*$ is distal, and thus, it is minimal.
\end{proof}

\subsubsection{Completion of the Proof of Theorem~\ref{thm:C}}

Before proceeding with the proof, we need the following technical lemmas.

\begin{lemma}
\label{lemPTC01} Let $\Phi :\mathbf{T}^{2}\to \mathbf{T}^{2}$ be a
homeomorphism homotopic to the identity and bi-Lipschitz. Then there exist
constants $0<d<D$ such that 
\[
d.\mathrm{Leb}_{\mathrm{T}^{2}}(A)\leq \mu (A)\leq D.\mathrm{Leb}_{\mathbf{T}%
^{2}}(A),
\]
where $\mu :=\Phi _{*}\mathrm{Leb}_{\mathbf{T}^{2}}$ and $A$ is any Borel
subset of $\mathbf{T}^{2}$.
\end{lemma}

\begin{proof}
Since $\Phi $ is bi-Lipschitz, there exist constants $0<k_{1}<k_{2}$
satisfying 
\[
k_{1}d_{\mathbf{T}^{2}}(z,w)\le d_{\mathbf{T}^{2}}(\Phi (z),\Phi (w))\le
k_{2}d_{\mathbf{T}^{2}}(z,w),
\]
for all $z,w\in \mathbf{T}^{2}$.
Let $E\subset \mathbf{T}^{2}$ be a Borel set and $\varepsilon>0$. Choose a
covering 
$$
E\subset \bigcup_{i=1}^{\infty }B_{r_{i}}(z_{i})
$$ 
such that each $r_{i}<1/(10k_{2})$ and  
$\sum_{i=1}^{\infty }\mathrm{Leb}_{\mathbf{T}^{2}}(B_{r_{i}}(z_{i}))\le 
\mathrm{Leb}_{\mathbf{T}^{2}}(E)+\varepsilon.$

 As $\Phi (B_{r_{i}}(x_{i}))\subset
B_{k_{2}r_{i}}(\Phi (x_{i})),$ we get that  
\[
\mathrm{Leb}_{\mathbf{T}^{2}}(\Phi (E))\le k_{2}^{2}\sum_{i=1}^{\infty }%
\mathrm{Leb}_{\mathbf{T}^{2}}(B_{r_{i}}(z_{i}))\le k_{2}^{2}(\mathrm{Leb}_{%
\mathbf{T}^{2}}(E)+\varepsilon ).
\]
Letting $\varepsilon \to 0$ yields 
\[
\mathrm{Leb}_{\mathbf{T}^{2}}(\Phi (E))\le k_{2}^{2}\mathrm{Leb}_{\mathbf{T}%
^{2}}(E).
\]
Now let $\mu :=\Phi _{*}\mathrm{Leb}_{\mathbf{T}^{2}}$. Applying the
previous inequality to $E=\Phi ^{-1}(A)$ gives 
\[
\dfrac{1}{k_{2}^{2}}\mathrm{Leb}_{\mathbf{T}^{2}}(A)\le \mu (A).
\]
And the same argument above applied to $\Phi ^{-1}$ yields 
\[
\mu (A)\le \dfrac{1}{k_{1}^{2}}\mathrm{Leb}_{\mathbf{T}^{2}}(A).
\]
Therefore, setting $d:=1/k_{2}^{2}$ and $D:=1/k_{1}^{2}$, the lemma is proved.
\end{proof}


\begin{lemma}
\label{lemPTC02} Assume that the foliation $\mathcal{F}_{G}=\{\mathrm{Graph}%
(\phi _{t})\}_{t\in \mathbf{T}^{1}}$ is Lipschitz as in definition (8) of subsection 1.1. 
Then the map $\Phi$ defined in Proposition~\ref{propCFT01} is bi-Lipschitz.
\end{lemma}

\begin{proof}
Since $\Phi:=\Phi_2\circ \Phi_1$, where $\Phi_1$ is given in~\eqref{eqCFT02}
and $\Phi_2$ is given in~\eqref{novacaraphi2}, it suffices to prove
that both maps are bi-Lipschitz. From expression (\ref{novacaraphi2}),  
$\Phi _{2}=\operatorname{id}_{\mathbf{T}^{1}}\times m.$ If
we show that $m$ is bi-Lipschitz, then $\Phi _{2}$ is also bi-Lipschitz with
respect to any product metric on $\mathbf{T}^{2}$. 

In our reparameterization, the assumption that $\mathcal{F}_{G}$ is Lipschitz
means that there exists a constant $K_{Fol}>0$ such that for any $r_{1},r_{2}\in %
\mathbb{R}$ and any $x\in \mathbf{T}^{1},$ we have 
\[
\frac{1}{K_{Fol}}.\left| r_{1}-r_{2}\right| \leq \left| \widehat{\phi }%
_{r_{1}}(x)-\widehat{\phi }_{r_{2}}(x)\right| \leq K_{Fol}.\left|
r_{1}-r_{2}\right| .
\]

By construction, $\widehat{m}:\mathbb{R}\rightarrow \mathbb{R}$, a lift of $m$, is given by
 $\widehat{m}(r)=\int_{\mathbf{T}^{1}}\widehat{\phi }_{r}(x)dx.$ So it satisfies 
\[
\frac{1}{K_{Fol}}.\left| r_{1}-r_{2}\right| \leq \left| \widehat{m}(r_{1})-%
\widehat{m}(r_{2})\right| \leq K_{Fol}.\left| r_{1}-r_{2}\right| ,
\]
which means that both $\widehat{m}$ and $\widehat{m}^{-1}$ are Lipschitz. Clearly,
the same holds for $m$ and $m^{-1}.$

Therefore, it remains to show that $\Phi_1$ is
bi-Lipschitz.

Let $(x,y),(x^{\prime },y^{\prime })\in \mathbf{T}^{2}$, where $y=\phi
_{t}(x)$ and $y^{\prime }=\phi _{t^{\prime }}(x^{\prime })$. Since
expression (\ref{condjulian}) says that
\[
d_{\mathbf{T}^{1}}(\phi _{t}(x),\phi _{t^{\prime }}(x))\ge \frac{1}{K_{Fol}}%
d_{\mathbf{T}^{1}}(t,t^{\prime }),
\]
and each $\phi _{t}$ is $K$-Lipschitz, we obtain 
\[
\frac{1}{K_{Fol}}d_{\mathbf{T}^{1}}(t,t^{\prime })\le d_{\mathbf{T}%
^{1}}(y,y^{\prime })+Kd_{\mathbf{T}^{1}}(x,x^{\prime }).
\]
Hence, 
\[
d_{\mathbf{T}^{2}}(\Phi _{1}(x,y),\Phi _{1}(x^{\prime },y^{\prime })):=d_{%
\mathbf{T}^{1}}(x,x^{\prime })+d_{\mathbf{T}^{1}}(t,t^{\prime })\le \Xi .d_{%
\mathbf{T}^{2}}((x,y),(x^{\prime },y^{\prime })),
\]
where $\Xi :=\max \{1+K.K_{Fol},K_{Fol}\}$.

Similarly, using 
\[
d_{\mathbf{T}^{1}}(\phi _{s}(x),\phi _{s^{\prime }}(x^{\prime}))\le K.d_{\mathbf{T}%
^{1}}(x,x^{\prime})+K_{Fol}.d_{\mathbf{T}^{1}}(s,s^{\prime }),
\]
we obtain 
\[
d_{\mathbf{T}^{2}}(\Phi _{1}^{-1}(x,s),\Phi _{1}^{-1}(x^{\prime },s^{\prime
}))\le \Xi ^{\prime }.d_{\mathbf{T}^{2}}((x,s),(x^{\prime },s^{\prime })),
\]
where $\Xi ^{\prime }:=\max \{1+K,K_{Fol}\}$. Therefore, $\Phi _{1}$ is
bi-Lipschitz.
\end{proof}

As was explained in the beginning of subsection 3.3.2, the proof of 
Theorem~\ref{thm:C} is now reduced to showing that the map $f^{*}$
is distal for pairs of points lying on the same leaf of $\mathcal{F}_{C}$. 
 Thus, Theorem~\ref{thm:C} now follows from Lemma~\ref
{lemPTC02} and the following lemma.

\begin{lemma}
\label{lemPTC03} Let $f^{*}$ be as in~(\ref{deff*}). Assume moreover that the map $%
\Phi $ given by Proposition~\ref{propCFT01} is bi-Lipschitz. Then $f^*$ is
distal for pairs of points lying on the same leaf of $\mathcal{F}_{C}$.
\end{lemma}

\begin{proof}

Let $w_{1},w_{2}\in \mathbf{T}^{1}\times \{t\}$ for some $t\in \mathbf{T}^{1}
$ be different points. Suppose that  
\begin{equation}
\lim_{k\to \infty }d_{\mathbf{T}^{2}}((f^*)^{n_{k}}(w_{1}),(f^*)^{n_{k}}(w_{2}))=0.
\label{eqCFT05}
\end{equation}
If we arrive at a contradiction, then the lemma is proved. 

 Write $w_{1}:=(x_{1},t)$ and $w_{2}:=(x_{2},t)$ with $%
x_{1}<x_{2}$. By~\eqref{eqCFT05}, after passing to a subsequence if
necessary, we can suppose that as $k\to \infty,$ 
\begin{equation}
(f^*)^{n_{k}}(w_{1})\longrightarrow z\quad \text{and}\quad
(f^*)^{n_{k}}(w_{2})\longrightarrow z.  \label{eqCFT06}
\end{equation}

By~\eqref{eqCFT05} and~\eqref{eqCFT06}, either with $\eta=[x_{1},x_{2}]\times \{t\}$ or with
$\eta=[x_{2},x_{1}]\times \{t\},$ again passing to a subsequence if needed, $(f^*)^{n_{k}}(\eta)$ converges to $z$. 
So, assuming $\eta=[x_{1},x_{2}]\times \{t\}$, for every $\varepsilon >0$ there
exists $k_{0}\in \mathbb{N}$ such that 
\[
(f^*)^{n_{k}}([x_{1},x_{2}]\times \{t\})\subset B(z,\varepsilon ),\quad \text{%
for all}\ k\geq k_{0}.
\]
Since $(f^*)^{n_{k}}$ is uniformly continuous on a neighborhood of the compact
segment $[x_{1},x_{2}]\times \{t\}$, for each $k\ge k_{0}$, there exists $%
\delta _{k}>0$ such that $R_{k}:=[x_{1},x_{2}]\times [t,t+\delta _{k}]$
satisfies
\[
(f^*)^{n_{k}}(R_{k})\subset B(z,\varepsilon ).
\]
Moreover, by taking $\delta _{k}$ smaller if necessary, we may assume that 
\[
\mathrm{diam}((f^*)^{n_{k}}([x_{1},x_{2}]\times \{a\}))<\dfrac{11}{10}\mathrm{%
diam}((f^*)^{n_{k}}([x_{1},x_{2}]\times \{t\})),
\]
for all $t\leq a\leq t+\delta _{k}.$ Since $(f^*)^{n_{k}}(\{x_{i}\}\times
[t,t+\delta _{k}])$ is a graph over both coordinate axes for $i=1,2$, we may
further decrease $\delta _{k}$ if necessary, so that the horizontal distance 
$a_{k,i}$ between the vertical segments through $(f^*)^{n_{k}}(x_{i},t)$ and $%
(f^*)^{n_{k}}(x_{i},t+\delta _{k})$ satisfies 
\[
a_{k,i}<\dfrac{1}{10}\mathrm{diam}((f^*)^{n_{k}}([x_{1},x_{2}]\times \{t\})),\quad
i=1,2.
\]
Consequently, 
\begin{equation}
\mathrm{Leb}_{\mathbf{T}^{2}}((f^*)^{n_{k}}(R_{k}))< \dfrac{13}{10}\mathrm{%
diam}((f^*)^{n_{k}}([x_{1},x_{2}]\times \{t\})).\delta _{k}  \label{eqCFT07}
\end{equation}
By Lemma~\ref{lemPTC01}, and since $f^*$ preserves the probability measure $%
\mu :=\mathrm{Leb}_{\mathbf{T}^{2}}\circ \Phi ^{-1}$, there exist constants $%
0<d<D$ such that 
\[
d.\mathrm{diam}([x_{1},x_{2}]\times \{t\})\delta _{k}\leq \mu (R_{k})=\mu
((f^*)^{n_{k}}(R_{k}))<D.\frac{13}{10}\mathrm{diam}((f^*)^{n_{k}}([x_{1},x_{2}]%
\times \{t\}))\delta _{k}.
\]
Therefore, 
\[
\mathrm{diam}((f^*)^{n_{k}}([x_{1},x_{2}]\times \{t\}))>\dfrac{10d}{13D}.\mathrm{%
diam}([x_{1},x_{2}]\times \{t\}).
\]
On the other hand, from our choice of $\eta$, 
the diameter of $(f^*)^{n_{k}}([x_{1},x_{2}]\times \{t\})$ 
converges to zero. This contradicts the preceding uniform lower bound
and proves that
$f^*$ is distal for points lying on the same leaf of $%
\mathcal{F}_{C}.$
\end{proof}

As we explained before, this concludes the proof of Theorem \ref{thm:C}.

\end{document}